\documentclass{article}
\usepackage{amsmath, amsthm, amscd, amsfonts, amssymb, graphicx, color,mathrsfs,mathtools,enumerate}
\usepackage[mathcal]{euscript}
\usepackage{yfonts}
\usepackage{lineno}
\usepackage{authblk}
\usepackage{dsfont}
\usepackage{bm}
\usepackage[bookmarksnumbered, colorlinks, plainpages]{hyperref}
\hypersetup{colorlinks=true,linkcolor=red, anchorcolor=green, citecolor=cyan, urlcolor=red, filecolor=magenta, pdftoolbar=true} 
\modulolinenumbers[5]
\usepackage{xcolor}

\newtheorem{thm}{Theorem}[section]
\newtheorem{lemma}[thm]{Lemma}
\newtheorem{proposition}[thm]{Proposition}
\newtheorem{corollary}[thm]{Corollary}
\newtheorem{remark}[thm]{Remark}
\newtheorem{defn}[thm]{Definition}
\newtheorem{hyp}[thm]{Hypothesis}

\theoremstyle{definition}

\theoremstyle{remark}

\numberwithin{equation}{section}

\newcommand{\N}{\mathbb N}
\newcommand{\R}{\mathbb R}

\newcommand{\norm}[1]{\left\lVert#1\right\rVert}
\newcommand{\abs}[1]{\left\lvert#1\right\rvert}
\newcommand{\ps}[2]{\langle#1,#2\rangle}
\newcommand{\ix}[2]{\langle#1,#2\rangle_X}

\allowdisplaybreaks

\title{\textbf{Second Quantization and Evolution Operators in infinite dimension}}
\author[$\star$]{Davide Addona}
\author[$\dagger$]{Paolo De Fazio}
\date{}

\affil[$\star$]{\small{corresponding author, Department of Mathematical, Physical and Computer Sciences, Universit\`a degli studi di Parma, email-address: davide.addona@unipr.it}}
\affil[$\dagger$]{\small{Departement of Theoretical and Applied Sciences, Universit\`a degli studi dell'Insubria, email-address: paolo.defazio@uninsubria.it}}

\providecommand{\keywords}[1]{{\textit{Keywords}:} #1}
\providecommand{\subjclass}[1]{{\textit{2020 Mathematics Subject Classification}:} #1}

\begin{document}

\maketitle
\begin{abstract}
\noindent
In an infinite dimensional separable Hilbert space $X$, we study compactness properties and the hypercontractivity of the Ornstein-Uhlenbeck evolution operators $P_{s,t}$ in the spaces $L^p(X,\gamma_t)$, $\{\gamma_t\}_{t\in\R}$ being a suitable evolution system of measures for $P_{s,t}$. Moreover, we study the asymptotic behavior of $P_{s,t}$. Our results are produced thanks to a representation formula for $P_{s,t}$ through the second quantization operator. Among the examples, we consider the transition evolution operator associated to a non-autonomous stochastic parabolic PDE.\\\phantom{a}\\
\\
\end{abstract}




%

\section{Introduction}
\label{intro}
Let $(X, \ix{\cdot}{\cdot},\norm{\,\cdot\,}_X)$ be a separable Hilbert space, let $\{U(t,s)\}_{s\leq t}$ be an evolution operator in $X$ and let $\{B(t)\}_{t\in\R}$ be a strongly continuous family of linear bounded operators on $X$. In this paper we consider the evolution operator $\{P_{s,t}\}_{s\leq t}$ defined on the space of bounded and Borel measurable functions $\varphi:X\to \R$ by
\begin{align*}
 P_{t,t}&=I,\qquad  t\in\R, \\
P_{s,t}\varphi(x)&=\int_X\varphi(y)\,\mathcal{N}_{U(t,s)x ,Q(t,s)}(dy),\quad s<t,\  x\in X,
\end{align*}
where $\mathcal{N}_{U(t,s)x ,Q(t,s)}$ is the Gaussian measure in $X$ with mean $U(t,s) x$
and covariance operator
\begin{align*} 
Q(t,s)=\int_s^tU(t,r) Q(r) U(t,r)^\star\,dr,\quad s<t, \qquad Q(r):=B(r)B(r)^\star, \quad r\in\R.
\end{align*}
As shown in \cite{def_2022,ouy_roc_2016}, $P_{s,t}$ extends to a linear bounded operator on suitable $L^p$ spaces with respect to an evolution systems of measure for $P_{s,t}$, namely families of Borel probability measures $\{\gamma_t\}_{t\in\R}$ in $X$ such that 
\begin{equation*}
\int_XP_{s,t}\varphi(x)\gamma_s(dx)=\int_X\varphi(x)\gamma_t(dx), \qquad s<t, \ \varphi\in L^p(X,\gamma_t).
\end{equation*}
Such evolution system of measures $\{\gamma_t\}_{t\in\R}$ in unique and $\gamma_t$ is the Gaussian measure with mean zero 
and covariance operator
\begin{align*} 
Q(t,-\infty)=\int_{-\infty}^tU(t,r) Q(r) U(t,r)^\star\,dr, \qquad t\in\R. 
\end{align*}

Our aim is to study further properties of $\{P_{s,t}\}_{s\leq t}$ through  a representation formula given by a suitable version of the second quantization operator. In finite dimension evolution operators associated ti Kolmogorov equations have already been widely investigated, see \cite{add_2013,add_ang_lor_2017,ang_lor_2014,ang_lor_20142,ang_lor_lun_2013,dap_lun_2007,gei_lun_2008,gei_lun_2009,kun_lor_lun_2010}. Instead, in infinite dimension, only few results are available, see \cite{big_def_2024,cer_lun_2021,def_2022,kna_2011,ouy_roc_2016}. 

In the autonomous case, for every $s\leq t$ we have $P_{s,t}=P_{0,t-s}=:R_{t-s}$, where $\{R(\sigma)\}_{\sigma\geq 0}$ is a semigroup and it is usually settled in Lebesgue spaces with respect to its invariant measure $\mu$, that exists and is unique under suitable assumptions (see e.g. \cite{lun_pal_2020}).
For all $\varphi\in L^p(X,\mu)$ with $p\in[1,\infty)$, $R_t$ is given by
\begin{align*}
R_0&=I,\\
R_t\varphi(x)&=\int_X\varphi(y)\, \mathcal{N}_{e^{tA}x, Q_t}(dy),\ \ t>0,\ x\in X,
\end{align*}
where $A:D(A)\subseteq X\to X$ is the infinitesimal generator of a strongly continuous semigroup $\{e^{tA}\}_{t\geq0}$,
\begin{equation*}
Q_t=\int_0^t e^{sA}Qe^{sA^\star}ds, \qquad t\geq0
\end{equation*}
and $Q\in\mathscr{L}(X)$ is a self-adjoint non-negative operator.
In \cite{cho_gol_1996} the authors prove that, for every $t\geq 0$, 
\begin{equation}\label{rappresentazione_rt}
R_tf=\Gamma[(Q_\infty^{-\frac{1}{2}}e^{tA}Q_\infty^{\frac{1}{2}})^\star]f,\ \ f\in L^p(X,\mu),\ p\in[1,\infty),
\end{equation}
where, for every contraction $T\in\mathscr L(X)$, $\Gamma(T)$ is the so-called \textit{second quantization operator} (associated to $T$) and $\mu$ is the centered non-degenerate Gaussian measure on $X$ with covariance operator $Q_\infty$. Under suitable assumptions on the operator $T$, the authors prove compactness properties, Hilbert-Schmidt properties, smoothing properties in $L^p(X,\mu)$ and $W^{k,p}(X,\mu)$  and through \eqref{rappresentazione_rt} they deduce analogous result for $R_t$. Thanks to \eqref{rappresentazione_rt} in \cite{cho_gol_1996} the authors also prove a hypercontractivity result for $R_t$ that was known for $\Gamma(T)$ (see e.g. \cite{gli_1968,nel_1966,nel_1973,sim_1974,sim_hoe_1972}). In \cite[Section 2]{cho_2001} the results of \cite{cho_gol_1996} were generalized to the degenerate case.

The operator $\Gamma(T)$ was introduced between the 60ies and the 70ies 
(see e.g. \cite{coo_1953,gli_1968,gli_jaf_1970,nel_1966, nel_1973,seg_1970,seg_1971,sim_1973,sim_1974,sim_hoe_1972}). We underline that in these works the operator $\Gamma(T)$ appears without a proper name and it was introduced studying topics in quantum field theory (see e.g. \cite[Definition page 31 \& Example 3 page 32]{sim_1974}). However, the second quantization of the physics is not given by $\Gamma(T)$ as observed in \cite[page 26]{sim_1974}. More recently, $\Gamma(T)$ appears with the name of \textit{Fock space operator} in \cite{jan_1997} and with the name of \textit{second quantization operator} in \cite{cho_gol_1996}. This last choice has become particularly common in the mathematical literature and $\Gamma(T)$ is universally known as the second quantization operator. In last twenty years the second quantization operator has been widely used to produce further results on Ornstein-Uhlenbeck semigroups (see e.g. \cite{app_1995,app_van_2014,app_van_2015,bog_1998,cho_gol_2001, fey_del_1994, gol_van_2003, van_1998, van_2005}) and generalized Mehler semigroups (see e.g. \cite{app_2015, pes_2011}). For a modern approach to quantum stochastic calculus we refer to \cite{par_1992}.

In this paper we provide a representation formula for $P_{s,t}$ which generalize \eqref{rappresentazione_rt} to prove compactness properties, Hilbert-Schmidt properties and hypercontractivity of $P_{s,t}$ and to study the asymptotic behaviour of $P_{s,t}f$, $s<t$ and $f\in L^p(X,\gamma_t)$. Let $\mu,\nu$ be two centered Gaussian measures on $X$ with Cameron-Martin spaces $H_\mu$ and $H_\nu$, respectively. Given $T\in\mathscr{L}(H_\mu,H_\nu)$, we construct an operator $\Gamma_{\nu,\mu}(T):L^p(X,\mu)\rightarrow L^p(X,\nu)$  which generalizes the second quantization operator $\Gamma(T)$.  Differently from \cite{sim_1974}, we consider $T\in\mathscr L(H_\mu,H_\nu)$ and not necessarily belonging to $\mathscr L(X)$ in order to cover also the degenerate case. As shown in Remark \ref{rmk:eq_operatore_sec_quant_choj}, these two approaches are equivalent in the non-degenerate case, but considering $T\in\mathscr L(H_\mu,H_\nu)$ seems to fit better to the degenerate setting since it allows us to develop a simpler theory than \cite[Section 2]{cho_2001}. Moreover, such an approach does not need any structure on the underlying probability spaces $(X,\mu)$ and $(X,\nu)$ as in the study of Gaussian Hilbert spaces.

The paper is organized as follows.

In Section \ref{notations} we fix the notation and we recall basic result on the tensor product of Hilbert spaces and on the Wiener Chaos decomposition. In particular, we show that, for every Gaussian measure $\gamma$ on $X$, there is an isomorphism between the $n$-th Wiener chaos $\mathscr H_n^\gamma$ and the symmetric tensor space $H_\gamma^{\odot n}$.

In Section \ref{second_serie} we study a representation formula for the second quantization operator via a series analogous to \cite[Chapter 4]{jan_1997}. To be more precise, to study the Ornstein-Uhlenbeck evolution operator we need a second quantization operator of the type $\widetilde \Gamma_{\nu,\mu}(T):L^2(X,\mu)\to L^2(X,\nu)$. Hence, we expand the theory developed in \cite[Chapter 4]{jan_1997} and we give a complete treatment exploiting the properties of the generalized Hermite polynomials and avoiding Wick products and Feynman diagrams. More precisely, we consider two centered Gaussian measures $\mu,\nu$ on $X$ with Cameron Martin spaces $H_\mu$ and $H_\nu$, respectively.  Given a contraction $T\in\mathscr L(H_\mu,H_\nu)$, we define the second quantization operator as the linear operator $\widetilde \Gamma_{\nu,\mu}(T):L^2(X,\mu)\to L^2(X,\nu)$ such that
\begin{align}
\label{operatore_tilde_gamma_intro}
\widetilde \Gamma_{\nu,\mu}(T):=\sum_{n=0}^\infty \widetilde \Gamma_{\nu,\mu,n}(T),    
\end{align}
where $\widetilde \Gamma_{\nu,\mu,n}(T)$, $n\in\N$, are suitable operators defined on the $n$-th Wiener Chaos (see Definition \ref{def:op_gamma_tilden}). 

In Section \ref{second_integrale} we introduce an integral formula similar to that in \cite{fey_del_1994}. More precisely,  given two centered Gaussian measures $\mu,\nu$ on $X$ with Cameron-Martin spaces $H_\mu$ and $H_\nu$, respectively, and a contraction $T\in\mathscr L(H_\mu,H_\nu)$,  we set
\begin{align}
\label{gammamunuintro}
(\Gamma_{\nu,\mu}(T)f)(x)
= & \int_Xf\left[(T^\star )_\infty^{\mu,\nu}x+\left((I-T^\star T)^{\frac{1}{2}}\right)_\infty^{\mu}y\right]\,\mu(dy),\ \ \ \  \nu\textrm{-a.e.}\  x\in X,
\end{align}
where $(T^\star )_\infty^{\mu,\nu}$ and $\left((I-T^\star T)^{\frac{1}{2}}\right)_\infty^{\mu}$ are suitable extensions of $T^\star$ and $(I-T^\star T)^{\frac12}$, respectively, to the whole $X$ as measurable linear operators (see Proposition \ref{prop:op_L_infty2}, Lemma \ref{lemma:prop_L_infty3} and Appendix \ref{mis_op_app}) and $f$ is a $\mu$-measurable function with at most exponential growth.  
We prove that $\widetilde \Gamma_{\nu,\mu}(T)=\Gamma_{\nu,\mu}(T)$ for every contraction $T\in\mathscr L(H_\mu,H_\nu)$ and exploiting \eqref{gammamunuintro} we extend the second quantization operator as a bounded linear operator from $L^p(X,\mu)$ into $L^p(X,\nu)$ for every $p\in[1,\infty)$.
We also show that, under additional assumptions on $T$, $\Gamma_{\nu,\mu}(T)$ is hypercontractive
and fulfills some compactness and Hilbert-Schmidt properties. 
We notice that the representation of $\Gamma(T)$ via an integral-type formula was given for the first time in infinite dimension in \cite{fey_del_1994} in the context of locally convex spaces. In \cite{bog_1998} the same formula was given in separable Banach spaces while in \cite{cho_2001,cho_gol_1996} in a Hilbert setting. First results in finite dimension were proven in \cite{doo_1942,hil_1926,sim_hoe_1972}.

In Section \ref{risultati_OU} we prove 
that, for every $s<t$, $U(t,s)_{|\mathcal H_{s}}$ belongs to $\mathscr{L}(\mathcal H_s,\mathcal H_t)$, where for every $t\in\R$, we denote by $\mathcal H_t$ the Cameron-Martin space of the Gaussian measure $\gamma_t$. Moreover, $\norm{U(t,s)}_{\mathscr{L}(\mathcal H_s,\mathcal H_t)}\leq 1$ and for all $p\in[1,\infty)$
we obtain
\begin{equation}\label{formulafinale}
P_{s,t}f=\Gamma_{\gamma_s,\gamma_t}((U(t,s)_{|\mathcal H_s})^\star)f,\ \ f\in L^p(X,\gamma_t), \ s<t.
\end{equation}
Thanks to \eqref{formulafinale} we deduce for $P_{s,t}$ the results of the previous section.
In particular, if we have $\|U(t,s)_{|\mathcal H_s}\|_{\mathscr L(\mathcal H_s,\mathcal H_t)}<1$, then $P_{s,t}$ is hypercontractive, namely for every $s<t$, $p>1$ and $q\leq q_0:= (p-1)\norm{U(t,s)}^{-2}_{\mathscr{L}(\mathcal H_s,\mathcal H_t)}+1$ we have
\begin{equation*}
\norm{P_{s,t} \varphi}_{L^q(X,\, \gamma_s)}\leq \norm{\varphi}_{L^p(X,\, \gamma_t)},\quad \varphi\in L^p(X,\, \gamma_t).
\end{equation*}
A similar result was proven in \cite[Theorem 6]{big_def_2024}. However, we stress that our assumptions are more general, as shown in Theorem \ref{prop_bignamini} and in Example \ref{ex_diag}.
Moreover, differently from \cite[Theorem 6]{big_def_2024}, we are able to prove that our choice of $q_0$ is optimal. For results in the autonomous case see e.g. \cite{ang_big_fer_2023,dap_zab_2002,gro_1975,lun_pal_2020}.

In Corollary \ref{compattezzapst} and in Corollary \ref{Hilbert-Schmidtpst} we provide necessary and sufficient conditions under which $\{P_{s,t}\}_{s\leq t}$ is compactness and of Hilbert-Schmidt type, respectively.
In Proposition \ref{comportamento_asintotico} we show that the asymptotic behavior of $P_{s,t}f$ depends on $\norm{U(t,s)_{|\mathcal H_s}}_{\mathscr L(\mathcal H_s,\mathcal H_t)}$. More precisely, for every $p\in(1,\infty)$ there exists a positive constant $c_p$ such that  
\begin{align*}
\|P_{s,t}f-m_t(f)\|_{L^p(X,\gamma_s)}\leq c_p\|U(t,s)_{|\mathcal H_s}\|_{\mathcal L(\mathcal H_s,\mathcal H_t)}\|f\|_{L^p(X,\gamma_t)}, \qquad  f\in L^p(X,\gamma_t),\ s<t,   
\end{align*}
where, for every $t\in\R$ and $f\in L^p(X,\gamma_t)$, $m_t(f)$ is the mean of $f$ with respect to the measure $\gamma_t$.
In Proposition \ref{prop_bignamini} a sufficient condition for $P_{s,t}f$ to decay exponentially to $m_t(f)$ is given. 

In Section \ref{Examples} we present different examples of $\{P_{s,t}\}_{s\leq t}$ satisfying our assumptions. In the first one, for every $t\in\R$ the operator $A(t)$ is the realization of a second-order elliptic differential operator in $X=L^2(\mathcal{O})$ either with Dirichlet, Neumann or Robin boundary conditions and sufficiently smooth coefficients. Here, $\mathcal{O}$ is a bounded open subset of $\R^d$ with smooth boundary, $\{U(t,s)\}_{s\leq t}$ is the evolution operator associated to $\{A(t)\}_{t\in\R}$ according to  \cite{acq_1988,acq_ter_1987}, and $B(t):=(-A(t))^{-\gamma}$ with a suitable choice of $\gamma\geq 0$. $\{P_{s,t}\}_{s\leq t}$ turns out to be the transition evolution operator associated to the time inhomogeneous Markov process that is the unique mild solution of the non-autonomous stochastic heat equation
\begin{equation}\label{SPDE}
\begin{cases}
dZ(t)=A(t)Z(t)dt+(-A(t))^{-\gamma}dW(t),\\
Z(s)=x\in X,
\end{cases}
\end{equation}
where $\{W(t)\}_{t\in\R}$ is a $X$-cylindrical Wiener process. We refer to \cite{dap_ian_tub_1982,seid_1993,seid_2003,tub_1982,ver_zim_2008} for a study of SPDEs of the type \eqref{SPDE}.
Finally, we show that our abstract results apply to a non-autonomous version of the Ornstein-Uhlenbeck operator in the Malliavin setting and in case of diagonal operators.

\section{Notation and preliminary results}\label{notations}
If $(E, \norm{\, \cdot\, }_E)$ and $(F, \norm{\, \cdot\, }_F)$ are real Banach spaces we denote by $\mathscr{L}(E;F)$ the space of bounded linear operators from $E$ to $F$. If $F=E$ then we write $\mathscr L(E)$, while if $F=\R$ then we write $E^\star$ instead of $\mathscr{L}(E;F)$, respectively. 
By $B_b(E;F)$ and $C_b(E;F)$ we denote the space of bounded Borel functions from $E$ to $F$ and the space of bounded and continuous functions from $E$ to $F$, respectively. We endow them with the sup norm $$\norm{G}_\infty=\sup_{x\in E}\norm{G(x)}_F.$$  
If $F=\R$, then we simply write $B_b(E)$ and $C_b(E)$ instead of $B_b(E;\R)$ and $C_b(E;\R)$, respectively.

Let $(X,\norm{\, \cdot\, }_X,\ps{\cdot}{\cdot}_X)$ $(Y,\norm{\, \cdot\, }_Y,\ps{\cdot}{\cdot}_Y)$ be two real separable Hilbert spaces. We denote by $I_X$ the identity operator on $X$ and given a subspace $V$ of $X$ we denote by $I_V$ the identity operator on $V$. If no confusion occurs, we do not consider the subscript and simply write $I$.

 We say that $Q\in\mathscr{L}(X)$ is non-negative (respectively negative, non-positive, positive) if for every $x\in X\setminus\{0\}$
\[
\langle Qx,x\rangle_X\geq 0\ (<0,\ \geq,\ >0).
\]
We say that $Q\in\mathscr L(X;Y)$ is a trace class operator if
\begin{align}\label{trace_defn}
{\rm Tr}[Q]:=\sum_{n=1}^{\infty}\langle \abs{Q}e_n,e_n\rangle_X<\infty,\quad \abs{Q}:=\sqrt{Q^\star Q},
\end{align}
for some (and hence for all) orthonormal basis $\{e_n:n\in\N\}$ of $X$. We recall that the trace of an operator, defined in \eqref{trace_defn}, is independent of the choice of the orthonormal basis. We denote by $\mathscr{L}_1(X)$ the space of the trace class operators. It is complete with respect to the norm
\[
\norm{Q}_{\mathscr{L}_1(X)}:={\rm Tr}\abs{Q},\ \ \  Q\in\mathscr L_1(X).
\]

We say that $T\in\mathscr L(X;Y)$ is a Hilbert-Schmidt operator if there exists a Hilbert basis $\{e_k:k\in\N\}$ of $ X $ such that
\begin{equation}\label{serieHS}
\sum_{k=1}^{\infty}\norm{Te_k}_{Y}^2<\infty.
\end{equation}
It can be proven easily that the series in \eqref{serieHS} does not depend on the choice of the Hilbert basis of $X$. We denote by $\mathscr{L}_2(X;Y)$ the space of  the Hilbert-Schmidt operators from $ X $ to $Y$, if $Y= X $ we set $\mathscr{L}_2(X;Y)=\mathscr{L}_2( X )$.
$\mathscr{L}_2( X;Y)$ is a Hilbert space if endowed with the norm
\begin{equation}\label{normaHS}
\norm{T}_{\mathscr{L}_2( X;Y)}:=\sqrt{\sum_{k=1}^{\infty}\norm{Te_k}_{Y}^2},
\end{equation} 
for every $T\in\mathscr{L}_2( X;Y)$ and for every Hilbert basis $\{e_k:k\in\N\}$ of $X$.
\eqref{normaHS} comes from the inner product 
\begin{equation}\label{innerHS}
\ps{T}{S}_{\mathscr{L}_2( X;Y)}:=\sum_{k=1}^{\infty}\ps{Te_k}{Se_k}_{Y},
\end{equation} 
where for every $S,T\in \mathscr{L}_2(X;Y)$, the series on the right hand side of \eqref{innerHS} does not depend on the choice of the Hilbert basis $\{e_k:k\in\N\}$ of $X$.

$Q\in\mathscr L(X;Y)$ is a trace class operator if and only if the operator $\sqrt{\abs{Q}}$ is a Hilbert-Schmidt operator and
\[
{\rm Tr}\abs{Q}=\norm{\sqrt{\abs{Q}}}^2_{\mathscr{L}_2( X )}.
\]

Let $R\in\mathscr{L}(X)$ be a self-adjoint operator. We denote by ${\rm Ker}(R)$ the kernel of $R$ and by $({\rm Ker}(R))^{\bot}$ its orthogonal complement in $X$. 
We set by $H_R:=R(X)$ the range of the operator $R$ and we recall that $({\rm ker} R)^{\bot}=\overline{R(X)}$. In order to provide $H_R$ with a Hilbert structure, we recall that the restriction $R_{|_{({\rm ker} R)^{\bot}}}$ is an injective operator, and so
\[
R_{|_{({\rm ker} R)^{\bot}}}:({\rm ker} R)^{\bot}\subseteq X\rightarrow H_R
\] 
is bijective. We call pseudo-inverse of $R$ the linear bounded operator $R^{-1}:H_R\rightarrow X$ where for all $y\in H_R$, $R^{-1}y$ is the unique $x\in ({\rm ker} R)^{\bot}$ such that $Rx=y$, see \cite[Appendix C]{LI-RO1}. We introduce the scalar product 
\begin{equation*}
\ps{x}{y}_{H_R}:=\langle R^{-1}x,R^{-1}y\rangle_X,\quad x,y\in H_R,
\end{equation*}
and its associated norm $\norm{x}_{H_R}:=\|R^{-1}x\|_X$. With this inner product, $H_R$ is a separable Hilbert space and a Borel subset of $X$ (see \cite[Theorem 15.1]{kec_2012}). A possible orthonormal basis of $H_R$ is given by $\{Re_k\}_{k\in\N}$, where $\{e_k: k\in\N\}$ is any orthonormal basis of $({\rm ker} R)^{\bot}$. Denoting by $P$ the orthogonal projection on ${\rm ker} R$, we recall that 
\begin{align*}
RR^{-1} &=I_{H_R},\qquad R^{-1}R=I-P.
\end{align*}
Notice that, for every $x\in H_R$,
\begin{align}
\label{norma_forte_R}
\norm{x}_X=\|RR^{-1}x\|_X\leq \norm{R}_{\mathscr{L}(X)}\|R^{-1}x\|_{X}\leq \norm{R}_{\mathscr{L}(X)}\norm{x}_{H_R},
\end{align}
which means that $H_R\subseteq X$ with embedding constant $c_R\leq\|R\|_{\mathscr L(X)}$.

The following proposition will be useful in the sequel.

\begin{proposition}\label{pseudo} Let $(X, \norm{\,\cdot\,}_{X},\ix{\cdot}{\cdot})$, $(X_1, \norm{\,\cdot\,}_{X_1},\ps{\cdot}{\cdot}_{X_1})$ and $(X_2, \norm{\,\cdot\,}_{X_2},\ps{\cdot}{\cdot}_{X_2})$ be real separable Hilbert spaces and let $L_1\in\mathscr L(X_1,X)$ and $L_2\in\mathscr L(X_2,X)$. The following statements hold.
\begin{enumerate}[\rm (i)]
\item ${\rm Range}(L_1)\subseteq{\rm Range}(L_2)$ if and only if there exists a constant $C>0$ such that 
\begin{equation}\label{cstar}
\norm{L_1^\star x}_{X_1}\leq C\norm{L_2^\star x}_{X_2},\ \ x\in X.
\end{equation}
In this case $\norm{L^{-1}_2L_1}_{\mathscr{L}(X_1,X_2)}\leq C$, where
\begin{equation*}
\norm{L^{-1}_2L_1}_{\mathscr{L}(X_1,X_2)}=\inf\{C>0\ \textrm{such that \eqref{cstar} holds}\}.
\end{equation*}
\item If $\norm{L_1^\star x}_{X_1}=\norm{L_2^\star x}_{X_2}$ for every $x\in X$, then ${\rm Range}(L_1)={\rm Range}(L_2)$ and
$\norm{L^{-1}_1 x}_{X_1}=\norm{L^{-1}_2 x}_{X_2}$ for every $x\in X$.
\end{enumerate}
\end{proposition}

\begin{proof}
See \cite[Appendix B, Proposition B.1]{bog_1998}.
\end{proof}

Let $\mu$ be a Borel probability measure on $X$. We denote by $\widehat{\mu}$ its characteristic function defined by
\[
\widehat{\mu}(x):=\int_X e^{i\ix{x}{y}}\mu(dy), \qquad x\in X.
\]
Let $Q$ be a self-adjoint non-negative trace class operator and let $a\in X$. We denote by $\mathcal{N}_{a,Q}$ or $\mathcal{N}(a,Q)$ the Gaussian measure in $X$ with mean $a$ and covariance operator $Q$. If $a=0$, we say that the Gaussian measure is centered. We recall that
\begin{equation*}
\widehat{\mathcal{N}}_{a,Q}(x):=\int_X e^{i\ix{x}{y}}\mathcal{N}_{a,Q}(dy)=e^{i\ix{a}{x}-\frac{1}{2}\ix{Qx}{x}},\quad x\in X.
\end{equation*}

The following proposition summarizes the most relevant properties of the Cameron-Martin space of a Gaussian measure on a separable real Hilbert space.

\begin{proposition}
\label{teoria}
Let $\gamma=\mathcal{N}_{0,Q_\gamma}$ be a centered Gaussian measure with covariance operator $Q_\gamma$ and let $P_\gamma$ be the orthogonal projection on ${\rm ker}(Q_\gamma)$. Let $\{e_k:k\in\N\}$ be a Hilbert basis of $X$ consisting of eigenvectors of $Q_\gamma$, i.e. $Q_\gamma e_k=\lambda_ke_k$, $\lambda_k\geq 0$ for all $k\in\N$. If for every $x\in X$ we set $x_k:=\ix{x}{e_k}$, then for every $z\in X$ the series 
\begin{equation*}
\sum_{1\leq k\leq n:\, \lambda_k\neq 0}x_kz_k\lambda_k^{-\frac{1}{2}}
\end{equation*}
converges in $L^p(X,\gamma)$ for every $p\in[1,\infty)$. The reproducing kernel of $\gamma$ is the closed subspace of $L^2(X,\gamma)$ given by 
\begin{equation*}
X _\gamma ^\star :=\left\lbrace W^\gamma_{z}: X \rightarrow\R\; :\; W^\gamma_{z}(x)=\sum_{k\in\N:\,\lambda_k\neq 0}x_kz_k\lambda_k^{-\frac{1}{2}}\ \textrm{a.e. }x\in X\right\rbrace.
\end{equation*}
The Cameron-Martin space $H_\gamma$ of $\gamma$ coincides with the subspace $Q_\gamma^{\frac{1}{2}}(X)$ of $X$ as Hilbert space, namely
\begin{equation*}
H_\gamma:=\left\lbrace x\in X \; :\; \sum_{k\in\N:\,\lambda_k\neq 0}x_k^2\lambda_k^{-1}<\infty\right\rbrace
\end{equation*}
and $$\ps{h}{\ell}_{H_\gamma}=\ix{Q_\gamma^{-\frac{1}{2}} h}{Q_\gamma^{-\frac{1}{2}} \ell}=\sum_{k\in\N:\,\lambda_k\neq 0}h_k \ell_k\lambda_k^{-1},\ \ \ h,\ell\in Q_\gamma^{\frac{1}{2}}(X),$$
where $Q_\gamma^{-\frac{1}{2}}$ is the pseudo-inverse of $Q_\gamma^{\frac{1}{2}}$. The spaces $X^\star_\gamma$ and $H_{\gamma}$ are isomorphic by means of the isomorphism $W^\gamma_{z}\leftrightarrow Q_\gamma^{\frac{1}{2}}z$ and $$\|W^\gamma_{z}\|_{L^2(X,\gamma)}=\norm{Q_\gamma^{\frac{1}{2}}z}_{H_\gamma}, \qquad z\in X.
$$ 
The following statements hold.
\begin{enumerate}[\rm (i)]
\item $H_\gamma$ is the intersection of all measurable subspaces $L$ of $X$ such that $\gamma(L)=1$.
\item $\left\{\lambda^{\frac{1}{2}}_ke_k:\ k\in\N\ \mbox{and}\ \lambda_k\neq 0\right\}$ is an Hilbert basis of the Cameron-Martin space $H_\gamma=Q_\gamma^{\frac{1}{2}}(X)$.
\item The Cameron-Martin space $H_\gamma$ has finite dimension if and only if the set $\sigma_p(Q_\gamma)$ of the eigenvalues of $Q_\gamma$ is finite.
\item The Cameron-Martin space $H_\gamma$ is dense in $X$ if and only if $\lambda_k > 0$ for every $k\in\N$, i.e., $Q_\gamma$ is positive or, equivalently, $\gamma$ is a non-degenerate Gaussian measure on $X$.
\end{enumerate}
\end{proposition}

\begin{proposition}
\label{prop:cameron-martin}
Let $\gamma=\mathcal{N}_{0,Q_\gamma}$ be a centered Gaussian measure with covariance operator $Q_\gamma$ and let $P_\gamma$ be the orthogonal projection on ${\rm ker}(Q_\gamma)$.
For every $h=Q^{\frac{1}{2}}_\gamma z\in Q_\gamma^{\frac{1}{2}}(X)$, the Cameron-Martin formula 
$$\mathcal{N}_{h,Q_\gamma}(dy)=\exp\biggl(-\frac{1}{2}\norm{h}^2_{H_{\gamma}}+W^\gamma_{z}(y)\biggr)\mathcal{N}_{0,Q_\gamma}(dy)$$ 
holds true. Moreover $W^\gamma_{z}$ is a Gaussian random variable with law $$\mathcal{N}\left(0,\norm{W^\gamma_{z}}^2_{L^2(X,\gamma)}\right)=\mathcal{N}(0,\norm{h}^2_{H_\gamma}).$$
Moreover, if $z\in H_\gamma$ then there exists $y\in X$ such that $W^\gamma_{z}(x)=\langle y,x\rangle_X $ for every $x\in X$.
\end{proposition}

Here, we state some known results on Gaussian measures useful to perform change of variables. 

\begin{proposition} \label{corDecompositionHilbert}
Let $X$ be a separable Hilbert space endowed with a Gaussian measure $\gamma$, with mean $a$ and covariance operator $Q$. 
\begin{enumerate}[\rm (i)]
\item If $Y$ is a  separable Hilbert space endowed with a Gaussian measure $\mu$ with  mean $b$ and covariance operator $R$, then 
$\gamma\otimes\mu$ is a Gaussian measure on $X\times Y$ with mean $(a, b)$ and covariance operator $Q_{\gamma\otimes \mu}$ defined by 
\[
Q_{\gamma\otimes \mu}(f,g) =(Q(f),R(g)), \quad f\in X, \;g\in Y.
\]
\item If $\mu$ is another Gaussian measure on $X$ with mean $b$ and covariance operator $R$, then the convolution measure 
\index{convolution of measures} $\gamma\ast\mu$, defined as the image measure in $X$ of $\gamma\otimes\mu$ on $X\times X$ under the map $(x,y) \mapsto x+y$, is a Gaussian measure with mean  $a+ b$ and  covariance operator $Q + R$. 
\item If $Y$ is a  separable Hilbert space and $A\in {\mathscr L}(X;Y)$, $b\in Y$, setting $T(x) = Ax + b$ for $x\in X$, the measure $\mu = \gamma \circ T^{-1}$ is a Gaussian measure in $Y$ with mean $Aa + b$ and covariance operator $AQA^\star$. 
\end{enumerate}
\end{proposition}

\begin{defn}\label{def_e_gamma}Let $\gamma=\mathcal{N}_{0,Q_\gamma}$ be a centered Gaussian measure on $X$ and $H_\gamma$ be the Cameron-Martin space of $\gamma$. 
We denote by $\mathcal E_\gamma$ the set of functions 
\begin{equation*}
x\in X\mapsto E^\gamma_h(x):=e^{W^\gamma_h(x)-\frac12\|h\|_{X}^2},\ h\in H_\gamma.
\end{equation*}
\end{defn}

In the following, we will need of the next lemma. We recall that a similar result (with $h$ varying in $X$ instead of in $Q_{\gamma}^{\frac{1}{2}}(X)$) have been obtained in \cite[Proposition  1.2.5] {dap_zab_2002}. 


\begin{lemma}
\label{lem:density_exp_funct}
Let $\gamma=\mathcal{N}_{0,Q_\gamma}$ be a centered Gaussian measure on $X$ and let $H_\gamma$ be its Cameron-Martin space. Then the set of the functions $$\mathcal{D_\gamma}=\left\{e^{W_h^\gamma}:h\in H_\gamma\right\}$$ is linearly dense in $L^2(X,\gamma)$. 
\end{lemma}
\begin{proof}
We recall that $\mathcal{D_\gamma}$ is linearly dense in $L^2(\overline {H_\gamma}, \gamma_{P_\gamma}^{\perp})$, where $P_\gamma$ is the orthogonal projection on ${\rm Ker}(Q_\gamma)$ and $\gamma_{P_\gamma}^\perp:=\gamma \circ (I-P_\gamma)^{-1}$ on $\overline {H_\gamma}$. This follows from \cite[Lemma 1.1.2]{nual_2006} with $(\Omega,\mathcal F,\mathbb P)=(\overline {H_\gamma},\mathcal B(\overline {H_\gamma}),\gamma)$, $H=\overline {H_\gamma}$ and $W(z)=W^\gamma_z$ for every $h\in \overline {H_\gamma}$.
Since $H_\gamma$ is dense in $\overline{H_\gamma}$, it follows that for every $z\in \overline{H_\gamma}$ there exists a sequence $(z_n)_{n\in\N}\subseteq H_\gamma$ which converges to $z$ in $X$ as $n\to\infty$. By the dominated convergence theorem, it follows that the sequence $(e^{W_{z_n}})_{n\in\N}$ converges to $e^{W_{z}}$ in $L^2(X,\gamma)$ as $n\to\infty$. Hence, the statement is true if $Q_\gamma$ is injective.

In the general case, we consider an orthonormal basis $\{e_n:n\in\N\}$ of $X$ consisting of eigenvectors of $Q_\gamma$. 
Let $f\in\mathscr FC_{b}(X)$. Hence, for every $\varepsilon>0$ there exists $g\in {\rm span}\{\mathcal{D_\gamma}\}$ such that $\|f\circ (I-{P_\gamma})-g\|_{L^2(\overline {H_\gamma},\gamma_{P_\gamma}^{\perp})}\leq \varepsilon$. If we denote by $\gamma_{P_\gamma}$ the measure $\gamma\circ {P_\gamma}^{-1}$ on ${P_\gamma}(X)$, then  we get
\begin{align*}
\widehat{\gamma_{P_\gamma}}(v)
= \int_{{P_\gamma}(X)}e^{i\langle x,v\rangle_X}\gamma_{P_\gamma}(dx)
= \int_{X}e^{i\langle {P_\gamma} x,v\rangle_X}\gamma(dx)
= \exp\Big(-\frac12\langle Q_\gamma {P_\gamma}v,{P_\gamma}v\rangle_X\Big)=\widehat\delta_0(v)
\end{align*}
for every $v\in P_\gamma(X)$, which implies that $\gamma_{P_\gamma}=\delta_0$.
Therefore, we get
\begin{align*}
\int_X|f-g|^2d\gamma
= & \int_{\overline{H_\gamma}}\gamma_{P_\gamma}^\perp(dy)\int_{{P_\gamma}(X)}|f({P_\gamma}x+(I-{P_\gamma})y)-g({P_\gamma}x+(I-{P_\gamma})y)|^2\gamma_{P_\gamma}(dx) \\
= & \int_{\overline{H_\gamma}}|f((I-{P_\gamma})y)-g((I-{P_\gamma})y)|^2\gamma_{P_\gamma}^\perp(dy).
\end{align*}
We stress that, for every $z\in \overline{H_\gamma}$,
\begin{align*}
e^{W^\gamma_z((I-{P_\gamma})y)}=e^{W^\gamma_z(y)}, \qquad y\in X,   \end{align*}
Hence, $g((I-{P_\gamma})y)=g(y)$ for every $y\in X$ and
\begin{align*}
\|f-g\|_{L^2(X,\gamma)}
=\|f\circ(I-{P_\gamma})-g\|_{L^2(\overline{H_\gamma},\gamma_{P_\gamma}^\perp)}\leq \varepsilon.
\end{align*}
The thesis follows from the density of $\mathscr FC_b(X)$ into $L^2(X,\gamma)$.
\end{proof}

\subsection{Wiener Chaos Decomposition and Tensor Product Spaces}

\subsubsection{Wiener Chaos Decomposition}
For every $n\in\N\cup\{0\}$, we denote by $\phi_n$ the $n$-th Hermite polynomial, given by
\begin{align}
\label{def_pol_herm_1}
\phi_n(\xi)=\frac{(-1)^n}{n!}e^{\frac{\xi^2}{2}}\frac{d^n}{d\xi^n}e^{-\frac{\xi^2}{2}}, \qquad \xi\in\R, \ n\in\N\cup\{0\}. 
\end{align}
Let $\gamma$ be a centered Gaussian measure on $X$ with covariance operator $Q_\gamma$ and let $\{e^\gamma_n:n\in\N\}\subseteq H_\gamma$ be an orthonormal basis of $H_\gamma$. From Proposition \ref{teoria}(ii), without loss of generality, hereafter we assume that $e_j^\gamma$ is an eigenvector of $Q_\gamma$ for every $j\in\N$. We set $\Lambda\subseteq(\N\cup\{0\})^{\N}$ the set of multi-indices with only a finite number of components different from $0$ and, for every $\alpha\in\Lambda$ and $j\in\N$, $\alpha_j$ denotes the $j$-th component of $\alpha$. For every multiindex $\alpha\in\Lambda$, 
we set
\begin{align}
\label{pol_herm_gen}
\Phi_\alpha^\gamma(x):=\sqrt{\alpha!}\prod_{j=1}^\infty \phi_{\alpha_j}\Bigl(W^\gamma_{f_j}(x)\Bigr), \quad x\in X, \ \alpha=(\alpha_1,\alpha_2,\ldots,)\in\Lambda,  
\end{align}
where $f_j:=Q_\gamma^{-\frac 12}e^\gamma_j$ for every $j\in\N$ and ${ \alpha!}:=\prod_{j=1}^\infty\alpha_j!$. We also set $|\alpha|:=\sum_{j=1}^\infty\alpha_j$. It is well-known that, if we set
\begin{align*}
\mathscr H_n^\gamma:=\overline{{\rm span}\{\Phi_\alpha:\alpha\in\Lambda, \ |\alpha|=n\}}, \qquad n\in\N\cup\{0\}, 
\end{align*}
then $\mathscr H_n^\gamma$ is the $n$-th Wiener Chaos, $\mathscr H_n^\gamma\perp\mathscr H_m^\gamma$ if $n\neq m$ and
\begin{align}
\label{wiener_chaos_dec}
L^2(X,\gamma)=\bigoplus_{n\in\N\cup\{0\}}\mathscr H_n^\gamma.    
\end{align}
Decomposition \eqref{wiener_chaos_dec} is known as {\it Wiener Chaos Decomposition} of $L^2(X,\gamma)$. For every $n\in\N$ we denote by $I_n^\gamma$ the projection from $L^2(X,\gamma)$ onto $\mathscr H_n^\gamma$, that is,
\begin{align*}
f=\sum_{n=0}^\infty I_n^\gamma(f), \qquad I_n^\gamma(f)=\sum_{\alpha\in\Lambda, |\alpha|=n}\Big(\int_Xf\Phi_\alpha^\gamma d\gamma\Big) \Phi_\alpha^\gamma, \qquad f\in L^2(X,\gamma).   
\end{align*}
For every $n\in\N$, we denote by $\Lambda_n$ the subset of $\Lambda$ which consists of those elements $\alpha\in\Lambda$ with $|\alpha|=n$.

\subsubsection{Product tensor spaces}
Let $(V_1,\norm{\,\cdot\,}_{V_1},\ps{\cdot}{\cdot}_{V_1})$ and $(V_2,\norm{\,\cdot\,}_{V_2},\ps{\cdot}{\cdot}_{V_2})$ be two real separable Hilbert spaces. We follow the approach of \cite[Appendix E]{jan_1997} to introduce the tensor product space $V_1\otimes V_2$. 
Given $v_1\in V_1$ and $v_2\in V_2$, let $v_1\otimes v_2:V_1\times V_2\to\R$ be the so called conjugate bilinear form given by $$(v_1\otimes v_2)(z_1,z_2)=\ps{v_1}{z_1}_{V_1}\ps{v_2}{z_2}_{V_2},\ \ (z_1,z_2)\in V_1\times V_2.$$
We denote by $\mathcal{V}$ the set of all real linear combinations of such conjugate bilinear forms and we define an inner product $\left[\cdot,\cdot\right]_{V_1\otimes V_2}$ by $$\left[v_1\otimes v_2,w_1\otimes w_2\right]_{V_1\otimes V_2}=\ps{v_1}{w_1}_{V_1}\ps{v_2}{w_2}_{V_2},\ \ v_1,w_1\in V_1,\ v_2,w_2\in V_2.$$
We extend $[\cdot,\cdot]_{V_1\otimes V_2}$ to  $\mathcal{V}$ by linearity. In \cite[Appendix E, Definition E.7]{jan_1997}, it is shown that $\left[\cdot,\cdot\right]_{V_1\otimes V_2}$ is well-defined and positive definite.
 
\begin{defn} The tensor product of the Hilbert spaces $V_1$ and $V_2$ is the completion of $\mathcal{V}$ under the inner product $[\cdot,\cdot]_{V_1\otimes V_2}$ defined above. The tensor product of $V_1$ and $V_2$ is denoted by $V_1\otimes V_2$.
\end{defn}

\begin{proposition}
\label{prop:hilbert_basis_V1V2}
If $\{v_n:n\in\N\}$ and $\{w_n: n\in\N\}$ are Hilbert bases of $V_1$ and $V_2$, respectively, then $\{v_n\otimes w_m:n,m\in\N\}$ is a Hilbert basis of $V_1\otimes V_2$.
\end{proposition}
\begin{proof}
See \cite[Appendix E, Example E.9]{jan_1997}.
\end{proof}

Let $\{v_n:n\in\N\}$ and $\{w_n:n\in\N\}$ be Hilbert bases of $V_1$ and $V_2$, respectively. From Proposition \ref{prop:hilbert_basis_V1V2}, any element $y$ of $V_1\otimes V_2$ can be written in a unique way as
\begin{align*}
y=\sum_{i,j=1}^\infty c_{ij}v_i\otimes w_j, \qquad c_{ij}\in\R, \  i,j\in\N, \quad \|v\|_{\mathscr L(V_1\otimes V_2)}^2:=\sum_{i,j=1}^{\infty} c_{ij}^2<\infty.   \end{align*}

Let $\widetilde V_i$, $i=1,2$, be other two separable Hilbert spaces and let $T_i\in\mathscr L(V_i,\widetilde V_i)$, $i=1,2$. We define the operator $T_1\otimes T_2:V_1\otimes V_2\to \widetilde V_1\otimes \widetilde V_2$ as
\begin{align*}
(T_1\otimes T_2)y=\sum_{i,j=1}^\infty c_{ij}(T_1v_i)\otimes (T_2w_j), \qquad y\in V_1\otimes V_2.    
\end{align*}
\begin{proposition}
\label{prop:tens_op_1_1}
$T_1\otimes T_2$ is well-defined, it belongs to $\mathscr L(V_1\otimes V_2,\widetilde V_1\otimes \widetilde V_2)$ and satisfies 
\begin{align*}
\|T_1\otimes T_2\|_{\mathscr L(V_1\otimes V_2,\widetilde V_1\otimes \widetilde V_2)}=\|T_1\|_{\mathscr L(V_1,\widetilde V_1)}\|T_2\|_{\mathscr L(V_2,\widetilde V_2)}.
\end{align*}
\end{proposition}
\begin{proof}
See \cite[Prop.  E.20, Appendix E]{jan_1997}.
\end{proof}
In a similar way it is possible to define the tensor product $\displaystyle{\bigotimes_{i=1}^nV_i:=V_1\otimes....\otimes V_n}$ of $n$ real separable Hilbert spaces $V_1,...,V_n$. Given a separable Hilbert space $V$, we write $V^{\otimes n}$ instead of $\displaystyle{\bigotimes_{i=1}^nV}$. From Proposition \ref{prop:tens_op_1_1} we deduce the following result.
\begin{corollary}
\label{coro:op_tensor_spaces}
If $V, \widetilde V$ are real separable Hilbert spaces, $\{v_n:n\in\N\}$ is a Hilbert basis of $V$ and $T\in\mathscr L(V,\widetilde V)$, then for every $n\in\N$ the operator $T^{\otimes n}$, defined as
\begin{align*}
T^{\otimes n}y=\sum_{i_1,\ldots,i_n=1}^\infty c_{i_1\cdots i_n}(Tv_{i_1})\otimes\cdots \otimes (Tv_{i_n}),  \qquad y= \sum_{i_1,\ldots,i_n=1}^\infty c_{i_1\cdots i_n}v_{i_1}\otimes\cdots \otimes  v_{i_n}\in V^{\otimes n},
\end{align*}
belongs to $\mathscr L(V^{\otimes n},\widetilde V^{\otimes n})$ and satisfies $\|T^{\otimes n}\|_{\mathscr L(V^{\otimes n},\widetilde V^{\otimes n})}=\|T\|_{\mathscr L(V,\widetilde V)}^n$.
\end{corollary}
\begin{proof}
We prove the statement by induction.

If $n=2$ then it follows from Proposition \eqref{prop:tens_op_1_1} with $V_1=V_2=V$, $\widetilde V_1=\widetilde V_2=V$ and $T_1=T_2=T$. If it is true for $n-1\in \N$, $n\geq 3$, then it is true also for $n$ by applying Proposition \ref{prop:tens_op_1_1} with $V_1=V$, $V_2=V^{\otimes n-1}$, $\widetilde V_1=\widetilde V$, $\widetilde V_2=\widetilde V^{\otimes n-1}$, $T_1=T$ and $T_2=T^{\otimes n-1}$, and taking the inductive hypothesis into account.
\end{proof}

For every $n\in\N$, we denote by $\Sigma_n$ the symmetric group on $\{1,...,n\}$, namely the set of all bijective functions $\sigma:\{1,...,n\}\to\{1,...,n\}$ whose group operation is the standard composition of functions. We recall that the cardinality of $\Sigma_n$ is $n!$.
\begin{defn}
Let $V$ be a separable real Hilbert space and let $\{v_n:n\in\N\}$ be a Hilbert basis of $V$. For every $n\in\N$, we denote by $V^{\odot n}$ the symmetric tensor product space, namely the closed subspace of $V^{\otimes n}$ consisting of those elements $y\in V^{\otimes n}$ which satisfy
\begin{align}
\label{v_in_vodot_n_1}
y=\sum_{i_1,\ldots,i_n=1}^\infty c_{i_1\cdots i_n}v_{i_1}\otimes \cdots\otimes v_{i_n}, \quad c_{i_{\sigma(1)}i_{\sigma(2)}\cdots i_{\sigma(n)}}=c_{i_1i_2\cdots i_n},\ \ \sigma\in\Sigma_n. 
\end{align}
\end{defn}

\begin{lemma}
Let $V,\widetilde V$ be two real separable Hilbert spaces. If $T\in\mathscr L(V,\widetilde V)$, then $T^{\otimes n}$ maps $V^{\odot n}$ into $\widetilde V^{\odot n}$.
\end{lemma}
\begin{proof}
Let $\{v_n:n\in\N\}$ be a Hilbert basis of $V$. For every $y\in V^{\odot n}$ of the form of \eqref{v_in_vodot_n_1}, it follows that
\begin{align*}
T^{\otimes n}y
= & \sum_{i_1,\ldots,i_n=1}^\infty c_{i_1\cdots i_n}(Tv_{i_1})\otimes \cdots\otimes (Tv_{i_n})
= \sum_{k_1,\ldots,k_n=1}^\infty
\widetilde c_{k_1\cdots k_n}\widetilde v_{k_1}\otimes \cdots\otimes \widetilde v_{k_n},
\end{align*}
where $\{\widetilde v_n:n\in\N\}$ is a Hilbert basis of $\widetilde V$. For every $k_1,\ldots,k_n\in\N$, we get
\begin{align*}
\widetilde c_{k_1\cdots k_n}
=\sum_{i_1,\ldots,i_n=1}^\infty c_{i_1\cdots i_n}\langle Tv_{i_1},\widetilde v_{k_1}\rangle_{\widetilde V}\cdots\langle Tv_{i_n},\widetilde v_{k_n}\rangle_{\widetilde V}.
\end{align*}
If $\sigma\in \Sigma_n$, then
\begin{align*}
\widetilde c_{k_{\sigma(1)}\cdots k_{\sigma(n)}}
= & \sum_{i_1,\ldots,i_n=1}^\infty c_{i_1\cdots i_n}\langle Tv_{i_1},\widetilde v_{k_{\sigma(1)}}\rangle_V\cdots\langle Tv_{i_n},\widetilde v_{k_{\sigma(n)}}\rangle_V \\
= & \sum_{i_1,\ldots,i_n=1}^\infty c_{i_{\sigma(1)}\cdots i_{\sigma(n)}}\langle Tv_{i_{\sigma(1)}},\widetilde v_{k_{\sigma(1)}}\rangle_V\cdots\langle Tv_{i_{\sigma(n)}},\widetilde v_{k_{\sigma(n)}}\rangle_V \\
= & \sum_{i_1,\ldots,i_n=1}^\infty c_{i_1\cdots i_n}\langle Tv_{i_1},\widetilde v_{k_1}\rangle_V\cdots\langle Tv_{i_n},\widetilde v_{k_n}\rangle_V =\widetilde c_{k_1\cdots k_n},
\end{align*}
where we have exploited the fact that $y\in V^{\odot n}$.
\end{proof}

Let $V$ be a real separable Hilbert space and let $\{v_n:n\in\N\}$ be a Hilbert basis of $V$. We now introduce a notation for the elements of $V^{\odot n}$ which will be quite useful to express elements of $V^{\odot n}$ in terms of elements of $\Lambda_n$. In particular, we exploit elements of $\Lambda_n$ to count how many times each element of the Hilbert basis $\{v_n:n\in\N\}$ appears in \eqref{v_in_vodot_n_1}. \\
To be more precise, we associate to the term $v_{i_1}\otimes \cdots\otimes v_{i_n}$ a multi-index $\alpha=(\alpha_j)_{j\in\N}\in\Lambda_n$ such that for all $j\in\N$ the number $\alpha_j$ is the number of the repetitions of the vector $v_j$ in $v_{i_1}\otimes \cdots\otimes v_{i_n}$.  Without loss of generality (due to the fact that we deal with symmetric tensors), we may assume that the indices $i_1,\ldots,i_n$ in $v_{i_1}\otimes \cdots\otimes v_{i_n}$ appear in an increasing way. For every $\alpha\in\Lambda_n$, we define the indices  $i^\alpha_1,i^\alpha_2,\ldots,i^\alpha_n\in\N$ such that, if $j_1,\ldots,j_m\in\N$ and $a_{j_1},\ldots,\alpha_{j_m}$ are the components of $\alpha$ different from $0$, then for every $\ell=1,\ldots,m$,
$$
i^\alpha_{1+\sum_{k=1}^{\ell-1}\alpha_{j_k}}=i^\alpha_{2+\sum_{k=1}^{\ell-1}\alpha_{j_k}}=\ldots=i^\alpha_{\alpha_{j_\ell}+\sum_{k=1}^{\ell-1}\alpha_{j_k}}=j_\ell,
$$ 
where we used the standard notation $\sum_{i=1}^{0}\alpha_i=0$. Clearly $i^\alpha_k\leq i^\alpha_{k+1}$ for every $k=1,\ldots,n-1$ and
\begin{align*}
i^\alpha_1=\ldots=i^\alpha_{\alpha_{j_1}}=j_1, \quad
i^\alpha_{\alpha_{j_1}+1}=\ldots=i^\alpha_{\alpha_{j_1}+\alpha_{j_2}}=j_2, \ \ldots, \ i^\alpha_{\alpha_{j_1}+\ldots+\alpha_{j_{m-1}}+1}=\ldots=i^\alpha_{\alpha_{j_1}+\ldots+\alpha_{j_m}}=j_m.
\end{align*}

We notice that every $y\in V^{\odot n}$ can be written as 
\begin{align*}
y=\sum_{\alpha\in\Lambda_n}\frac{c_{\alpha}}{\alpha!}\sum_{\sigma\in\Sigma_n}v_{i^\alpha_{\sigma(1)}}\otimes \cdots \otimes v_{i^\alpha_{\sigma(n)}},    
\end{align*}
where $c_\alpha=c_{i^\alpha_1\cdots i^\alpha_n}$ and the factor $\frac{c_\alpha}{\alpha!}$ appears since, when we consider any permutation of $i^\alpha_1,\ldots,i^\alpha_n$, due to repetitions we are taking $\alpha!$ times each element of the form $v_{i^\alpha_{\sigma(1)}}\otimes \cdots \otimes v_{i^\alpha_{\sigma(n)}}$. To avoid such repetitions, we fix $\alpha\in\Lambda_n$ and we define  on $\Sigma_n$ the relation $\sim_\alpha$ as follows: given $\sigma_1,\sigma_2\in\Sigma_n$, we say that $\sigma_1\sim_\alpha\sigma_2$ if $i^\alpha_{\sigma_1(j)}=i^\alpha_{\sigma_2(j)}$ for every $j=1,\ldots,n$. It follows that $\sim_\alpha$ is an equivalence relation on $\Sigma_n$ and we set  
\begin{align}\label{definizionesigmaalphan}
\Sigma_n^\alpha:=\Sigma_n \ \!/ \sim_\alpha, \qquad \#\Sigma_n^\alpha=\frac{n!}{\alpha!}.
\end{align}
For every $\alpha\in \Lambda_n$ we have
\begin{align*}
\sum_{\sigma\in\Sigma_n}v_{i^\alpha_{\sigma(1)}}\otimes \cdots \otimes v_{i^\alpha_{\sigma(n)}}=\alpha!\sum_{\sigma\in\Sigma_n^\alpha}v_{i^\alpha_{\sigma(1)}}\otimes \cdots \otimes v_{i^\alpha_{\sigma(n)}}.    
\end{align*}
Hence, every $y\in V^{\odot n}$ can be written as
\begin{align}
\label{espressione_elemento_V_symm_n}
y=\sum_{\alpha\in\Lambda_n}
c_\alpha\sum_{\sigma\in\Sigma_n^\alpha}v_{i^\alpha_{\sigma(1)}}\otimes \cdots \otimes v_{i^\alpha_{\sigma(n)}}.    
\end{align}
Since for every $\sigma_1,\sigma_2\in\Sigma_n^\alpha$, with $\sigma_1\neq \sigma_2$, the corresponding elements $v_{i^\alpha_{\sigma_1(1)}}\otimes \cdots \otimes v_{i^\alpha_{\sigma_1(n)}}$ and $v_{i^\alpha_{\sigma_2(1)}}\otimes \cdots \otimes v_{i^\alpha_{\sigma_2(n)}}$ are orthogonal, from \eqref{definizionesigmaalphan} we deduce that
\begin{align}
\label{norma_elemento_v_symm_n}
\|v\|_{V^{\otimes n}}^2
= \sum_{\alpha\in\Lambda_n}c_\alpha^2\cdot (\#\Sigma_n^\alpha)
=\sum_{\alpha\in\Lambda_n}c_\alpha^2\frac{n!}{\alpha!}, \qquad v\in V^{\odot n}.
\end{align}

\subsubsection{Wiener Chaos Decomposition and Tensor Product Spaces}

We conclude this section by showing an explicit isometry between $H_{\gamma}^{\odot n}$ and $\mathscr H_n^\gamma$. In particular, we generalize the result stated at the end of \cite[Subsection 1.1.1]{nual_2006} for the non-degenerate case.

\begin{lemma}
\label{lemma:isomorfismo}
Let $\gamma$ be a centered Gaussian measure on $X$, let $H_\gamma$ be its Cameron-Martin space and let $\{v^\gamma_n:n\in\N\}$ be an orthonormal basis of $H_\gamma$. There is a linear isomorphism $\Psi_n^\gamma$ between $(H_\gamma^{\odot n}, \sqrt{n!}\|\cdot\|_{H_\gamma^{\otimes n}}, n!\ps{\cdot}{\cdot}_{H_\gamma^{\otimes n}})$ and $(\mathscr H_n^\gamma,\|\cdot \|_{L^2(X,\gamma)},\ps{\cdot}{\cdot}_{L^2(X,\gamma)})$, given by 
\begin{align}
\label{iso_wiener_chaos_tensor_prod}
\Psi_n^\gamma\left(\frac{\alpha!}{n!}\sum_{\sigma \in \Sigma_n^\alpha}v_{i^\alpha_{\sigma(1)}}^\gamma\otimes \cdots \otimes v_{i^\alpha_{\sigma(n)}}^\gamma\right)=\Psi_n^\gamma\left(\frac{1}{n!}\sum_{\sigma \in \Sigma_n}v_{i^\alpha_{\sigma(1)}}^\gamma\otimes \cdots \otimes v_{i^\alpha_{\sigma(n)}}^\gamma\right)=\sqrt{\alpha!}\Phi_\alpha^\gamma
\end{align}
for every $\alpha\in\Lambda_n$, where $\Phi_\alpha^\gamma$ is the generalized Hermite polynomial defined in \eqref{pol_herm_gen} with $e^\mu_j$ replaced by $v_j^\mu$, $j\in\N$. 
\end{lemma}
\begin{proof}
The mapping $\Psi_n^\gamma$ is clearly well-defined, linear and surjective. We prove now that $\Psi_n^\gamma$ is an isometry. For every $\alpha\in \Lambda_n$, from \eqref{norma_elemento_v_symm_n} we get
\begin{align*}
\left\|\frac1{n!}\sum_{\sigma \in \Sigma_n}ve^\gamma_{i^\alpha_{\sigma(1)}}\otimes \cdots \otimes v^\gamma_{i^\alpha_{\sigma(n)}}\right\|^2_{H_\gamma^{\otimes n}}
= & \left\|\frac{\alpha!}{n!}\sum_{\sigma \in \Sigma_n^\alpha}v_{i^\alpha_{\sigma(1)}}^\gamma\otimes \cdots \otimes v_{i^\alpha_{\sigma(n)}}^\gamma\right\|_{H_\gamma^{\otimes n}}^2
= \frac{(\alpha!)^2}{(n!)^2} \cdot (\#\Sigma_{n}^\alpha)
= \frac{\alpha!}{n!} \\ 
= & \frac{\|\sqrt{\alpha!}\Phi_\alpha^\gamma\|_{L^2(X,\gamma)}^2}{n!},
\end{align*}
which gives the thesis.
\end{proof}
 

\section{The second quantization operator as a series}\label{second_serie}
Let $\mu=\mathcal{N}(0,Q_\mu)$ and $\nu=\mathcal{N}(0,Q_\nu)$ be two centered Gaussian measures on $X$ with Cameron-Martin spaces $H_\mu$ and $H_\nu$, respectively. 
Our aim is to study in a deeper way a generalized version of the second quantization of \cite{sim_1974} as operator from $L^2(X,\mu)$ into $L^2(X,\nu)$ by means of an operator $T\in\mathscr L(H_\mu,H_\nu)$ (see also \cite[Chap. 4]{jan_1997}). Given $n\in\N$, the first step consists in defining an operator $\widetilde \Gamma_{\nu,\mu,n}(T)$ from $(\mathscr H_n^\mu,\|\cdot\|_{L^2(X,\mu)})$ into $(\mathscr H_n^\nu,\|\cdot\|_{L^2(X,\nu)})$. This is possible, exploiting the isomorphisms $\Psi_n^\mu$ and $\Psi_n^\nu$ between $(\mathscr H_n^\mu,\|\cdot\|_{L^2(X,\mu)})$ and $(H_\mu^{\odot n},\sqrt{n!}\|\cdot\|_{H_\mu^{\otimes n}})$, and $(\mathscr H_n^\nu,\|\cdot\|_{L^2(X,\nu)})$ and $(H_\nu^{\odot n},\sqrt{n!}\|\cdot\|_{H_\nu^{\otimes n}})$, respectively, introduced in \eqref{iso_wiener_chaos_tensor_prod}. For every $T\in\mathscr L(H_\mu,H_\nu)$, we consider the operator 
\begin{align*}
T^{\otimes n}(\Psi_n^\mu)^{-1}:(\mathscr H_n^\mu,\|\cdot\|_{L^2(X,\mu)})\to   (H_\nu^{\odot n},\sqrt{n!}\|\cdot\|_{H_\nu^{\otimes n}}),
\end{align*}
which, given $\alpha\in \Lambda_n$, from \eqref{iso_wiener_chaos_tensor_prod} acts on $\Phi_\alpha^\mu$ as
\begin{align}
\label{espansione_psi_mu}
T^{\otimes n}(\Psi_n^\mu)^{-1}\Phi_\alpha^\mu
= & \frac{\sqrt{\alpha!}}{n!}\sum_{\sigma\in\Sigma_n^\alpha}Te_{i^\alpha_{\sigma(1)}}^\mu\otimes\cdots \otimes Te_{i^\alpha_{\sigma(n)}}^\mu,
\end{align}
where $\{e_n^\mu:n\in\N\}$ is an orthonormal basis of $H_\mu$. In the second step, we apply $\Psi_n^\nu$ to $T^{\otimes n}(\Psi_n^\mu)^{-1}$ to map elements of $H_\nu^{\odot n}$ into elements of $\mathscr H_n^\nu$, obtaining the element $\Psi_n^\nu(T^{\otimes n}(\Psi_n^\mu)^{-1}(\Phi_\alpha^\mu))\in\mathscr H_n^\nu$.
We are ready to introduce the first operator we are interested in.
\begin{defn}\label{def:op_gamma_tilden}
Given two centered Gaussian measures $\mu$ and $\nu$ on $X$ and $T\in\mathscr L(H_\mu,H_\nu)$, for every $n\in\N$ we define
$\widetilde \Gamma_{\nu,\mu,n}(T):\mathscr H_n^\mu\to \mathscr H_n^\nu$ as
\begin{align*}
\widetilde \Gamma_{\nu,\mu,n}(T):=\Psi_n^\nu(T^{\otimes n}(\Psi_n^\mu)^{-1}).   
\end{align*}
\end{defn}
The following result is crucial to define the operator $\widetilde \Gamma_{\nu,\mu}$ when $\|T\|_{\mathscr L(H_\mu,H_\nu)}\leq 1$.
\begin{proposition}
For every $T\in\mathscr L(H_\mu,H_\nu)$ and $n\in\N$, we get $\widetilde \Gamma_{\nu,\mu,n}\in \mathscr L(\mathscr H_{n}^{\mu},\mathscr H_n^\nu)$ and
\begin{align}
\label{norma_tilde_gamma_n}
\|\widetilde \Gamma_{\nu,\mu,n}(T)\|_{\mathscr L(\mathscr H_{n}^{\mu},\mathscr H_n^\nu)}=\|T^{\otimes n}\|_{\mathscr L(H_\mu^{\odot n},H_\nu^{\odot n})}=\|T\|_{\mathscr L(H_\mu,H_\nu)}^n, \qquad n\in\N.
\end{align}    
\end{proposition}
\begin{proof}
Let $n\in\N$ and let $T\in\mathscr l(H_\mu,H_\nu)$. The linearity of $\widetilde \Gamma_{\nu,\mu,n}(T)$ is trivial. To prove \eqref{norma_tilde_gamma_n}, we notice that, for every $\alpha,\beta\in\Lambda_n$, from the definition of $\widetilde \Gamma_{\nu,\mu,n}(T)$ we get
\begin{align}
\label{prod_scal_tilde_gamma}
\langle \widetilde \Gamma_{\nu,\mu,n}(T)(\Phi_\alpha^\mu),\widetilde \Gamma_{\nu,\mu,n}(T)(\Phi_\beta^\mu)\rangle_{\mathscr H_n^\nu} 
= & n!\langle T^{\otimes n}(\Psi_n^\mu)^{-1}(\Phi_\alpha^\mu),T^{\otimes n}(\Psi_n^\mu)^{-1}(\Phi_\beta^\mu)\rangle_{H_\nu^{\otimes n}}.
\end{align}
Now we consider $f\in \mathscr H_n^\mu$. It admits the representation
\begin{align*}
f=\sum_{\alpha\in\Lambda_n}c(f,\mu)_\alpha \Phi_\alpha^\mu, \qquad c(f,\mu)_\alpha=\langle f,\Phi_\alpha^\mu\rangle_{\mathscr H_n^\mu}, \qquad \alpha\in\Lambda_n.    
\end{align*}
From \eqref{espansione_psi_mu} and \eqref{prod_scal_tilde_gamma}, it follows that
\begin{align}
\label{norma_operatore_gamma_tilde_n}
\|\widetilde \Gamma_{\nu,\mu,n}(T)(f)\|_{\mathscr H_n^\nu}^2 
=& n!\langle T^{\otimes n}y_f,T^{\otimes n}y_f\rangle_{H_\nu^{\otimes n}}
=n!\|T^{\otimes n}y_f\|_{H_\nu^{\otimes n}}^2,
\end{align}
where $y_f=(\Psi_n^\mu)^{-1}(f)\in H_\mu^{\odot n}$ is defined as (see also \eqref{espressione_elemento_V_symm_n})
\begin{align*}
y_f=\sum_{\alpha\in\Lambda_n}c_\alpha\sum_{\sigma\in\Sigma_n^\alpha}e_{i^\alpha_{\sigma(1)}}^\mu\otimes \cdots \otimes e_{i^\alpha_{\sigma(n)}}^\mu, \qquad c_\alpha=\frac{\sqrt{\alpha!}}{n!}c(f,\mu)_\alpha, \qquad  \alpha\in\Lambda_n.    
\end{align*}
From \eqref{norma_elemento_v_symm_n}, we infer that
\begin{align}\label{norma_y_f}
\|y_f\|_{H_\mu^{\otimes n}}^2 =\sum_{\alpha\in\Lambda_n}c_\alpha^2\frac{n!}{\alpha!}=\frac{1}{n!}\sum_{\alpha\in\Lambda_n}(c(f,\mu)_\alpha)^2
=\frac{1}{n!}\|f\|_{\mathscr H_n^\mu}^2,  \end{align}
Since $\Psi_n^\mu$ is an isomorphism, for every $y\in H_\mu^{\odot n}$ there exists exactly one $f\in \mathscr H_n^\mu$ such that $y=y_f$ and $\|f\|_{\mathscr H_n^\mu}^2=n!\|y\|_{H_\mu^{\otimes n}}^2$ (it is enough to take $f=\Psi_n^\mu(y)$ and recall that $\Psi_n^\mu$ is an isometry). Thank to \eqref{norma_operatore_gamma_tilde_n} and \eqref{norma_y_f}, it follows that
\begin{align*}
\|\widetilde \Gamma_{\nu,\mu,n}(T)\|_{\mathscr L(\mathscr H_n^\mu,\mathscr H_n^\nu)}&=\sup_{f\neq 0}\frac{\|\widetilde \Gamma_{\nu,\mu,n}(T)(f)\|_{\mathscr H_n^\nu}}{\|f\|_{\mathscr H_n^\mu}}=\sup_{y_f\neq 0}\frac{\sqrt{n!}\|T^{\otimes n}y_f\|_{H_{\mu}^{\otimes n}}}{\sqrt{n!}\|y_f\|_{H_\mu^{\otimes n}}}=\|T^{\otimes n}\|_{\mathscr L(H_\mu^{\odot n},H_\nu^{\odot n})}.
\end{align*}
The first equality in \eqref{norma_tilde_gamma_n} is proved. The second one follows from Corollary \ref{coro:op_tensor_spaces}.
\end{proof}

Equality \eqref{norma_tilde_gamma_n} allows us to define an operator $\widetilde \Gamma_{\nu,\mu}$ whenever $\|T\|_{\mathscr L(H_\mu,H_\nu)}\leq 1$.

\begin{defn}
\label{def:op_gamma_tilde}
Let $\mu$, $\nu$ be two centered Gaussian measures on $X$ with Cameron Martin spaces $H_\mu$ and $H_\nu$, respectively, and let $T\in\mathscr L(H_\mu,H_\nu)$ be such that $\|T\|_{\mathscr L(H_\mu,H_\nu)}\leq 1$. We define the second quantization operator as the linear operator $\widetilde \Gamma_{\nu,\mu}(T):L^2(X,\mu)\to L^2(X,\nu)$ given by
\begin{align}
\label{operatore_tilde_gamma}
(\widetilde \Gamma_{\nu,\mu}(T))f:=\sum_{n=0}^\infty \widetilde \Gamma_{\nu,\mu,n}(T)(I_n^\mu f)=\sum_{n=0}^\infty \Psi_n^\nu(T^{\otimes n}(\Psi_n^\mu)^{-1}I_n^\mu f), \qquad f\in L^2(X,\mu).    
\end{align}
\end{defn}
From \eqref{norma_tilde_gamma_n} and the assumption on $T$, it follows that
\begin{align*}
\|(\widetilde \Gamma_{\nu,\mu}(T))f\|_{L^2(X,\nu)}^2
=\sum_{n=0}^\infty \|\widetilde \Gamma_{\nu,\mu,n}(T)(I_n^\mu f)\|_{L^2(X,\nu)}^2
\leq \sum_{n=0}^\infty \|T\|_{\mathscr L(H_\mu,H_\nu)}^{2n}\|I_n^\mu f\|_{L^2(X,\mu)}^2
\leq \|f\|_{L^2(X,\mu)}^2
\end{align*}
for every $f\in L^2(X,\mu)$, which means that $\widetilde \Gamma_{\nu,\mu}(T)$ is a contraction from $L^2(X,\mu)$ into $L^2(X,\nu)$.

The following result holds true.
\begin{proposition}
\label{prop:prop_second_quant_serie}
Let $\mu$, $\nu$ and $\sigma$ be three centered Gaussian measures on $X$ with Cameron Martin spaces $H_\mu$, $H_\nu$ and $H_\sigma$, respectively. 
\begin{enumerate}[{\rm (i)}] 
\item  If $\nu=\mu$ and $T=I_{H_\mu}$, then $\widetilde \Gamma_{\mu,\mu}(T)=I_{L^2(X,\mu)}$.
\item If $T\in\mathscr L(H_\mu,H_\nu)$ satisfies $\|T\|_{\mathscr L(H_\mu,H_\nu)}\leq 1$, then $(\widetilde \Gamma_{\nu,\mu}(T))^\star =\widetilde \Gamma_{\mu,\nu}(T^\star )$. 
\item If $T\in\mathscr L(H_\mu,H_\sigma)$ and $S\in\mathscr L(H_\sigma,H_\nu)$ are such that $\|T\|_{\mathscr L(H_\mu,H_\sigma)},\|S\|_{\mathscr L(H_\sigma,H_\nu)}\leq 1$ then
\begin{align}
\label{composizione_completa}
\widetilde \Gamma_{\nu,\mu}(ST)=\widetilde \Gamma_{\nu,\sigma}(S)\widetilde \Gamma_{\sigma,\mu}(T).    
\end{align}
\end{enumerate}
\end{proposition}
\begin{proof}
\leavevmode
\begin{enumerate}[{\rm (i)}]

\item It follows from the definition of $\widetilde\Gamma_{\mu,\mu}$.

\item We firstly observe that $\|T^\star \|_{\mathscr L(H_\nu,H_\mu)}=\|T\|_{\mathscr L(H_\mu,H_\nu)}\leq 1$, which means that $\widetilde \Gamma_{\mu,\nu}(T^\star )$ is well-defined. By \eqref{iso_wiener_chaos_tensor_prod}, for every $n\in\N$ and every $\alpha,\beta\in\Lambda_n$, we have
\begin{align*}
\langle \widetilde \Gamma_{\nu,\mu,n}(T)\Phi_\alpha^\mu,\Phi_\beta^\nu\rangle_{L^2(X,\nu)}
= & \langle \Psi_n^\nu(T^{\otimes n}(\Psi_n^\mu)^{-1}\Phi_\alpha^\mu), \Phi_\beta^\nu\rangle_{L^2(X,\nu)}\\
=& n!\langle T^{\otimes n}(\Psi_n^\mu)^{-1}\Phi_\alpha^\mu, (\Psi_n^\nu)^{-1}\Phi_\beta^\nu\rangle_{H_\nu^{\otimes n}} \\
= & \frac{\sqrt{\alpha!\beta!}}{n!}\sum_{\sigma_1,\sigma_2\in\Sigma_n}\langle  Te_{i^\alpha_{\sigma_1(1)}}^\mu,e_{i^\beta_{\sigma_2(1)}}^\nu\rangle_{H_\nu}\cdots\langle Te_{i^\alpha_{\sigma_1(n)}}^\mu,e_{i^\beta_{\sigma_2(n)}}^\nu\rangle_{H_\nu} \\
= & \frac{\sqrt{\alpha!\beta!}}{n!}\sum_{\sigma_1,\sigma_2\in\Sigma_n}\langle  e_{i^\alpha_{\sigma_1(1)}}^\mu,T^\star e_{i^\beta_{\sigma_2(1)}}^\nu\rangle_{H_\mu}\cdots\langle e_{i^\alpha_{\sigma_1(n)}}^\mu,T^\star e_{i^\beta_{\sigma_2(n)}}^\nu\rangle_{H_\mu} \\
= & n!\langle (T^\star )^{\otimes n}(\Psi_n^\nu)^{-1}\Phi_\beta^\nu, (\Psi_n^\mu)^{-1}\Phi_\alpha^\mu\rangle_{H_\mu^{\otimes n}}\\
=& \langle \widetilde \Gamma_{\mu,\nu,n}(T^\star )\Phi_\beta^\nu,\Phi_\alpha^\mu\rangle_{L^2(X,\mu)}.
\end{align*}

Since $\alpha,\beta\in \Lambda_n$, it follows that $\widetilde \Gamma_{\nu,\mu}(T)\Phi_\alpha^\mu=\widetilde \Gamma_{\nu,\mu,n}(T)\Phi_\alpha^\mu$ and $\widetilde \Gamma_{\mu,\nu}(T^\star )\Phi_\beta^\nu=\widetilde \Gamma_{\mu,\nu,n}(T^\star )\Phi_\beta^\nu$. Therefore, $\langle \widetilde\Gamma_{\nu,\mu}(T)\Phi_\alpha^\mu,\Phi_\beta^\nu\rangle_{L^2(X,\mu)}=\langle \widetilde\Gamma_{\mu,\nu}(T^\star )\Phi_\beta^\nu,\Phi_\alpha^\mu\rangle_{L^2(X,\mu)}$. By density we get the assertion. 

\item For every $f\in L^2(X,\mu)$, we get
\begin{align}
\label{forma_comp_1}
(\widetilde \Gamma_{\nu,\mu}(ST))f
= \sum_{n=0}^\infty \Psi_n^\nu((ST)^{\otimes n}(\Psi_n^\mu)^{-1}I_n^\mu f).
\end{align}
For every $n\in\N$ and $v_1,\ldots,v_n\in H_\mu$, we have
\begin{align*}
(ST)^{\otimes n}v_1\otimes \cdots \otimes v_n
&= (STv_1)\otimes \cdots \otimes (STv_n)
= S^{\otimes n}(T^{\otimes n}v_1\otimes \cdots \otimes v_n)\\
&=(S^{\otimes n}T^{\otimes n})v_1\otimes \cdots\otimes v_n,
\end{align*}
which means that $(ST)^{\otimes n}=S^{\otimes n}T^{\otimes n}$. Replacing in \eqref{forma_comp_1}, we deduce that
\begin{align*}
(\widetilde \Gamma_{\nu,\mu}(ST))f
= \sum_{n=0}^\infty \Psi_n^\nu((S^{\otimes n}T^{\otimes n})(\Psi_n^\mu)^{-1}I_n^\mu f), \qquad f\in L^2(X,\mu).    
\end{align*}
Since $\Psi_n^\sigma(T^{\otimes n}(\Psi_n^{\mu})^{-1}I_n^\mu f)\in \mathscr H_n^\sigma$ and $S^{\otimes n}(\Psi_n^{\sigma})^{-1}(\Psi_n^\sigma (T^{\otimes n}(\Psi_n^{\mu})^{-1}I_n^\mu f))\in H_\nu^{\odot n} $ for every $n\in\N$, we get
\begin{align*}
(\widetilde \Gamma_{\nu,\mu}(ST))f
= & \sum_{n=0}^\infty \Psi_n^\nu((S^{\otimes n}(\Psi_n^\sigma)^{-1}\Psi_n^\sigma T^{\otimes n})(\Psi_n^\mu)^{-1}I_n^\mu f) \\
= & \sum_{n=0}^\infty \Psi_n^\nu\big(S^{\otimes n}(\Psi_n^\sigma)^{-1}(\Psi_n^\sigma (T^{\otimes n}(\Psi_n^\mu)^{-1}I_n^\mu f))\big) \\
= & \sum_{n=0}^\infty \widetilde \Gamma_{\nu,\sigma,n}(S)(\widetilde \Gamma_{\sigma,\mu,n}(T)I_n^\mu f) 
= \sum_{n=0}^\infty \widetilde \Gamma_{\nu,\sigma,n}(S)\left(I_n^\sigma\Big(\sum_{m=0}^\infty\widetilde \Gamma_{\sigma,\mu,m}(T)I_m^\mu f\Big)\right) \\
= & \widetilde \Gamma_{\nu,\sigma}(S)(\widetilde \Gamma_{\sigma,\mu}(T)f), \qquad f\in L^2(X,\mu),
\end{align*}
where we have used the fact that 
\begin{align*}
I_n^\sigma(\widetilde \Gamma_{\sigma,\mu}(T))
= I_n^\sigma\Big(\sum_{m=0}^\infty\widetilde\Gamma_{\sigma,\mu,m}(T)I_m^\mu f\Big)=\widetilde \Gamma_{\sigma,\mu,n}(T)I_n^\mu f, \qquad n\in\N, \ f\in L^2(X,\mu).
\end{align*}
\end{enumerate}
\end{proof}

We write the following corollary of Proposition \ref{prop:prop_second_quant_serie}(iii), that will be useful in the sequel.
\begin{corollary}
\label{composizione0}
Let $\mu,\nu$ be two Gaussian measures on $X$ with Cameron-Martin spaces $H_\mu$ and $H_\nu$, respectively. 
If $T\in\mathscr L(H_\mu,H_\nu)$ with $\norm{T}_{\mathscr L(H_\mu,H_\nu)}\leq 1$, then, for every $c\in[0,1]$ and $f\in L^2(X,\mu)$,
\begin{align}
&\widetilde{\Gamma}_{\nu,\mu}(cL)f
= \widetilde{\Gamma}_{\nu,\mu}(L)\widetilde{\Gamma}_{\mu,\mu}(cI)f. \label{composizione1}\\ \label{composizione2}
&\widetilde{\Gamma}_{\nu,\mu}(cL)f = \widetilde{\Gamma}_{\nu,\nu}(cI)\widetilde{\Gamma}_{\mu,\mu}(L)f.
\end{align}
\end{corollary}

We now compute the action of $\widetilde \Gamma_{\nu,\mu}(T)$ on functions belonging to $\mathcal{E}_{\mu}$ (see Definition \ref{def_e_gamma}). To this aim, we introduce an operator which will be useful also in the sequel.
\begin{defn}
\label{def:hatT}
Given $T\in\mathscr L(H_\mu,H_\nu)$, we set $\widehat T=Q_\nu^{-\frac12}TQ_\mu^\frac12\in\mathscr L(X)$.   
\end{defn} 

\begin{proposition}
\label{prop:conto_gamma_tilde_exp_fct}
Let $\mu,\nu$ be centered Gaussian measures on $X$ with Cameron-Martin spaces $H_\mu$ and $H_\nu$, respectively. Let $T\in\mathscr L(H_\mu,H_\nu)$ be such that $\|T\|_{\mathscr L(H_\mu,H_\nu)}\leq 1$. Therefore,
\begin{align}
\label{azione_tilde_gamma_exp}
\widetilde \Gamma_{\nu,\mu}(T)(E^\mu_h)=\sum_{n=0}^\infty I_n^\nu(E^\nu_{\widehat Th})=E^\nu_{\widehat Th}, \qquad E^\mu_h\in\mathcal{E}_\mu,\ h\in X. 
\end{align}  
\end{proposition}
\begin{proof}
Let $h\in X$ and consider $E^\mu_h=e^{W^\mu_h(x)-\frac12\| h\|_{X}^2}$ for all $x\in X$. Let $\Phi_\alpha^\mu$ be as in \eqref{pol_herm_gen} for some $\alpha=(\alpha_1,\alpha_2,\ldots)\in \Lambda$. Since $\{e_n^\mu:n\in\N\}$ is an orthonormal basis of $H_\mu$ consisting of eigenvectors of $Q_\mu$, it follows that
\begin{align}
\label{calcolo_proiez_E_mu_alpha}
\int_X E^\mu_h\Phi_\alpha^\mu d\mu
=  & e^{-\frac12\|h\|^2_{X}}\sqrt{\alpha!}\int_Xe^{W^\gamma_h(x)}\prod_{j=1}^\infty\phi_{\alpha_j}\Bigl( W^\gamma_{Q_\mu^{-\frac 12}e_j^\mu}(x)\Bigr)\mu(dx) \notag \\
= & e^{-\frac12\|h\|_{X}^2}\sqrt{\alpha!}\int_X\exp\Bigl({\displaystyle\sum_{n=1}^\infty\langle h,Q_\mu^{-\frac12}e_n^\mu\rangle_X W^\gamma_{Q_\mu^{-\frac 12}e_n^\mu}(x)}\Bigr)\prod_{j=1}^\infty\phi_{\alpha_j}\Bigl( W^\gamma_{Q_\mu^{-\frac 12}e_j^\mu}(x)\Bigr)\mu(dx)\notag  \\
= & e^{-\frac12\|h\|_{X}^2}\sqrt{\alpha!}\prod_{j=1}^\infty\int_X\exp\Bigl({\displaystyle\langle h,Q_\mu^{-\frac12}e_j^\mu\rangle_X W^\gamma_{Q_\mu^{-\frac 12}e_j^\mu}(x)}\Bigr)\phi_{\alpha_j}\Bigl( W^\gamma_{Q_\mu^{-\frac 12}e_j^\mu}(x)\Bigr)\mu(dx),
\end{align}
where we have exploited the independence of $ W^\gamma_{Q_\mu^{-\frac 12}e_j^\mu}$ with $j$ varying in $\N$ and the fact that
\begin{align*}
W^\gamma_h(x)
= \sum_{n=1}^\infty \langle h,Q_\mu^{-\frac12}e_n^\mu\rangle_X W^\gamma_{Q_\mu^{-\frac 12}e_n^\mu}(x), \qquad h,x\in X.
\end{align*}
By applying the change of variables $\xi=W^\gamma_{Q_\mu^{-\frac 12}e_j^\mu}(x)$ for every $j\in\N$, we get
\begin{align}
\label{change_of_variables}
\int_X\exp\Bigl({\displaystyle\langle h,Q_\mu^{-\frac12}e_j^\mu\rangle_X W^\gamma_{Q_\mu^{-\frac 12}e_j^\mu}(x)}\Bigr)\phi_{\alpha_j}\Bigl( W^\gamma_{Q_\mu^{-\frac 12}e_j^\mu}(x)\Bigr)\mu(dx)
= & \frac1{\sqrt{2\pi}}\int_{\R}e^{\langle h,Q_\mu^{-\frac12}e_j^\mu\rangle_X\xi}\phi_{\alpha_j}(\xi)e^{-\frac12\xi^2}d\xi
\end{align}
and, recalling that
\begin{align*}
\phi_n(\xi)=\frac{(-1)^n}{n!}e^{\frac12 \xi^2}\frac{d^n}{d\xi^n}e^{-\frac12\xi^2},  \quad \xi\in\R, \ n\in\N\cup\{0\}, 
\end{align*}
integrating by parts $\alpha_j$ times in \eqref{change_of_variables} it follows that
\begin{align}
\label{int_parti_E_mu_alpha}
& \int_X\exp\Bigl({\displaystyle\langle h,Q_\mu^{-\frac12}e_j^\mu\rangle_X W^\gamma_{Q_\mu^{-\frac 12}e_j^\mu}(x)}\Bigr)\phi_{\alpha_j}\Bigl( W^\gamma_{Q_\mu^{-\frac 12}e_j^\mu}(x)\Bigr)\mu(dx) \notag \\
& = \frac{\langle h,Q_\mu^{-\frac12}e_j^\mu\rangle_X^{\alpha_j}}{\sqrt{2\pi}\alpha_j!}\int_{\R}e^{\langle h,Q_\mu^{-\frac12}e_j^\mu\rangle_X\xi-\frac12\xi^2}d\xi 
=  \frac{\langle h,Q_\mu^{-\frac12}e_j^\mu\rangle_X^{\alpha_j}}{\alpha_j!}e^{\frac12\langle h,Q_\mu^{-\frac12}e_j^\mu\rangle_X^2}.    
\end{align}
From \eqref{calcolo_proiez_E_mu_alpha} and \eqref{int_parti_E_mu_alpha}, and recalling that $h=\sum_{j=1}^\infty \langle h,Q_\mu^{-\frac12}e_j^\mu\rangle_XQ_{\mu}^{-\frac12}e_j^\mu$, we infer that
\begin{align}
\label{dec_E_h}
\int_X E^\mu_h\Phi_\alpha^\mu d\mu
= & \frac{1}{\sqrt{\alpha!}}\prod_{j=1}^\infty 
\langle h,Q_\mu^{-\frac12}e_j^\mu\rangle_X^{\alpha_j}
, \qquad 
\left(\int_X E^\mu_h\Phi_\alpha^\mu d\mu\right) \Phi_\alpha^\mu
= \frac{1}{\sqrt{\alpha!}}\prod_{j=1}^\infty\langle h,Q_\mu^{-\frac12}e_j^\mu\rangle_X^{\alpha_j}\Phi_\alpha^\mu.
\end{align}

Fix $n\in\N$ and let us compute $\widetilde \Gamma_{\nu,\mu,n}(T)E^\mu_h$. From \eqref{iso_wiener_chaos_tensor_prod} and \eqref{dec_E_h}, it follows that
\begin{align}
\label{formula_psi_-1_exp}
(\Psi_n^\mu)^{-1}I_n^\mu E^\mu_h
= &(\Psi_n^\mu)^{-1}\left(\sum_{\alpha\in\Lambda_n}\frac{1}{\sqrt{\alpha!}}\prod_{j=1}^\infty \langle h,Q_\mu^{-\frac12}e_j^\mu\rangle_X^{\alpha_j}\Phi_{\alpha}^\mu\right) \notag \\
= & \frac1{n!}\sum_{\alpha\in\Lambda_n}\langle h,Q_\mu^{-\frac12}e_{i^\alpha_1}^\mu\rangle_X \cdots\langle h,Q_\mu^{-\frac12}e_{i^\alpha_n}^\mu\rangle_X\sum_{\sigma\in\Sigma_n^\alpha}e_{i^\alpha_{\sigma(1)}}^\mu\otimes \cdots \otimes e_{i^\alpha_{\sigma(n)}}^\mu \notag \\
= & \frac{1}{n!}\sum_{i_1,\ldots,i_n=1}^\infty\langle h,Q_\mu^{-\frac12}e_{i_1}^\mu\rangle_X e_{i_1}^\mu\otimes\cdots \otimes \langle h,Q_\mu^{-\frac12}e_{i_n}^\mu\rangle_X e_{i_n}^\mu \notag \\
= & \frac{1}{n!}\sum_{i_1,\ldots,i_n=1}^\infty\langle Q_\mu^\frac12 h,e_{i_1}^\mu\rangle_{H_\mu} e_{i_1}^\mu\otimes\cdots \otimes \langle Q_\mu^\frac12 h,e_{i_n}^\mu\rangle_{H_\mu} e_{i_n}^\mu,
\end{align}
where we have exploited the fact that 
$\langle h,Q_\mu^{-\frac12}e_{i_j}^\mu\rangle_X
= \langle Q_\mu^{\frac12}h,e_{i_j}^\mu\rangle_{H_\mu}$ for every $j\in\{1,\ldots,n\}$. Hence,
\begin{align}
\label{passaggo_mu_nu_esponenziali}
& T^{\otimes n}(\Psi_n^{\mu})^{-1}I_n^\mu E^\mu_h \notag \\
= & \frac{1}{n!}\sum_{i_1,\ldots,i_n=1}^\infty\langle Q_\mu^\frac12 h,e_{i_1}^\mu\rangle_{H_\mu} (Te_{i_1}^\mu)\otimes\cdots \otimes \langle Q_\mu^\frac12 h,e_{i_n}^\mu\rangle_{H_\mu} (Te_{i_n}^\mu) \notag \\
= & \frac{1}{n!}\sum_{i_1,\ldots,i_n,k_1,\ldots,k_n=1}^\infty\langle Q_\mu^{\frac12}h,e_{i_1}^\mu\rangle_{H_\mu} \langle Te_{i_1}^\mu,e_{k_1}^\nu\rangle_{H_\nu}e_{k_1}^\nu\otimes\cdots \otimes \langle Q_\mu^{\frac12}h,e_{i_n}^\mu\rangle_{H_\mu} \langle Te_{i_n}^\mu,e_{k_n}^\nu\rangle_{H_\nu}e_{k_n}^\nu \notag \\
= & \frac{1}{n!}\sum_{k_1,\ldots,k_n=1}^\infty\langle TQ_\mu^{\frac12}h,e_{k_1}^\nu\rangle_{H_\nu}e_{k_1}^\nu\otimes \cdots \otimes \langle TQ_\mu^{\frac12}h,e_{k_n}^\nu\rangle_{H_\nu}e_{k_n}^\nu.
\end{align}
From \eqref{formula_psi_-1_exp} and \eqref{passaggo_mu_nu_esponenziali}, with $\mu$ and $h$ replaced by $\nu$ and $TQ_\mu^{\frac12}h$, respectively, it follows that
\begin{align*}
& \frac{1}{n!}\sum_{k_1,\ldots,k_n=1}^\infty\langle TQ_\mu^{\frac12}h,e_{k_1}^\nu\rangle_{H_\nu}e_{k_1}^\nu\otimes \cdots \otimes \langle TQ_\mu^{\frac12}h,e_{k_n}^\nu\rangle_{H_\nu}e_{k_n}^\nu 
=  (\Psi_n^\nu)^{-1}(I_n^\nu E^\nu_{Q_\nu^{-\frac12}TQ_\mu^{\frac12}h}),
\end{align*}
which gives $\widetilde \Gamma_{\nu,\mu,n}(T)(I_n^{\mu}E^\mu_h)=\Psi_n^{\nu}(T^{\otimes n}(\Psi_n^{\mu})^{-1}I_n^\mu E_{h}^\mu)=I_n^\nu(E^\nu_{\widehat Th})$. We conclude that 
\begin{align*}
\widetilde \Gamma_{\nu,\mu}(T)(E_h^\mu)
= & \sum_{n=0}^\infty \widetilde \Gamma_{\nu,\mu,n}(T)(I_n^\mu E_h^\mu)
= \sum_{n=0}^\infty I_n^\nu(E_{\widehat T h}^\nu)=E_{\widehat Th}^\nu.
\end{align*}
\end{proof}
\begin{remark}
\label{rmk:eq_operatore_sec_quant_choj}
We claim that Proposition \ref{prop:conto_gamma_tilde_exp_fct} gives the identity $\Gamma(\widehat T)=\widetilde{\Gamma}_{\mu,\mu}(T)$ when $\mu$ is a non-degenerate Gaussian measure on $X$, $T\in\mathscr L(H_\mu)$, $\widehat T\in\mathscr L(X)$ has been defined in \ref{def:hatT},
$\Gamma(\widehat T)$ is the operator defined in \cite{cho_gol_1996} for $T\in\mathscr L(X)$ with $\|T\|_{\mathscr L(X)}\leq 1$ and $\widetilde \Gamma_{\mu,\mu}(\widetilde T)$ is the operator introduced in \eqref{operatore_tilde_gamma}. Indeed, $\|T\|_{\mathscr L(H_\mu)}=\|\widehat T\|_{\mathscr L(X)}$. Further, in \cite[Proposition  1]{cho_gol_1996} it has been proved that $\Gamma(\widehat T)E_h^\mu=E_{\widehat Th}^\mu$ for every $h\in H_\mu$. From \eqref{azione_tilde_gamma_exp}, with $\nu=\mu$, we infer that
\begin{align*}
\widetilde \Gamma_{\mu,\mu}(T)E_h^\mu =E_{\widehat Th}^\mu, \qquad h\in H_\mu.  
\end{align*}
The claim follows from the linear density of $\mathcal D_\mu$ in $L^2(X,\mu)$.
\end{remark}

\section{The second quantization operator via an integral representation formula}\label{second_integrale}

Let $\mu=\mathcal{N}(0,Q_\mu)$ and $\nu=\mathcal{N}(0,Q_\nu)$ be two centered Gaussian measures on $X$ with Cameron-Martin spaces $H_\mu$ and $H_\nu$, respectively. We recall that $H_\mu=Q_\mu^{\frac{1}{2}}(X)$ and $H_\nu=Q_{\nu}^{\frac{1}{2}}(X)$ and we denote by $P_\mu$ e $P_\nu$ the orthogonal projections on ${\rm ker} (Q_\mu)$ and on ${\rm ker} (Q_\nu)$, respectively. 

We fix $T\in \mathscr L(H_\mu,H_\nu)$. 
If $\{e_n^\mu:n\in\N\}$ and $\{e_n^\nu:n\in\N\}$ are orthonormal bases of $H_\mu,H_\nu$ consisting of eigenvectors of $Q_\mu$ and $Q_\nu$, with corresponding eigenvalues $\{\lambda_n^\mu\}_{n\in\N}$ and $\{\lambda_n^{\nu}\}_{n\in\N}$, respectively, then for every $h\in H_\mu$ we get
\begin{align*}
Th
=\sum_{n=1}^\infty \sum_{k=1}^\infty \langle e_k^\mu,h\rangle_{H_\mu}\langle e_n^\nu,T e_k^\mu\rangle_{H_\nu} e_n^\nu.
\end{align*}
In particular, for every $n\in\N$ there exist $f_n^\mu,f_n^\nu\in X$, eigenvectors of $Q_\mu$ and $Q_\nu$ with corresponding eigenvalues $\lambda_n^\mu$ and $\lambda_n^\nu$, respectively, such that $e_{n}^\mu=(\lambda_n^\mu)^{\frac12}f_n^\mu$ and $e_{n}^\nu=(\lambda_n^\nu)^{\frac12}f_n^\nu$ and $\{f_n^\mu:n\in\N\}$ and $\{f_n^\nu:n\in\N\}$ are orthonormal systems in $X$. This implies that $\langle e_n^\mu,e_\ell^\mu\rangle_X=\delta_{n\ell}\lambda_n^\mu$, $\langle e_n^\nu,e_\ell^\nu\rangle_X=\delta_{n\ell}\lambda_n^\nu$ and $\lambda_n^\mu\neq 0\neq \lambda_n^\nu$ for every $n,\ell\in\N$. \\  
We shall study an extension of $T$ belonging to $L^2(X,\mu;X)$. For this purpose, we introduce the linear approximations $T_{\infty,m}^{\nu,\mu}:X\to X$, $m\in\N $, defined for every $x\in X$ as
\begin{align}
\label{Linftym}
T_{\infty,m}^{\nu,\mu}x
= & T\pi_m x
= \sum_{n=1}^\infty\sum_{k=1}^m \langle e_k^\mu,\pi_mx\rangle_{H_\mu}\langle e_n^\nu,Te_k^\mu\rangle_{H_\nu} e_n^\nu
= \sum_{n=1}^\infty\sum_{k=1}^m \frac{\langle Q_\mu^{-\frac12}e_k^\mu,x\rangle_X}{(\lambda_k^\mu)^{\frac{1}{2}}}\langle e_n^\nu,T e_k^\mu\rangle_{H_\nu} e_n^\nu, \end{align}
where
\begin{align*}
\pi_m x=\sum_{k=1}^m\frac{\langle Q_\mu^{-\frac12}e_k^\mu,x\rangle_X}{(\lambda_k^\mu)^{\frac12}}e_k^\mu, \qquad x\in X, \ m\in\N.     
\end{align*}
We notice that the operators $T_{\infty,m}^{\nu,\mu}$ are well-defined and belong to $\mathscr L(X)$ for every $m\in\N$, since
\begin{align*}
\|T_{\infty,m}^{\nu,\mu}x\|_X
\leq c_\nu\sum_{k=1}^m\lambda_k^{-\frac12}\|Te_k^\mu\|_{H_\nu} \|x\|_X, 
\qquad m\in\N,  
\end{align*}
where $c_\nu$ denotes the embedding constant of $H_\nu$ in $X$ (see \eqref{norma_forte_R}). 

The sequence $\{T_{\infty,m}^{\nu,\mu}\}_{m\in\N }$ converges in $L^2(X,\mu;X)$, as stated in the following result. 

\begin{proposition}
\label{prop:op_L_infty2}
Let $\mu=\mathcal{N}(0,Q_\mu)$ and $\nu=\mathcal{N}(0,Q_\nu)$ be two centered Gaussian measures on $X$ with Cameron-Martin spaces $H_\mu$ and $H_\nu$, respectively. If $T\in\mathscr L(H_\mu,H_\nu)$, then the sequence $\{T_{\infty,m}^{\nu,\mu}\}_{m\in\N }$ converges in $L^2(X,\mu;X)$ as $m\to \infty$. Its limit $T_\infty^{\nu,\mu}$ is a centered $X$-valued Gaussian random variable with law
\begin{align}\label{leggeLinfinito}
\mu\circ (T_\infty^{\nu,\mu})^{-1}=
\mathcal{N}(0,Q_\nu^{\frac12}(\widehat {T^\star})^\star\widehat {T^\star}Q_\nu^{\frac12}),
\end{align}
where $\widehat {T^\star}=Q_\mu^{-\frac12}T^\star Q_\nu^{\frac12}$ has been introduced in Definition \ref{def:hatT}. Moreover,
\begin{align*}
\int_X|T_\infty^{\nu,\mu}x|_X^2\mu(dx)={\rm Tr}_X[Q_\nu^{\frac12}(\widehat {T^\star})^\star\widehat {T^\star} Q_\nu^{\frac12}],  
\end{align*}
the Cameron-Martin space of $\mathcal{N}(0,Q_\nu^{\frac12}(\widehat {T^\star})^\star\widehat {T^\star} Q_\nu^{\frac12})$ is continuously embedded in $H_\nu$ and
\begin{equation}
\label{gauss_rand_var}
\langle T_\infty^{\nu,\mu}x,h\rangle_X=W^\mu_{\widehat{T^\star}Q_\nu^\frac12 h}(x)\ , \qquad h,x\in X,
\end{equation}
where both the sides of \eqref{gauss_rand_var} are meant as elements of $L^2(X,\mu)$.
\end{proposition}
\begin{proof}
We show that $\{T_{\infty,m}^{\nu,\mu}\}_{m\in\N }$ is a Cauchy sequence in $L^2(X,\mu;X)$. Since $\langle e_n^\nu,e_\ell^\nu\rangle_X=\delta_{n\ell}\lambda_n^\nu$ for every $\ell,n\in\N$, it follows that, for every $m,h\in\N$, with $h>m$,
\begin{align}
\notag
\|T_{\infty,m}^{\nu,\mu}-T_{\infty,h}^{\nu,\mu}\|^2_{L^2(X,\mu;X)}
= & \int_X\norm{\sum_{n=1}^\infty \sum_{k=m+1}^h \frac{\langle Q_\mu^{-\frac12}e_k^\mu,x\rangle_X}{(\lambda_k^\mu)^{\frac{1}{2}}}\langle e_n^\nu,Te_k^\mu\rangle_{H_\nu} e_n^\nu}^2_X\mu(dx) \\
= & \sum_{n=1}^\infty \int_X \lambda_n^\nu\left(\sum_{k=m+1}^h \langle e_n^\nu,Te_k^\mu\rangle_{H_\nu}\frac{\langle Q_\mu^{-\frac12}e_k^\mu,x\rangle_X}{(\lambda_k^\mu)^{\frac{1}{2}}}\right)^2\mu(dx).
\label{cauchy_sequ_L_infty_12}
\end{align}
We recall that
\begin{align} 
\label{fannozero}&\int_X \langle Q_\mu^{-\frac12}e_k^\mu,x\rangle_X\langle Q_\mu^{-\frac12}e_j^\mu,x\rangle_X\mu(dx)=\delta_{jk}\lambda_k^\mu, \qquad  k,j\in\N.
\end{align}
By replacing \eqref{fannozero} in \eqref{cauchy_sequ_L_infty_12}, we infer 
\begin{align*}
\|T_{\infty,m}^{\nu,\mu}-T_{\infty,h}^{\nu,\mu}\|^2_{L^2(X,\mu;X)}
= & \sum_{n=1}^\infty \lambda_n^\nu\sum_{k=m+1}^h\langle e_n^\nu,Te_k^\mu\rangle_{H_\nu}^2
=\sum_{n=1}^\infty \lambda_n^\nu\sum_{k=m+1}^h\langle T^\star e_n^\nu,e_k^\mu\rangle_{H_\mu}^2. 
\end{align*}
Since
\begin{align*}
\sum_{n=1}^\infty\lambda_n^\nu \sum_{k=1}^\infty\langle T^\star e_n^\nu,e_k^\mu\rangle_{H_\mu}^2
= & \sum_{n=1}^\infty\sum_{k=1}^\infty\langle T^\star Q_\nu^{\frac{1}{2}}e_n^\nu,e_k^\mu\rangle_{H_\mu}^2
=\sum_{n=1}^\infty\norm{T^\star Q_\nu^{\frac12}e_n^\nu}_{H_\mu}^2 \\
= & \sum_{n=1}^\infty \|(Q_\mu^{-\frac12}T^\star Q_\nu^{\frac12})Q_\nu^\frac12 f_n\|_X^2
=\sum_{n=1}^\infty \|\widehat {T^\star}Q_\nu^{\frac12}f_n^\nu\|_X^2
={\rm Tr}[Q_\nu^{\frac12}(\widehat {T^\star})^\star\widehat {T^\star} Q_\nu^{\frac12}],
\end{align*}
where we recall that 
$e_n=Q_\nu^{\frac12}f_n$ for every $n\in\N$. It follows that $\{T_{\infty,m}^{\nu,\mu}\}_{m\in\N}$ is a Cauchy sequence in $L^2(X,\mu;X)$ and therefore it converges in $L^2(X,\mu;X)$. 
Let us denote by $T_\infty^{\nu,\mu}$ its limit. Arguing as above, we get
\begin{align*}
\|T_\infty^{\nu,\mu}\|_{L^2(X,\mu;X)}^2
= & \int_X\norm{T_\infty^{\nu,\mu}x}_X^2\mu(dx)
= \sum_{n=1}^\infty \sum_{k=1}^\infty\langle T^\star Q_\nu^{\frac12}e_n^\nu,e_k^\mu\rangle_{H_\mu}^2
={\rm Tr}[Q_\nu^{\frac12}(\widehat {T^\star})^\star\widehat {T^\star}Q_\nu^{\frac12}].
\end{align*}
Since for every $x\in X$ we get
\begin{align*}
\|\widehat {T^\star} Q_\nu^\frac12x\|_X^2
= & \langle \widehat {T^\star} Q_\nu^{\frac12} x,\widehat {T^\star} Q_\nu^{\frac12}x \rangle_X 
= \langle T^\star Q_\nu x,T^\star Q_\nu x\rangle_{H_\mu} 
= \|T^\star Q_\nu x\|_{H_\mu}^2 
\leq  \|T^\star\|_{\mathscr L(H_\nu,H_\mu)}^2\|Q_\nu x\|_{H_\nu}^2 \\
= & \|T^\star\|_{\mathscr L(H_\nu,H_\mu)}^2\|Q_\nu^\frac12 x\|_X^2,
\end{align*}
from Proposition \ref{pseudo} the Cameron-Martin space of $\mathcal{N}(0,Q_\nu^{\frac12}(\widehat {T^\star})^\star\widehat {T^\star}Q_\nu^{\frac12})$ is continuously embedded in $H_\nu$.

Now we fix $h\in X$. Since $\langle e_n^\nu,h\rangle_X=\langle e_n^\nu,Q_\nu h\rangle_{H_\nu}$ for every $n\in\N$, it follows that 
\begin{align*}
\langle T_{\infty,m}^{\nu,\mu}x,h\rangle_X
= & \sum_{n=1}^\infty \sum_{k=1}^m \frac{\langle Q_\mu^{-\frac12}e_k^\mu,x\rangle_X}{(\lambda_k^\mu)^{\frac{1}{2}}}\langle T^\star e_n^\nu,e_k^\mu\rangle_{H_\mu}\langle e_n^\nu,h\rangle_X \\
= & \sum_{k=1}^m \frac{\langle Q_\mu^{-\frac12}e_k^\mu,x\rangle_X}{(\lambda_k^\mu)^{\frac{1}{2}}}\langle Q_\mu^{-\frac12}T^\star Q_\nu h,Q_\mu^{-\frac12}e_k^\mu\rangle_X,
\end{align*}
which converges to $W^\mu_{\widehat {T^\star} Q_\nu^\frac12 h}$ in $L^2(X,\mu)$ by Proposition \ref{teoria} and, from Proposition \ref{prop:cameron-martin},
$$\langle T_\infty^{\nu,\mu}\cdot,h\rangle_X=W^\mu_{\widehat {T^\star} Q_\nu^\frac12 h}(\cdot)$$ is a real centered Gaussian random variable with variance $\norm{T^\star Q_{\nu}h}_{H_\mu}^2$. \\
To conclude, we prove that the law of $T_\infty^{\nu,\mu}$ is given by \eqref{leggeLinfinito}. For every $h\in X$ we get
\begin{align*}
\widehat{(\mu\circ (T_\infty^{\nu,\mu})^{-1})}(h)
= & \int_Xe^{i\langle x,h\rangle_X}(\mu\circ (T_\infty^{\nu,\mu})^{-1})(dx)= \int_Xe^{i\langle T_\infty^{\nu,\mu}x,h\rangle_X}\mu(dx) \\
=&
\exp\left(-\frac{1}{2}\norm{T^\star Q_\nu h}^2_{H_\mu}\right)
= 
\exp\left(-\frac12\langle Q_\nu^{\frac12}(\widehat {T^\star})^\star\widehat {T^\star}  Q_\nu^{\frac12} h,h\rangle_X\right).
\end{align*}
\end{proof}

Before to state some properties of the operator $T_\infty^{\nu,\mu}$, we recall basic notions on measurable operators. 

\begin{defn}  Let $X$ and $Y$ be two separable Hilbert spaces and let $\gamma$ be a Gaussian measure on X. Let $\mathcal{B}(X)$ be the $\sigma$-algebra of the Borel sets on $X$ and let  $\mathcal{B}(X)_\gamma$ the completion of $\mathcal{B}(X)$ with respect to $\gamma$. A mapping $F:X\to Y$ is called a $\gamma$-measurable linear operator if there exists a linear mapping $F_0:X\to Y$, measurable with respect to the pair of $\sigma$-fields $(\mathcal{B}(X)_\gamma, \mathcal{B}(Y))$ and such that $F=F_0$ $\gamma$-a.e. in $X$. We say that $F_0$ is a linear version of $F$.
\end{defn}

\begin{remark} \label{opmisoss}  Let $X$ and $Y$ be two separable Hilbert spaces, let $\gamma$ be a Gaussian measure on X and let $H_\gamma$ be the Cameron-Martin space of $\gamma$.
\begin{enumerate}[\rm (i)]
\item If $F,G: X\rightarrow Y$ are $\gamma$-measurable linear operators such that $F_0= G_0$ on $H_\gamma$, then $F=G$ $\gamma$-a.e. in $X$ (see e.g. \cite[Proposition  4(b)]{fey_del_1994}).
\item If  $F:X\rightarrow Y$ is  a $\gamma$-measurable linear operator, than there exists e Borel linear space $V\subseteq X$ such that $\gamma(V)=1$ and $(F_0)_{|V}$ is Borel measurable. We say that $(F_0)_{|V}$ is a Borel version of $F$ and we denote it by $\widetilde{F}$ (see e.g \cite[Corollary 3 (a)]{fey_del_1994}). Another proof  is \cite[Proposition 1]{urb_1975} for measurable functionals that can be generalized to the case of measurable linear operators considering a suitable version of the Lusin Theorem (see e.g. \cite[Theorem 7.1.13 Volume II]{bog_2007}). We remark that if $F_0$ is not continuous then $V$ is strictly contained in $X$ (see for instance \cite[Theorem 9.10]{kec_1995}).
\end{enumerate}
For more details on measurable linear operator we refer to Appendix \ref{mis_op_app}.
\end{remark}

The operator $T_\infty^{\mu,\nu}$ enjoys nice properties, which we investigate deeper. 
\begin{lemma}
\label{lemma:prop_L_infty3}
Let $\mu,\nu$ be two centered Gaussian measures on $X$ with Cameron-Martin spaces $H_\mu $ and $H_\nu$, respectively, and let $T\in\mathscr L(H_\mu,H_\nu)$.
\begin{enumerate}[\rm (i)]
\item $T_\infty^{\nu,\mu}$ is a $\mu$-measurable linear operator on $X$ and there exists a linear subspace $E$ of $X$ such that $\mu(E)=1$ and
\begin{align}
\label{pointiwise_limit_L_infty2}
T_\infty^{\nu,\mu}(x)=\lim_{m\to\infty}T_{\infty,m}^{\nu,\mu}(x), \qquad x\in E.    
\end{align}
In particular, $T_\infty^{\nu,\mu}(x)\in \overline {H_\nu}$ for $\mu$-a.e. $x\in X$.
\item Any linear version of $T_\infty^{\nu,\mu}$  satisfies $T_\infty^{\nu,\mu}h=Th$ for every $h\in H_\mu$. In particular, $T_\infty^{\nu,\mu}h\in H_\nu$ for every $h\in H_\mu$ and, if there exists a bounded linear extension of $T_{\infty}^{\nu,\mu}$ to the whole $X$, then such an extension is unique on $\overline {H_\mu}$. 
\item There exists a linear version of $T_\infty^{\nu,\mu}$ which belongs to $\mathscr{L}(X)$ if and only if $T$, a priori defined on $H_\mu$, extends to a linear bounded operator on $\overline {H_\mu}$ into $\overline {H_\nu}$. In this case, if we still denote by $T$ its bounded extension on $\overline{H_\mu}$, then the operator $\widetilde T\in\mathscr L(X)$, defined as
\begin{align}\label{estensione_continua}
\widetilde Tx=\begin{cases}
Tx, & x\in \overline {H_\mu}, \\
0, & x\in {\rm  Ker}(Q_\mu^\frac12)=\overline{H_\mu}^\perp,
\end{cases}    
\end{align}
is a version of $T_\infty^{\nu,\mu}$. In particular, if $L\in\mathscr L(X)$ and $L_{|H_\mu}\in\mathscr L(H_\mu,H_\nu)$, then $L$ is a version of $(L_{|H_\mu})_\infty^{\nu,\mu}$.
\end{enumerate}
\end{lemma}
\begin{proof}
\leavevmode
\begin{enumerate}[\rm (i)]
\item From Proposition \ref{prop:op_L_infty2}, there exists a set $A\subseteq X$ of $\mu$-full measure such that (up to a subsequence)
\begin{align}
\label{pointwise_conv_L_infty}
T_\infty^{\nu,\mu}(x)=\lim_{m\to\infty}T_{\infty,m}^{\nu,\mu}(x), \qquad x\in A.    
\end{align}
Since $T_{\infty,m}^{\mu,\nu}$ belongs to $\mathscr L(X)$ for every $m\in\N$, it follows that, for every $k\in\N$, every $x_1,\ldots,x_k\in A$ and every $\alpha_1,\ldots\alpha_k\in\R$,  \eqref{pointwise_conv_L_infty} implies
\begin{align*}
\lim_{m\to\infty} T_{\infty,m}^{\mu,\nu}\left(\sum_{j=1}^k\alpha_j x_j\right)
= & \sum_{j=1}^k\lim_{m\to\infty}\alpha_jT_{\infty,m}^{\mu,\nu}(x_j)= \sum_{j=1}^k\alpha_j T_\infty^{\mu,\nu}(x_j).
\end{align*}
Hence, we can set
\begin{align*}
T_\infty^{\nu,\mu}\left(\sum_{j=1}^k\alpha_jx_j\right):=\sum_{j=1}^k\alpha_j T_\infty^{\nu,\mu}(x_j)=\lim_{m\to\infty} T_{\infty,m}^{\mu,\nu}\left(\sum_{j=1}^k\alpha_j x_j\right), 
\end{align*}
and \eqref{pointiwise_limit_L_infty2} holds for all $x\in E={\rm span}(A)$.
Taking a Hamel basis of $X$, we extend such a version of $T_\infty^{\nu,\mu}$ to the whole $X$ and so there exists a version of $T_\infty^{\nu,\mu}$ which is a linear operator on $X$.


\item Let us consider a version $\widetilde{T_\infty^{\nu,\mu}}$ of $T_\infty^{\nu,\mu}$ which is linear on $X$ and let $h\in H_\mu$. Hence, there exists a subset $B\subseteq X$ of $\mu$ full measure such that
\begin{align*}
\widetilde{T_\infty^{\nu,\mu}}(x)=\lim_{m\to\infty}T_{\infty,m}^{\nu,\mu}(x), \qquad x\in B.    
\end{align*}

Arguing as in (i) and recalling that $\widetilde{T_\infty^{\nu,\mu}}$ is linear on $X$, it follows that $\{T_{\infty,m}^{\mu,\nu}(x)\}_{m\in\N}$ converges to $\widetilde{T_\infty^{\nu,\mu}}(x)$ in $X$ as $m\to\infty$ for every $x\in E={\rm span}(B)$. Since $\mu(E)=1$, we infer that $E$ contains $H_\mu$ (see \cite[Proposition 2.4.7]{bog_1998}) and by \eqref{Linftym} we get
\begin{align*}
\widetilde{T_\infty^{\nu,\mu}}(h)=\lim_{m\to\infty}T_{\infty,m}^{\mu,\nu}(h)
= \lim_{m\to\infty}\sum_{n=1}^\infty \sum_{k=1}^m \langle e_k^\mu,h\rangle_{H_\mu}\langle e_n^\nu,Te_k^\mu\rangle_{H_\nu} e_n^\nu
= \lim_{m\to\infty}\sum_{k=1}^m \langle e_k^\mu,h\rangle_{H_\mu}Te_k^\mu=Th.
\end{align*}

\item Let us consider a linear version of $T_\infty^{\nu,\mu}$ which is bounded on $X$. Since by (ii) this version coincides with $T$ on $H_\mu$, then $T_\infty^{\nu,\mu}$ maps $H_\mu$ into $H_\nu$. By continuity, $T_\infty^{\nu,\mu}$ maps $\overline{H_\mu}$ into $\overline{H_\nu}$. 

Conversely, if $T$ extends to a linear bounded operator from $H_\mu$ on $\overline{H_\mu}$ into $\overline{H_\nu}$, then we consider the operator $\widetilde T$ given by \eqref{estensione_continua}.
Since $\widetilde T$ is a linear bounded operator on $X$, it belongs to $L^2(X,\gamma;X)$. Let us show that $\widetilde T$ is a version of $T_\infty^{\nu,\mu}$. Given $m\in\N$ and $x\in\overline {H_\mu}$, by the definition of $\pi_m$ we infer that $(\pi_mx)_{m\in\N}$ converges to $x$ in $X$ as $m\to\infty$. If $x\in{\rm Ker}(Q_\mu^\frac12)$ then $\pi_mx=0$ for every $m\in\N$ . From \eqref{Linftym} and the fact that $T$ extends to a bounded linear operator on $\overline{H_\mu}$, we infer that $(T\pi_mx)_{m\in\N}$ converges to $\widetilde Tx$ in $X$ as $m\to\infty$ for every $x\in X$. From the dominated convergence theorem and \eqref{Linftym}, we get the converse implication. 

To conclude, if $L\in\mathscr L(X)$ satisfies $L_{|H_\mu}\in\mathscr L(H_\mu,H_\nu)$ and we set $T=L_{|H_\mu}$, then $T$ admits $L$ as bounded linear extension on $X$. Since $(L\pi_mx=T\pi_mx=T_{\infty,m}^{\nu,\mu})_{m\in\N}$ converges to $Lx$ for every $x\in \overline {H_\mu}$ and $\overline{H_\mu}^\perp$ is a set of null $\mu$-measure, from (i) it follows that $L$ is a version of $T_\infty^{\nu,\mu}$.
\end{enumerate}
\end{proof}

Hereafter, for every $T\in\mathscr L(H_\mu,H_\nu)$ we always consider a linear version of $T_\infty^{\nu,\mu}$. Moreover, if $L$
admits a linear bounded extension to $\overline {H_\mu}$, we consider the version of $T_\infty^{\nu,\mu}$ given by $\widetilde T$ (see Lemma \ref{lemma:prop_L_infty3}(iii)).

\begin{remark}
From now on, if $\mu=\nu$ we simply write $T_\infty^{\mu}$ instead of $T_\infty^{\mu,\mu}$.  
\end{remark}

Let $T\in\mathscr L(H_\mu,H_\nu)$ with $\|T\|_{\mathscr L(H_\mu,H_\nu)}\leq 1$. The operator $(I_\mu-T^\star T)^{\frac{1}{2}}$, where $I_\mu$ denotes the identity operator on $H_\mu$, belongs to $\mathscr L(H_\mu)$, is self-adjoint, non-negative and has operator norm less than or equal to $1$. When no confusion may arise, we simply write $I$ instead of $I_\mu$.
\begin{defn}
We denote by $\mathscr E(X)$ the set of measurable functions $f:X\to \R$ with exponential growth, i.e., $f\in \mathscr E(x)$ if and only if there exist $M,c>0$ such that $|f(x)|\leq Me^{c\|x\|_X}$ for every $x\in X$.     
\end{defn}

\begin{defn} 
\label{def:op_gamma}
Let $\mu,\nu$ be two centered Gaussian measures on $X$ with Cameron-Martin spaces $H_\mu$ and $H_\nu$, respectively. If $T\in\mathscr L(H_\mu,H_\nu)$ and $\norm{T}_{\mathscr L(H_\mu,H_\nu)}\leq 1$, then for every $f\in \mathscr E(X)$ we set
\begin{align}
\label{gammamunu}
(\Gamma_{\nu,\mu}(T)f)(x)
= & \int_Xf\left[(T^\star )_\infty^{\mu,\nu}x+\left((I-T^\star T)^{\frac{1}{2}}\right)_\infty^{\mu}y\right]\,\mu(dy), \qquad  x\in X.
\end{align}
Formula \eqref{gammamunu} is a generalization of the integral formula given for instance \cite[Proposition 1]{cho_gol_1996} and \cite[Theorem 10]{fey_del_1994} and, if $\mu=\nu$, then $T\mapsto \Gamma_\mu(T):=\Gamma_{\mu,\mu}(T)$ coincides with the second quantization operator on $L^2(X,\mu;X)$ $($see \cite{cho_gol_1996}, with $L$ replaced by $\widetilde T:=Q_{\nu}^{-\frac12}LQ_{\mu}^{\frac12}\in\mathscr L(X))$.
\end{defn}

Notice that $\Gamma_{\nu,\mu}(T)f$ is well-defined for every $f\in\mathscr E(X)$. Indeed, setting $z=((I-T^\star T)^{\frac{1}{2}})_\infty^{\mu}y$, by Proposition \ref{prop:op_L_infty2} we obtain
\begin{align}\label{gammabuona}
(\Gamma_{\nu,\mu}(T)f)(x)
= & \int_Xf\left[(T^\star )_\infty^{\mu,\nu}x+z\right]\,\gamma_{\mu, S}(dz), \qquad  x\in X,
\end{align}
where 
\begin{equation}
\label{gammanumuL}
\gamma_{\mu, S}=\mathcal{N}\left(0,Q_\mu^{\frac{1}{2}}(\widehat S)^\star \widehat S Q^{\frac{1}{2}}_\mu\right), \qquad \widehat S=Q_\mu^{-\frac12}(I-T^\star T)^{\frac12}Q_\mu^{\frac12}\in\mathscr L(X).
\end{equation}
Therefore, if $|f(x)|\leq Me^{c\|x\|_X}$ for every $x\in X$ and some $M,c>0$, then (recalling that we consider a linear version of $(T^\star)_\infty^{\mu,\nu}$)
\begin{align*}
|(\Gamma_{\nu,\mu}(T)f)(x)|
\leq Me^{c\|(T^\star)_\infty^{\mu,\nu}x\|_X}\int_Xe^{c\|z\|_X}\gamma_{\mu,S}(dz)<\infty, \qquad x\in X.
\end{align*}
Moreover, the change of variables $\eta=\left(T^\star\right)^{\mu,\nu}_\infty x$ yields to
\begin{align*}
\int_X\int_Xf((T^\star )_\infty^{\mu,\nu} x+z)\gamma_{\mu,S}(dz)\nu(dx)
= \int_X\int_Xf(\eta+z)\gamma_{\mu,S}(dz) \gamma_{\mu,T^\star}(dx),
\end{align*}
where  
$$ \gamma_{\mu,T^\star}=\mathcal{N}\left(0,Q_\mu^{\frac{1}{2}}(\widehat{T})^\star\widehat{ T}Q_\mu^{\frac{1}{2}}\right), \qquad \widehat {T}=Q_\nu^{-\frac12}TQ_\mu^{\frac12}.$$
Setting $\xi=\eta+z$, by Proposition \ref{corDecompositionHilbert}(ii) it follows that $\gamma_{\mu,T}*\widetilde \gamma_{\mu,T}$ is the Gaussian measure with covariance operator
\begin{align*}
Q=Q_\mu^{\frac{1}{2}}(\widehat S)^\star \widehat S Q^{\frac{1}{2}}_\mu+Q_\mu^{\frac{1}{2}}(\widehat T)^\star \widehat T Q^{\frac{1}{2}}_\mu.  
\end{align*}
Notice that, for every $x,y\in X$, we get
\begin{align}
\label{conto_somma_cov}
\langle Qx,y\rangle_X
= & \langle \widehat SQ_\mu^\frac12x,\widehat SQ_\mu^\frac12y\rangle_X
+ \langle \widehat {T}Q_\mu^\frac12x, \widehat{T}Q_\mu^\frac12y\rangle_X \notag \\
= & \langle (I-T^\star T)^{\frac12}Q_\mu x,(I-T^\star T)^\frac12Q_\mu y\rangle_{H_\mu}
+ \langle TQ_\mu x,TQ_\mu y\rangle_{H_\nu} \notag \\
= & \langle Q_\mu x,Q_\mu y\rangle_{H_\mu}-\langle T^\star TQ_\mu x,Q_\mu y\rangle_{H_\mu}+\langle T^\star TQ_\mu x,Q_\mu y\rangle_{H_\mu} \notag \\
= & \langle Q_\mu x,y\rangle_X,
\end{align}
which means that $Q=Q_\mu$. It follows that, for every $f\in\mathscr E(X)$ and $x\in X$,
\begin{align}
\label{super_change}
\int_X(\Gamma_{\nu,\mu}(T)f)(x)\nu(dx)
= \int_X\int_Xf[(T^\star )_\infty^{\mu,\nu} x+z]\,\gamma_{\mu,S}(dz)\nu(dx)=\int_X f(x)\mu(dx).
\end{align}

\begin{remark}
\label{rmk:esempi_second_quant_12} We compute here the second quantization operator for two special operators $T$. 
\leavevmode
\begin{itemize}
\item [\rm(i)] If $T=0$, then $(T^\star)_\infty^{\mu,\nu}=0$ and $((I-T^\star T)^{\frac{1}{2}})_\infty^{\mu}=I-P_\mu$. Hence,
\begin{align*}
(\Gamma_{\nu, \mu}(0)f)(x)= \int_Xf\left((I-P_\mu)y\right)\,\mu(dy)
= \int_Xf(y)\mu(dy) \qquad  x\in X, \ f\in \mathscr E(X),
\end{align*}
i.e., $\Gamma_{\nu, \mu}(0)f$ is constant for every $\ x\in X$ and such a constant coincides with the mean of $f$ on $X$ with respect to $\mu$.
\item [\rm(ii)] If $\nu=\mu$ and $T=cI$ with $c\in(0,1]$, then from Lemma \ref{lemma:prop_L_infty3}(iii) we get that $cI$ is a version of $(T^\star)_\infty^\mu=T_\infty^{\mu}$ and $\sqrt{1-c^2}I$ is a version of $((I-T^\star T)^\frac{1}{2})_\infty^{\mu}$. Therefore,
\begin{align*}
(\Gamma_{\mu}(cI)f)(x)
= \int_Xf\left(cx+\sqrt{1-c^2}\,y\right)\mu(dy) \qquad x\in X, \ f\in \mathscr E(X),
\end{align*}
so that, if $c=1$, then we get $(\Gamma_{\mu}(I)f)(x)=f(x)$ for  every $x\in X$ while, if $c\in(0,1)$, then
\begin{align*}
(\Gamma_{\mu}(cI)f)(x)    
=\int_X f\left(e^{-t}x+\sqrt{1-e^{-2t}}\,y\right)\mu(dy) \qquad  x\in X, \  f\in \mathscr E(X),
\end{align*}
with $c=e^{-t}$. It follows that $\Gamma_{\mu}(e^{-t})f=(T^\mu_t f)$ where $T^\mu_t$ is the classical Ornstein-Uhlenbeck semigroup in $X$ with respect to the Gaussian measure $\mu$, i.e.,
\begin{align}
\label{classical_OU}
T_t^\mu f(x)
= & \int_Xf\left(e^{-t}x+\sqrt{1-e^{-2t}}y\right)\mu(dy) \qquad x\in X, \  f\in \mathscr E(X).
\end{align}
\end{itemize}    
\end{remark}

\begin{remark}
We consider the space $\mathscr E(X)$ to define $\Gamma_{\nu,\mu}(T)$ in order to have an explicit expression of $\Gamma_{\nu,\mu}(T)f$ when $f$ is a generalized Hermite polynomial.
\end{remark}

\begin{proposition}
\label{prop_L_infty}
Let $\mu,\nu$ be two centered Gaussian measures on $X$ with Cameron-Martin spaces $H_\mu$ and $H_\nu$, respectively. If $T\in\mathscr L(H_\mu,H_\nu)$ and $\norm{T}_{\mathscr L(H_\mu,H_\nu)}\leq 1$, then for every $p\in[1,\infty)$ the operator $\Gamma_{\nu,\mu}(T)$ extends to a linear continuous operator from $L^p(X,\mu)$ into $L^p(X,\nu)$ with unitary norm.
\end{proposition}

\begin{proof}
Let $p\in[1,\infty)$. If $T=0$, then 
\begin{align*}
(\Gamma_{\nu, \mu}(0)f)(x)= \int_Xf\left((I-P_\mu)y\right)\mu(dy)
= \int_Xf(y)\mu(dy) \qquad \mu\textrm{-a.e.}\  x\in X, \ f\in \mathscr E(X),
   \end{align*}
and the statement follows from the density of $\mathscr E(X)$ in $L^p(X)$.

Let $T\in\mathscr L(H_\mu,H_\nu)\setminus\{0\}$ be such that $\norm{T}_{\mathscr L(H_\mu,H_\nu)}\leq 1$. For all $f\in C_b(X)$, by \eqref{super_change} we get
\begin{align*}
\int_X|(\Gamma_{\nu,\mu}(T)f)(x)|^p\nu(dx)
\leq & \int_X\int_X|f(y)|^p\mu(dy).
\end{align*}
It follows that $\|\Gamma_{\nu,\mu}(T)f\|_{L^p(X,\nu)}\leq \|f\|_{L^p(X,\mu)}$ for every $f\in C_b(X)$.
The statement follows from the density of $C_b(X)$ in $L^p(X,\mu)$. Taking $f \equiv 1$, we infer that $\Gamma_{\nu,\mu}(T)f\equiv 1 $, so that the norm of the extension is $1$.
\end{proof}

For every $p\in[1,\infty)$, we denote by $\Gamma^{(p)}_{\nu,\mu}(T)$ the extension of $\Gamma_{\nu,\mu}(T)$  to $L^p(X,\mu)$. Since such operators coincide on $C_b(X)$, they are consistent, i.e., for every  $p<q$ and $f\in L^q(X, \mu)$, it holds that $\Gamma^{(p)}_{\nu,\mu}(T)=\Gamma^{(q)}_{\nu,\mu}(T)$. For this reason, we omit the index $p$ if no confusion may arise, and we simply denote such extensions by $\Gamma_{\nu,\mu}(T)$.

Now we prove that, given $T\in\mathscr L(H_\mu,H_\nu)$ with $\|T\|_{\mathscr L(H_\mu,H_\nu)}\leq 1$, the operators $\widetilde \Gamma_{\nu,\mu}(T)$ and $\Gamma_{\nu,\mu}(T)$ actually coincide.

\begin{proposition}
\label{prop:uguaglianza_gamma_gamma_tilde}
For every contraction $T\in\mathscr L(H_\mu,H_\nu)$, we have $\widetilde \Gamma_{\nu,\mu}(T)=\Gamma_{\nu,\mu}(T)$ on $L^2(X,\mu)$.   
\end{proposition}
\begin{proof}
In view of Lemma \ref{lem:density_exp_funct}, it is enough to show that $\Gamma_{\nu,\mu}(T)(f)=\widetilde{\Gamma}_{\nu,\mu}(T)(f)$ for every $f\in\mathcal E_\mu$. Let $h\in H_\mu$ and let $z=Q_\mu^{-\frac12}h$. We set $\eta=Q_\nu^{-\frac12}TQ_\mu^{\frac12}h$ and $\zeta=Q_\mu^{-\frac12}(I-T^\star T)^{\frac12}Q_\mu^{\frac 12} h$. By Proposition \ref{prop:conto_gamma_tilde_exp_fct} we infer that
\begin{align*}
(\widetilde \Gamma_{\nu,\mu}(T)E_h^\mu)(x)
= & 
E_{\eta}^\nu(x)
= \exp\Big(W^\nu_\eta(x)-\frac12\|\eta\|_{X}^2\Big), \qquad \nu\textrm{-a.e. }x\in X.
\end{align*}
On the other hand, since $ E_h^\mu\in \mathscr E(X)$, it follows that, for $\nu$-a.e. $x\in X$,
\begin{align*}
(\Gamma_{\nu,\mu}(T)E_h^\mu)(x)
= & \exp\Big(\langle z,(T^\star)_\infty^{\mu,\nu}x\rangle_X-\frac12\|h\|_{X}^2\Big)\int_X\exp\big(\langle z,((I-T^\star T)^{\frac12})_\infty^\mu y\rangle_X\big)\mu(dy) \\
= & \exp\Big(W^\nu_\eta(x)-\frac12\|h\|_{X}^2\Big)\int_X\exp\big(W^\mu_{\zeta}(y)\big)\mu(dy),
\end{align*}
where we have exploited \eqref{gauss_rand_var} twice. Recalling that $W^\mu_\zeta$ is a centered Gaussian random variable with variance $\sigma^2=\langle (I-T^\star T)Q_\mu^{\frac12} h,Q_\mu^{\frac12} h\rangle_{H_\mu}$ (see Proposition \ref{prop:cameron-martin}), it follows that
\begin{align*}
\int_X\exp\big(W^\mu_\zeta(y)\big)\mu(dy)
= & \frac{1}{\sqrt{2\pi\sigma^2}}\int_{\R}e^\xi e^{-\frac{\xi^2}{2\sigma^2}}d\xi
= \exp\Big(\frac12\sigma^2\Big) \\
= & \exp\left(\frac12\|h\|_{X}^2-\frac12\langle TQ_\mu^{\frac12} h,TQ_\mu^\frac12 h\rangle_{H_\nu} \right) \\
= & \exp\left(\frac12\|h\|_{X}^2-\frac12\|Q_\nu^{-\frac12}TQ_\mu^\frac12 h\|_{X}^2 \right),
\end{align*}
which gives the thesis since, for $\nu$-a.e. $x\in X$, 
\begin{align*}
\Gamma_{\nu,\mu}(T)(E_h^\mu)(x)
= & \exp\left(W^\nu_\eta(x)-\frac12\|h\|_{X}^2\right)  \exp\left(\frac12\|h\|_X^2-\frac12\|\eta\|_{X}^2 \right)
= \widetilde \Gamma_{\nu,\mu}(T)(E_h^\mu)(x).
\end{align*}
\end{proof}

Hereafter, we always write $\Gamma(T)$ instead of $\widetilde \Gamma(T)$.
\begin{proposition}
\label{prop:hyp_contr_gen2}
If $T\in\mathscr L(H_\mu,H_\nu)$ with $\norm{T}_{\mathscr L(H_\mu,H_\nu)}< 1$, then $\Gamma_{\mu,\nu}(T)$ is a contraction from $L^p(X,\mu)$ into $L^q(X,\nu)$, for every $q\leq q_0:= 1+(p-1)\|T\|_{\mathscr L(H_\mu,H_\nu)}^{-2}$. 
Moreover, if $q>q_0$ then there exists $f\in L^p(X,\mu)$ such that $\Gamma_{\nu,\mu}(T)f\notin L^q(X,\nu)$.
\end{proposition}
\begin{remark}
The above statement implies that, if $\|T\|_{\mathscr L(H_\mu,H_\nu)}=1$, then $q=p$.     
\end{remark}
\begin{proof}
Let $f\in L^p(X,\mu)$. We consider the operator $M=T\|T\|_{\mathscr L(H_\mu,H_\nu)}^{-1}$ which belongs to $\mathscr L(H_\mu,H_\nu)$ and $\norm{M}_{\mathscr L(H_\mu,H_\nu)}=1$. From Propositions \ref{prop:prop_second_quant_serie}(iii), with $\sigma=\mu$, and \ref{prop:uguaglianza_gamma_gamma_tilde}, we infer that
\begin{align*}
\Gamma_{\nu,\mu}(T)f
= & \Gamma_{\nu,\mu}(M)(\Gamma_{\mu}(\|T\|_{\mathscr L(H_\mu,H_\nu)}I)f.
\end{align*}
By Remark \ref{rmk:esempi_second_quant_12}(ii), we have $\Gamma_{\mu}(\|T\|_{\mathscr L(H_\mu,H_\nu)}I)f=T_t^\mu f$ with $t=-\log(\|T\|_{\mathscr L(H_\mu,H_\nu)})$. It is well-known (see e.g \cite[Theorem 5.5.3]{bog_1998}, \cite{nel_1973}) that $T_t^\mu$ is a contraction operator from $L^p(X,\mu)$ into $L^q(X,\mu)$ with $q=1+(p-1)e^{2t}$. Recalling that $\Gamma_{\nu,\mu}(M)$ is a contraction in $\mathscr L(L^q(X,\mu),L^q(X,\nu))$ for every $q\in[1,\infty)$ and choosing $q\leq1+(p-1)\|T\|_{\mathscr L(H_\mu,H_\nu)}^{-2}$, we get
\begin{align*}
\|\Gamma_{\nu,\mu}(T)f\|_{L^q(X,\nu)}=&\|\Gamma_{\nu,\mu}(M)[\Gamma_{\mu}(\|T\|_{\mathscr L(H_\mu,H_\nu)}I)f]\|_{L^q(X,\nu)}
\leq \|T_t^\mu f\|_{L^q(X,\mu)}
\leq  \|f\|_{L^p(X,\mu)},
\end{align*}
where $t=-\log(\|T\|_{\mathscr L(H_\mu,H_\nu)})$. The arbitrariness of $f\in L^p(X,\mu)$ gives the first part of the statement.

To prove the optimality of $q_0$, let us consider $q>1+(p-1)\|T\|_{\mathscr L(H_\mu,H_\nu)}^{-2}$ and the function $f(x)=\exp(\alpha (\langle Q_\mu^{-\frac12} h,x\rangle_X)^2)$, $x\in X$, for some $h\in H_\mu$ and $\alpha>0$ to be properly chosen. We assume that $\alpha p<\frac{1}{2\|h\|_{H_\mu}^2}$, which implies that $f\in L^p(X,\mu)$.
It follows that, from \eqref{gauss_rand_var} and setting $\sigma_h:=\|(I-T^\star T)^\frac12 h\|_{H_\mu}$, the change of variables $\xi=\langle Q_\mu^{-\frac12}h,((I-T^*T)^{\frac12})_\infty^\mu y\rangle_X$ gives
\begin{align*}
& (\Gamma_{\nu,\mu}(T)f)(x) \\
= & \int_X\exp\left[\alpha\Bigl(\langle Q_\mu^{-\frac12}h,(T^\star )_\infty^{\mu,\nu}x\rangle_X+\langle Q_\mu^{-\frac12}h,((I-T^\star T)^{\frac12})_\infty^\mu y\rangle_X\Bigr)^2 \right]\mu(dy) \\
= & \frac{1}{\sqrt{2\sigma_h^2 \pi}}\int_{\R}\exp\left[\alpha\Bigl(\langle Q_\mu^{-\frac12}h,(T^\star )_\infty^{\mu,\nu}x\rangle_X+\xi\Bigr)^2 \right]\exp\left(-\frac{\xi^2}{2\sigma_h^2} \right)d\xi \\
= & \frac{1}{\sqrt{2\sigma_h^2\pi}}\int_{\R}\exp\left\{\alpha\left[\left(1-\frac{1}{2\alpha\sigma_h^2}\right)\xi^2+2\xi\langle Q_\mu^{-\frac12}h,(T^\star )_\infty^{\mu,\nu}x\rangle_X+(\langle Q_\mu^{-\frac12}h,(T^\star )_\infty^{\mu,\nu}x\rangle_X)^2\right]\right\}d\xi \\
= & \exp\left\{\alpha\left[\langle Q_\mu^{-\frac12}h,(T^\star )_\infty^{\mu,\nu}x\rangle_X)^2\left(1+\frac{1}{\frac{1}{2\alpha\sigma_h^2}-1}\right)\right]\right\} \\
& \times\frac{1}{\sqrt{2\sigma_h^2 \pi}}\int_{\R} \exp\left\{-\left[\left(\frac{1}{2\alpha\sigma_h^2}-1\right)^\frac12\xi-\left(\frac{1}{\frac{1}{2\alpha\sigma_h^2}-1}\right)^{-\frac12}\langle Q_\mu^{-\frac12}h,(T^\star )_\infty^{\mu,\nu}x\rangle_X\right]^2\right\}d\xi \\
= & \exp\left[\alpha\left(\langle Q_\mu^{-\frac12}h,(T^\star )_\infty^{\mu,\nu}x\rangle_X\right)^2\frac{1}{1-2\alpha\sigma_h^2}\right] \\ & \times\frac{1}{\sqrt{2\sigma_h^2 \pi}}\int_{\R} \exp\left\{-\left[\left(\frac{2\alpha\sigma_h^2}{1-2\alpha\sigma_h^2}\right)^{-\frac12}\xi-\left(\frac{1}{\frac{1}{2\alpha\sigma_h^2}-1}\right)^{-\frac12}\langle Q_\mu^{-\frac12}h,(T^\star )_\infty^{\mu,\nu}x\rangle_X\right]^2\right\}d\xi \\
= &\sqrt{\frac{\alpha}{1-2\alpha\sigma_h^2}}\exp\left[\frac{\alpha}{1-2\alpha\sigma_h^2}\left(\langle Q_\mu^{-\frac12}h,(T^\star )_\infty^{\mu,\nu}x\rangle_X\right)^2\right],
\end{align*}
provided that $$2\alpha<\sigma_h^{-2}=\|(I-T^\star T)^\frac12 h\|_{H_\mu}^{-2}=\frac{1}{\|h\|_{H_\mu}^2-\|Th\|_{H_\nu}^2}.$$ 
This condition is fulfilled, since we have $$2\alpha\leq 2\alpha p<\frac{1}{\|h\|_{H_\mu}^2}\leq \frac{1}{\|h\|_{H_\mu}^2-\|Th\|_{H_\nu}^2}.$$
Hence, by applying the change of variables $\eta=\langle Q_\mu^{-\frac12}h,(T^*)_\infty^{\mu,\nu}x\rangle_X$ and taking \eqref{gauss_rand_var} into account, we deduce that
\begin{align}
\label{ottim_hyper}
\int_X|(\Gamma_{\nu,\mu}(T)f)(x)|^q\nu(dx)  = & \sqrt{\frac{\alpha}{2\|Th\|_{H_\nu}\pi\,(1-2\alpha\sigma_h^2)}} \int_\R\exp\left(\frac{\alpha q}{1-2\alpha\sigma_h^2}\eta^2-\frac{\eta^2}{2\|Th\|_{H_\nu}^2}\right)d\eta
\end{align}
and the integral in the right-hand side of \eqref{ottim_hyper} s not finite if and only of $\frac{\alpha q}{1-2\alpha\sigma_h^2}-\frac{1}{2\|Th\|_{H_\nu}^2}\geq0$, which is equivalent to
\begin{align}
\label{dis_ott_1}
& 2\alpha q\|Th\|_{H_\nu}^2\geq 1-2\alpha\|(I-T^\star T)^\frac12 h\|_{H_\mu}^2=1-2\alpha\|h\|_{H_\mu}^2+2\alpha\|Th\|_{H_\nu}^2 \notag \\
& \Longleftrightarrow 2\alpha\|Th\|_{H_\nu}^2(q-1)\geq 1-2\alpha\|h\|_{H_\mu}^2. 
\end{align}
Recall that $q>1+(p-1)\|T\|_{\mathscr L(H_\mu,H_\nu)}^{-2}$. For every $\varepsilon\in\left(0,(q-1)\|T\|_{\mathscr L(H_\mu,H_\nu)}^2-(p-1)\right)$, let $h_\varepsilon\in H_\mu$ be such that $\|Th_\varepsilon\|_{H_\nu}\geq \sqrt{\frac{(p+\varepsilon-1)}{(q-1)}}\|h_\varepsilon\|_{H_\mu}$. Hence, a sufficient condition for \eqref{dis_ott_1} to hold is
\begin{align*}
2\alpha(p+\varepsilon-1)\|h_\varepsilon\|_{H_\mu}^2\geq 1-2\alpha\|h_\varepsilon\|_{H_\mu}^2 \Longleftrightarrow 2\alpha (p+\varepsilon)\|h_\varepsilon\|_{H_\mu}^2\geq 1.    
\end{align*}
Given $\varepsilon$ as above, conditions
\begin{align*}
\alpha p<\frac{1}{2\|h_\varepsilon\|_{H_\mu}^2}, \qquad \alpha(p+\varepsilon)\geq \frac{1}{2\|h_\varepsilon\|_{H_\mu}^2}    
\end{align*}
are satisfied for any $\alpha\in\left[\frac{1}{2(p+\varepsilon)\|h_\varepsilon\|_{H_\mu}^2},\frac{1}{2p\|h_\varepsilon\|_{H_\mu}^2}\right)$ and so with this choice of $\varepsilon,h_\varepsilon$ and $\alpha$ we infer that the function $f$ belongs to $ L^p(X,\mu)$ but $\Gamma_{\nu,\mu}(T)f\notin L^q(X,\nu)$.
\end{proof}

We conclude this section by providing sufficient (and necessary) conditions that ensure compactness and Hilbert-Schmidt properties of the operator $\Gamma_{\nu,\mu}(T)$. We start with some technical lemmas.

\begin{lemma}
\label{lemma:eigenvector_eigenvalue}
Let $\mu=\mathcal{N}(0,Q_\mu)$ be a centered Gaussian measure on $X$ with Cameron-Martin space $H_\mu$ and let $T\in\mathscr L(H_\mu)$ be a self-adjoint contraction with a complete set of orthonormal eigenvectors $\{v_k:k\in\N\}$ and corresponding sequence of eigenvectors $\{t_k\}_{k\in\N}$. Then $\Gamma_\mu(T)$ is self-adjoint with orthonormal set of eigenvectors
\begin{equation*}
\Bigl\{\psi^\mu_{\alpha,T}:=\sqrt{\alpha!}\prod_{j=1}^\infty \phi_{\alpha_j}\Bigl(W^\mu_{Q_\mu^{-\frac 12}v_j}\Bigr):\alpha\in\Lambda\Bigr\}
\end{equation*}
and corresponding sequence of eigenvalues 
\begin{equation*}
\Bigl\{t_{\alpha}:=\prod_{j=1}^{\infty}t_j^{\alpha_j}:\alpha\in\Lambda\Bigr\},
\end{equation*}
with the convention that $0^0=1$.
\end{lemma}
\begin{proof}
The statement follows from Lemma \ref{lemma:isomorfismo}, with the choice of $\{v_k:k\in\N\}$ as orthonormal basis of $H_\mu$, and Definitions \ref{def:op_gamma_tilden} and \ref{def:op_gamma_tilde}. Indeed, for every $n\in\N$ and $\alpha\in\Lambda_n$, we have
\begin{align*}
\Gamma_{\mu}(T)\psi_{\alpha,T}^\mu
= & \Gamma_{\mu,\mu,n}(T)\psi_{\alpha,T}^\mu
= \Psi_n^\mu\left(\frac{\sqrt{\alpha!}}{n!}\sum_{\sigma\in\Sigma_n^\alpha}Tv_{i^\alpha_{\sigma(1)}}^\mu\otimes \cdots \otimes Tv_{i^\alpha_{\sigma(n)}}^\mu\right) \\
= & \Psi_n^\mu\left(\prod_{j=1}^n t_j^{\alpha_j}\frac{\sqrt{\alpha!}}{n!}\sum_{\sigma\in\Sigma_n^\alpha}v_{i^\alpha_{\sigma(1)}}^\mu\otimes \cdots \otimes v_{i^\alpha_{\sigma(n)}}^\mu\right)
= \prod_{j=1}^n t_j^{\alpha_j}\psi_{\alpha,T}^\mu.
\end{align*}
\end{proof}
\begin{lemma}
\label{lemma:azione_gamma_white_noise}
Let $\mu$, $\nu$ be two centered Gaussian measures on $X$ with Cameron-Martin spaces $H_\mu$ and $H_\nu$, respectively, and let $T\in\mathscr L(H_\mu,H_\nu)$ with $\|T\|_{\mathscr L(H_\mu,H_\nu)}\leq 1$. For every $h\in \overline {H_\mu}$, we get
\begin{align*}
\Gamma_{\nu,\mu}(T)W_{h}^\mu=W^\nu_{\widehat Th}    
\end{align*}
\end{lemma}
\begin{proof}
Let $h\in H_\mu$ with $z=Q_\mu^{-\frac12}h$. From \eqref{gauss_rand_var}, with $(T^\star )_{\infty}^{\mu,\nu}$ instead of $T_\infty^{\nu,\mu}$, it follows that
\begin{align}
\label{white_noise_modificato}
\langle (T^\star )_\infty^{\mu,\nu}\cdot,Q_\mu^{-\frac12}h\rangle_X=\langle (T^\star )_\infty^{\mu,\nu}\cdot,z\rangle_X
=W^\nu_{\widehat T Q_\mu^\frac12 z}(\cdot)
=W^\nu_{\widehat T h}(\cdot)
\end{align}
is a real centered Gaussian random variable with variance $\norm{L Q_{\mu}z}_{H_\nu}^2$. \\
Let $h\in \overline {H_\mu}$ and let $(h_n)_{n\in\N}\subseteq H_\mu$ converge to $h$ in $X$. If we set $z_n=Q_\mu^{-\frac12}h_n$ for every $n\in\N$, then from Proposition \ref{prop:cameron-martin} it follows that, for every $n,m\in\N$,
\begin{align*}
\left\|W^\nu_{\widehat Th_n }-W^\nu_{\widehat Th_m }\right\|_{L^2(X,\nu)}^2
= & \|TQ_\mu^\frac12(h_n-h_m)\|_{H_\nu}^2
\leq \|T\|_{\mathscr L(H_\mu,H_\nu)}^2\|Q_\mu^\frac12(h_n-h_m)\|_{H_\mu}^2 \\
= & 
\|T\|_{\mathscr L(H_\mu,H_\nu)}^2\|h_n-h_m\|_{X}^2,
\end{align*}
which implies that $\left(W^\nu_{\widehat T h_n }\right)_{n\in\N}$ is a Cauchy sequence in $L^2(X,\nu)$ and its limit is $W_{\widehat Th}$. From \eqref{white_noise_modificato}, with $h_n$ and $z_n$ instead of $h$ and $z$, respectively, we deduce that
\begin{align*}
W_{\widehat Th}^\nu=L^2(X,\nu)-\lim_{n\to\infty}\langle (T^\star )_\infty^{\mu,\nu}\cdot, z_n\rangle_X
\end{align*}
and the linearity of $W_{h_n}^{\mu}=\langle \cdot,z_n\rangle_X$, $n\in\N$, gives
\begin{align*}
\Gamma_{\nu,\mu}(T)W_{h}^\mu
= & \lim_{n\to\infty}\Gamma_{\nu,\mu}(T)W_{h_n}^\mu
=  \lim_{n\to\infty}\int_X \langle (T^\star )_\infty^{\mu,\nu}\cdot+((I-T^\star T)^{\frac12})_\infty^\mu y,z_n\rangle_X \mu(dy) \\
= & \lim_{n\to\infty}\Big(\langle (T^\star )_\infty^{\mu,\nu}\cdot,z_n\rangle_X+\int_X \langle ((I-T^\star T)^{\frac12})_\infty^\mu y,z_n\rangle \mu(dy)\Big) \\
= & \lim_{n\to\infty}\langle (T^\star )_\infty^{\mu,\nu}\cdot,z_n\rangle_X=W_{\widehat Th}^\nu,
\end{align*}
where all the limits are meant in $L^2(X,\nu)$.
\end{proof}
\begin{lemma}
\label{lemma:esist_C}
Let $T\in\mathscr L(H_\mu,H_\nu)$ and set $B=(T^\star T)^{\frac12}\in\mathscr L(H_\mu)$. There exists an operator $C\in\mathscr L(H_\mu,H_\nu)$ such that $T=CB$ and $\|C\|_{\mathscr L(H_\mu,H_\nu)}=1$. Further, $C_{|\overline{{\rm Range}(B)}}$ is an isometry on $\overline{{\rm Range}(B)}$, which implies that $C^\star T=C^\star CB=B$ on $H_\mu$.    
\end{lemma}
\begin{proof}
For every $h\in H_\mu$, we set $C(Bh)=Th$. It follows that
\begin{align}
\label{ext_C}
\|C(Bh)\|_{H_\nu}^2
= & \|Th\|_{H_{\nu}}^2=\|Bh\|_{H_\mu}^2.
\end{align}
Hence, we can extend $C$ to the whole $H_\mu=\overline{{\rm Range}(B)}+{\rm Ker}(B)$ by setting $C\equiv0$ on ${\rm Ker}(B)$ and by means of \eqref{ext_C} on $\overline{{\rm Range}(B)}$, which also gives $\|C\|_{\mathscr L(H_\mu,H_\nu)}=1$. \\
From \eqref{ext_C}, for every $x,y\in H_\mu$ we deduce
\begin{align*}
\|Bx\|_{H_\mu}^2+2\langle Bx,By\rangle_{H_\mu}+\|By\|_{H_\mu}^2
= & \langle B(x+y),B(x+y)\rangle_{H_{\mu}}
= \langle CB(x+y),CB(x+y)\rangle_{H_{\nu}} \\
= & \|CBx\|_{H_\nu}^2+2\langle CBx,CBy\rangle_{H_\nu}+\|CBy\|_{H_\nu}^2\\
= & \|Bx\|_{H_\mu}^2+2\langle CBx,CBy\rangle_{H_\nu}+\|By\|_{H_\mu}^2,
\end{align*}
which gives $\langle Bx,By\rangle_{H_\mu}=\langle CBx,CBy\rangle_{H_\mu}$ for every $x,y\in H_\nu$. By density, we get $\langle h,k\rangle_{H_\mu}=\langle Ch,Ck\rangle_{H_\nu}$ for every $h,k\in \overline{{\rm Range}(B)}$ and this implies that $C_{|\overline{{\rm Range}B}}$ is an isometry on $\overline{{\rm Range}B}$. \\
Let us prove that $C^\star T=B$ on $H_\mu$. Fix $x\in H_\mu$. For every $y\in H_\mu$, we get
\begin{align*}
\langle C^\star Tx,y\rangle_{H_\mu}
= & \langle C^\star CB,y\rangle_{H_\mu}
= \langle CBx,Cy\rangle_{H_\nu}
= \langle CBx,Cy_1\rangle_{H_\nu}
= \langle Bx,y_1\rangle_{H_\mu}
= \langle Bx,y\rangle_{H_\mu},
\end{align*}
where $y=y_1+y_2$ with $y_1\in\overline{{\rm Range}(B)}$, $y_2\in{\rm Ker}(B)$ and we have used the fact that $Cy=Cy_1$ and $\langle Bx,y_1\rangle_{H_\mu}=\langle Bx,y\rangle_{H_\mu}$.
\end{proof}

\begin{proposition}\label{compattezza}
Let $\mu,\nu$ be two centered Gaussian measure on $X$ and let $T\in\mathscr L(H_\mu,H_\nu)$, $T\neq0$, be a contraction. 
\begin{enumerate}[\rm(i)]
\item If $T$ is a compact strict contraction, then the operator $\Gamma_{\nu,\mu}(T)$ is compact from $L^p(X,\mu)$ into $L^q(X,\nu)$ for every $p\in(1,\infty)$ and $1\leq q<1+(p-1)\|T\|_{\mathscr L(H_\mu,H_\nu)}^{-2}$.
\item If $\Gamma_{\nu,\mu}(T):L^p(X,\mu)\to L^1(X,\nu)$ is compact for some $p\in(1,\infty)$, then $T$ is compact. 
\item If $\Gamma_{\nu,\mu}(T)$ is compact from $L^2(X,\mu)$ into $L^2(X,\nu)$, then $T$ is a compact strict contraction.
\end{enumerate}
\end{proposition}
\begin{remark}
We notice that, if $T=0$, then $\Gamma_{\nu,\mu}(T)f=\int_Xfd\mu$ for every $f\in L^p(X,\mu)$ and $p\in[1,\infty)$. Hence, $\Gamma(T):L^1(X,\mu)\to L^\infty(X,\nu)$ is compact.
\end{remark}
\begin{proof}
Along the proof, we keep the same notation as that in Lemma \ref{lemma:esist_C}. 

Assume that $T$ is a compact strict contraction from $H_\mu$ into $H_\nu$. We claim that also $B=(T^\star T)^{\frac12}$ is a compact strict contraction on $H_\mu$. From \eqref{ext_C}, we deduce that $B$ is a strict contraction. Let now $(h_n)_{n\in\N}\subseteq H_\mu$ be a bounded sequence. Then there exists a subsequence $(h_{k_n})_{n\in\N}\subseteq (h_n)_{n\in\N}$ such that $(Th_{k_n})_{n\in\N}\subseteq H_\nu$ is a convergent subsequence. By \eqref{ext_C} we get
\begin{align*}
\|Bh_{k_n}-Bh_{k_m}\|_{H_\mu}^2=\|Th_{k_n}-Th_{k_m}\|_{H_\nu}^2, \qquad k,m\in\N,    
\end{align*}
which shows that $(Bh_{k_n})_{n\in\N}\subseteq H_\mu$ is a Cauchy sequence in $H_\mu$. The claim is proved. \\
By \eqref{composizione_completa}, we get
\begin{align}
\label{spezzamento_compattezza}
\Gamma_{\nu,\mu}(T)=\Gamma_{\nu,\mu}(C)\Gamma_{\mu}(B).    
\end{align}
From Remark \ref{rmk:eq_operatore_sec_quant_choj}, the operator $\Gamma_\mu(B)$ coincides with the second quantization operator defined in \cite{cho_gol_1996} with $\widehat B=Q_\mu^{-\frac12}BQ_{\mu}^{\frac12}\in\mathscr L(X)$ instead of $B$. Therefore, in order to apply \cite[Proposition  2(a)]{cho_gol_1996}, it is enough to show that $\widehat B$ is a compact strict contraction on $X$. The thesis follows by combining this fact and \eqref{spezzamento_compattezza}. Let $(x_n)_{n\in\N}\subseteq X$ be a bounded sequence. For all $n\in\N$ we set $h_n:=Q_\mu^{\frac12}x_n$ and $(h_n)_{n\in\N}\subseteq H_\mu$ is a bounded sequence in $H_\mu$. Since $B$ is compact, it follows that there exists a subsequence $(h_{k_n})_{n\in\N}\subseteq (h_n)_{n\in\N}$ such that $(Bh_{k_n})_{n\in\N}\subseteq H_\mu$ converges to some $h\in H_\mu$ as $n$ tends to infinity. It follows that
\begin{align*}
\|\widehat Bx_{k_n}-\widehat Bx_{k_m}\|_X^2
= & \|Bh_{k_n}-Bh_{k_m}\|_{H_\mu}^2, \qquad n,m\in\N,
\end{align*}
which implies that $(\widehat Bx_{k_n})_{n\in\N}\subseteq X$ is a Cauchy sequence. Further, there exists $c\in[0,1)$ such that, for every $x\in X$, we have $\|\widetilde Bx\|_X^2=\|BQ_{\mu}^\frac12x\|_{H_\mu}^2\leq c\|Q_{\mu}^{\frac12}x\|_{H_\mu}^2\leq c\|x\|_{X}^2$. \\ 
Assume by contradiction that $\Gamma_{\nu,\mu}(T):L^p(X,\mu)\to L^q(X,\nu)$, $q=1+(p-1)\|T\|_{\mathscr L(H_\mu,H_\nu)}^{-2}$, is compact, let $h\in H_\mu$ be a unit vector such that $Bh=\|B\|_{\mathscr L(H_\mu)}h$ and let $z=Q_\mu^{-\frac12}h$. For every $n\in\N$, we set
\begin{align*}
f_n:=\exp\left(W^\mu_{nz}-\frac p2n^2\right)=\exp\left(\frac{1-p}{2}n^2 \right)E^\mu_{nz}.    
\end{align*}
It follows that $\|f_n\|_{L^p(X,\mu)}=1$ for every $n\in\N$ and, from Proposition \ref{prop:conto_gamma_tilde_exp_fct}, \eqref{ext_C} and the fact that $Th=CBh$ and $Bh=\|B\|_{\mathscr L(H_\mu)}h$, we infer that
\begin{align*}
\Gamma_{\nu,\mu}(T)f_n
= & \exp\left(\frac{1-p}2n^2\right)E^\nu_{nQ_\nu^{-\frac12}Th}
= \exp\left(\frac{1-p}2n^2\right) E^\nu_{n\|B\|_{\mathscr L(H_\mu)}Q_\nu^{-\frac12}Ch}, \qquad n\in\N
\end{align*}
and $
\|Ch\|_{H_\nu}=\|CBh\|_{H_\nu}\|B\|_{\mathscr L(H_\mu)}^{-1}=\|Bh\|_{H_\mu}\|B\|_{\mathscr L(H_\mu)}^{-1}=\|h\|_{H_\mu}=1$. It follows that
\begin{align*}
\|\Gamma_{\nu,\mu}(T)f_n\|_{L^q(X,\nu)}
= \exp\left(\frac{1-p}2n^2\right)\exp\left(\frac{q-1}2n^2\|B\|_{\mathscr L(H_\mu)}^2\right)
=1, \qquad n\in \N, 
\end{align*}
where we have used the fact that $\|B\|_{\mathscr L(H_\mu)}=\|T\|_{\mathscr L(H_\mu,H_\nu)}$. The compactness of $\Gamma_{\nu,\mu}(T)$ from $L^p(X,\mu)$ into $L^q(X,\nu)$ implies that there exists a subsequence $(\Gamma_{\nu,\mu}(T)f_{k_n})$ which converges to some $f\in L^q(X,\nu)$ with $\|f\|_{L^q(X,\nu)}=1$. However, for every $y\in X$, for $\nu$-a.e. $x\in X$ we have $|W_{y}^{\nu}(x)|<\infty$, which implies that
\begin{align*}
(\Gamma_{\nu,\mu}(T)f_n)(x)
= \exp\left(\frac{1-p}2n^2\right) E^\nu_{n\|B\|_{\mathscr L(H_\mu)}Q_\nu^{-\frac12}Ch}(x)\to 0, \qquad n\to \infty, 
\end{align*}
for $\nu$-a.e. $x\in X$. This gives a contradiction and concludes the proof of (i).

To prove (ii), assume that $\Gamma_{\nu,\mu}(T):L^p(X,\mu)\to L^1(X,\nu)$ is compact and let $(h_n)_{n\in\N}\subseteq H_\mu$ be a bounded sequence. We have to show that $(Th_n)_{n\in\N}$ admits a converging subsequence. For every $n\in\N$, we set $z_n=Q_\mu^{-\frac12}h_n$ and notice that $(W_{z_n}^\mu)$ is bounded in $L^r(X,\mu)$ for every $r\in(1,\infty)$, since $W_{z_n}$ is a centered Gaussian random variable in $(X,\mu)$ with variance $\|z_n\|_X^2$ for every $n\in\N$. Claiming that 
\begin{equation}\label{uguaglianza_compattezza}
\norm{T(h_n-h_m)}_{H_\nu}
=\sqrt{\frac \pi2}
\norm{\Gamma_{\nu,\mu}(T)(W_{z_n}^\mu-W_{z_m}^\mu)}_{L^1(X,\nu)}, \qquad n,m\in\N,
\end{equation}
we obtain the compactness of $T$.

Now we prove \eqref{uguaglianza_compattezza}. From Lemma \ref{lemma:azione_gamma_white_noise}, with $h=z_n-z_m$, it follows that
\begin{align*}
\norm{\Gamma_{\nu,\mu}(T)(W_{z_n}^\mu-W_{z_m}^\mu)}_{L^1(X,\nu)}
=& \norm{\Gamma_{\nu,\mu}(W_{z_n-z_m}^\mu)}_{L^1(X,\nu)} 
=  \norm{W^\nu_{\widehat T(z_n-z_m)}}_{L^1(X,\nu)} \\
= & \frac{1}{\sqrt{2\pi}\norm{T(h_n-h_m)}_{H_\nu}}\int_\R\abs{\xi}e^{-\frac{1}{2}\xi^2\norm{T(h_n-h_m)}^{-2}_{H_\nu}}\,d\xi\\
= & \sqrt{\frac 2\pi}\norm{T(h_n-h_m)}_{H_\nu}, \qquad n,m\in\N,
\end{align*}
which gives the claim. 

It remains to prove (iii). Let us notice that, since $B=C^\star T$, from \eqref{composizione_completa} it follows that $\Gamma_{\mu}(B)=\Gamma_{\mu,\nu}(C^\star )\Gamma_{\nu,\mu}(T)$. The compactness of $\Gamma_{\nu,\mu}(T)$ from $L^2(X,\mu)$ into $L^2(X,\nu)$ and the continuity of $\Gamma_{\mu,\nu}(C^\star )$ from $L^2(X,\nu)$ into $L^2(X,\mu)$ imply that $\Gamma_{\mu}(B)$ is compact from $L^2(X,\mu)$ into itself and, from (ii), we already know that $T$ is compact. \\
Assume by contradiction that $1=\|T\|_{\mathscr L(H_\mu,H_\nu)}=\|B\|_{\mathscr L(H_\mu)}$. 
Computations in (i) show that the compactness of $T$ implies the compactness of $B$. So, let $v$ be an eigenvector of $B$ with corresponding eigenvalue $1$. From Lemma \ref{lemma:eigenvector_eigenvalue}, with $T=B$, we deduce that
\begin{align*}
f_n:=\sqrt{n!}\phi_n\left(W^\mu_{Q_\mu^{-\frac12}v}\right), \qquad n\in\N,    
\end{align*}
is a sequence of orthonormal eigenvectors of $\Gamma_{\mu}(B)$ in $L^2(X,\mu)$ with same corresponding eigenvalue $1$. Hence, the sequence $(\Gamma_\mu(B)f_n=f_n)$ does not admit any convergent subsequence, contradicting the compactness of $\Gamma_\mu(B)$ in $L^2(X,\mu)$. 
\end{proof}
\begin{remark}
By means of Lemma \ref{lemma:esist_C}, it is possible give a different proof of the fact that $\Gamma_{\nu,\mu}(T)$ is not compact from $L^p(X,\mu)$ into $L^q(X,\nu)$, $q=1+(p-1)\|T\|_{\mathscr L(H_\mu,H_\nu)}^{-2}$, and that the compactness of $\Gamma_{\nu,\mu}(T)$ from $L^p(X,\mu)$ into $L^1(X,\nu)$ implies the compactness of $T$. \\
Assume that $\Gamma_{\nu,\mu}(T)$ is compact from $L^p(X,\mu)$ into $L^q(X,\nu)$. The equality $\Gamma_{\mu}(B)=\Gamma_{\mu,\nu}(C^\star )\Gamma_{\nu,\mu}(T)$ (see the proof of (iii) in Proposition \ref{compattezza}) implies that $\Gamma_{\mu}(B)$ is compact from $L^p(X,\mu)$ into $L^q(X,\mu)$, but this contradicts \cite[Proposition 2]{cho_gol_1996}, since $\Gamma_\mu(B)$ is the second quantization operator defined in \cite{cho_gol_1996} with operator $\widehat B=Q_\mu^{-\frac12}BQ_\mu^{\frac12}$ (see Remark \ref{rmk:eq_operatore_sec_quant_choj}), which is a compact strict contraction as proved in the proof of (i) in Proposition \ref{compattezza}. \\
Analogously, if $\Gamma_{\nu,\mu}(T)$ is compact from $L^p(X,\mu)$ into $L^1(X,\nu)$, so does $\Gamma_\mu(B)$. Hence, arguing as in the proof of \cite[Proposition 2]{cho_gol_1996}, we infer that $\widehat B=Q_\mu^{-\frac12}BQ_\mu^{\frac12}$ is compact. Same computations as above reveal that this is equivalent to the fact that $B$ is compact, and the same holds for $T$.
\end{remark}

\begin{proposition}\label{Hilbert-Schmidt}
Let $\mu,\nu$ be two centered Gaussian measure on $X$ and let $T\in\mathscr L(H_\mu,H_\nu)$ be a contraction. $T$ is a strict contraction of {\it Hilbert-Schmidt} type if and only if $\Gamma_{\nu,\mu}(T)$ is of Hilbert-Schmidt type from $L^2(X,\mu)$ into $L^2(X,\nu)$ and in this case
\begin{equation*}
\norm{\Gamma_{\nu,\mu}(T)}_{\mathscr{L}_2(L^2(X,\mu),L^2(X,\nu))}=
\prod_{k=1}^\infty \frac{1}{1-t_k^2},
\end{equation*}
where $\{t_k\}_{k\in\N}$ are the eigenvalues of $T^\star T$.
\end{proposition}
\begin{proof}
If $T$ is a strict contraction of Hilbert-Schmidt type, then  $B=(T^\star T)^{\frac12}$ does. From \eqref{spezzamento_compattezza} and \cite[Proposition 2(b)]{cho_gol_1996}, we infer that $\Gamma_{\nu,\mu}(T)$ is of Hilbert-Schmidt type from $L^2(X,\mu)$ to $L^2(X,\nu)$. \\
Conversely, if $\Gamma_{\nu,\mu}(T)$ is of Hilbert-Schmidt type, then the same holds for $\Gamma_{\nu,\mu,1}(T)$. From the definition of $\Gamma_{\nu,\mu}$, it follows that \eqref{composizione_completa} holds true if we replace $\Gamma_{\nu,\mu}$, $\Gamma_{\nu,\sigma}$ and $\Gamma_{\sigma,\mu}$ with $\Gamma_{\nu,\mu,1}$, $\Gamma_{\nu,\sigma,1}$ and $\Gamma_{\sigma,\mu,1}$, respectively. From the identity $C^\star T=B$, it follows that $\Gamma_{\mu}(B)$ is of Hilbert-Schmidt type and \cite[Proposition 2]{cho_gol_1996} gives that $\widehat B=Q_{\mu}^-{\frac12}BQ_\mu^{\frac12}$ is a strict contraction of Hilbert-Schmidt type in $X$. Since this is equivalent to the fact that $B$ is of Hilbert-Schmidt type in $H_\mu$, the proof is complete.
\end{proof}

\section{Applications to the Ornstein-Uhlenbeck evolution operator}\label{risultati_OU}

In this section we give a representation formula for the Ornstein-Uhlenbeck evolution operator $P_{s,t}$, $s<t$, in terms of the operator $\Gamma_{\gamma_s,\gamma_t}$ where $\gamma_t$ is given by \eqref{gammat} for all $t\in\R$. 

Let $\{U(t,s)\}_{s\leq t}$ be an evolution operator in $X$ and let $\{B(t)\}_{t\in\R}$ be a strongly continuous family of linear bounded operators on $X$. The Ornstein-Uhlenbeck evolution operator on $B_b(X)$ is defined by
\begin{align}
 P_{t,t}&=I,\qquad \  t\in\R, \notag\\
P_{s,t}\varphi(x)&=\int_X\varphi(y)\,\mathcal{N}_{U(t,s)x ,Q(t,s)}(dy),\quad s<t,\ \varphi\in B_b(X),\ x\in X,
\label{OUnof}
\end{align}
where $\mathcal{N}_{U(t,s)x ,Q(t,s)}$ is the Gaussian measure in $X$ with mean $U(t,s) x$
and covariance operator
\begin{align*} 
Q(t,s)=\int_s^tU(t,r) Q(r) U(t,r)^\star\,dr,\qquad Q(r):=B(r)B(r)^\star, \qquad s<t, \ r\in\R.
\end{align*}
For every $s<t$, we denote by $\mathcal H_{t,s}$ the Cameron-Martin space of the Gaussian measure $\mathcal{N}_{0 ,Q(t,s)}$. \\
To study $\{P_{s,t}\}_{s\leq t}$ on suitable $L^p$-spaces, we need the notion of evolution system of measures.
\begin{defn}
An evolution system of measures for $\{P_{s,t}\}_{s\leq t}$ is a family of Borel probability measures $\{\nu_r\}_{r\in \R}$ on $X$ such that 
\begin{equation*}
\int_XP_{s,t}\varphi(x)\nu_s(dx)=\int_X\varphi(x)\nu_t(dx),\ \ s\leq t,\ \varphi\in C_b(X).
\end{equation*}
\end{defn}
We assume the following on the operators which appear in the definition of $\{P_{s,t}\}_{s\leq t}$.
\begin{hyp} \label{4} 
\leavevmode
\begin{enumerate}[\rm(i)]
\item $\{U(t,s)\}_{s\leq t}\subseteq\mathscr{L}(X)$ is a strongly continuous evolution operator, namely, for every $x\in X$ the map
\begin{equation*}
(s,t)\mapsto U(t,s)x\in X
\end{equation*}
is continuous and
\begin{enumerate}
\item $U(t,t)=I$ for every $t\in\R$,
\item $U(t,r)U(r,s)=U(t,s)$ for $s\leq r\leq t$.
\end{enumerate}
\item $\{B(t)\}_{t\in\R}\subseteq\mathscr{L}(X)$ is a bounded and strongly measurable family of linear and bounded operators, namely,
\begin{enumerate}[{\rm (a)}]
\item there exists $K>0$ such that 
\begin{equation*}
\sup_{t\in\R}\norm{B(t)}_{\mathscr{L}(X)}\leq K,
\end{equation*}
\item the map 
\begin{equation*}
t\in\R\mapsto B(t)x\in X
\end{equation*}
is bounded and measurable for every $x\in X$.
\end{enumerate}
\item For every $s<t$, the operator $Q(t,s):X\rightarrow X$ given by
\begin{align*} 
Q(t,s)&=\int_s^tU(t,r)B(r)B(r)^\star U(t,r)^\star \,dr,
\end{align*}
has finite trace.
\item For every $t\in\R$,
\begin{equation*}
\sup_{s<t}\left[{\rm Trace}[Q(t,s)]\right]<\infty
\end{equation*}
and, for every $t\in\R$ and $x\in X$,
\begin{equation*}
\lim_{s\rightarrow-\infty}U(t,s)x=0.
\end{equation*}
\end{enumerate}
\end{hyp}

Under Hypothesis \ref{4}, $\{P_{s,t}\}_{\overline\Delta}$ is well-defined and by \cite[Theorem 4 and Remark 12]{big_def_2024} there exists a unique evolution system of measures $\{\gamma_t\}_{t\in\R}$ for $P_{s,t}$ such that 
\begin{equation}\label{gammat}
\gamma_t:=\mathcal{N}_{0,Q(t,-\infty)},\qquad  t\in\R.
\end{equation}
Moreover, for every $f\in C_b(X)$ we have
\begin{equation*}
\lim_{s\rightarrow-\infty}P_{s,t}f(x)=m_t(f):=\int_Xf(y) \gamma_t(dy), \qquad t\in\R,\; x\in X.
\end{equation*}
For every $p\in[1,\infty)$ and $s<t$, the operator $P_{s,t}$ extends to a linear bounded operator from $L^p(X,\gamma_t)$ into $L^p(X,\gamma_s)$ with unitary norm (see \cite[Lemma 4]{big_def_2024}). We still denote by $P_{s,t}$ such an extension and we denote by $\mathcal H_{t,s}$ the Cameron-Martin space of $\mathcal N(0, Q(t,s))$. If $s=-\infty$, then we simply write $\mathcal{H}_t$ instead of $\mathcal H_{t,-\infty}$.

We need a preliminary result.
\begin{proposition}
\label{tesoretto} Let Hypothesis \ref{4} be verified. The following statements hold for every $s<t$.
\begin{enumerate}[{\rm(i)}]
\item $\mathcal H_{t,s}\subseteq \mathcal H_t$ with continuous embedding, the norm of the embedding is smaller than or equal to $1$ and ${\rm Ker}(Q(t,-\infty)^{\frac{1}{2}})\subseteq{\rm Ker}(Q(t,s)^{\frac{1}{2}})\subseteq{{\rm Ker}(Q(t))}$.
\item $U(t,s)$ maps $\mathcal H_s$ into $\mathcal H_t$ and 
\begin{equation*}
\norm{U(t,s)}_{\mathscr{L}(\mathcal H_s,\mathcal H_t)}=\norm{Q^{-\frac{1}{2}}(t,-\infty)U(t,s)Q(s,-\infty)^{\frac{1}{2}}}_{\mathscr{L}(X)}\leq 1.
\end{equation*}
\item $\mathcal H_t\subseteq \mathcal H_{t,s}$ if and only if 
\begin{equation}\label{normaL2}
\norm{U(t,s)}_{\mathscr{L}(\mathcal H_s,\mathcal H_t)}=\norm{Q(t,-\infty)^{-\frac{1}{2}}U(t,s)Q(s,-\infty)^{\frac{1}{2}}}_{\mathscr{L}(X)}<1.
\end{equation}
\end{enumerate}
\end{proposition}
\begin{proof}
The main tool to prove this result is Proposition \ref{pseudo} with $X_1=X_2=X$. Fix $s<t$.

(i) Let $x\in X$. We have
\begin{align*}
\norm{Q(t,s)^{\frac{1}{2}}x}_X^2=\int_s^t\norm{Q^{\frac{1}{2}}(r)U(t,r)^\star x}_X^2dr\leq\int_{-\infty}^t\norm{Q^{\frac{1}{2}}(r)U(t,r)^\star x}_X^2dr=\left\|Q(t,-\infty)^{\frac12}x\right\|_X^2.
\end{align*}
The thesis follows by applying Proposition \ref{pseudo} with $L_1=Q(t,s)^{\frac{1}{2}}$ and $L_2=Q(t,-\infty)^{\frac{1}{2}}$. The inclusion 
${\rm Ker}(Q(t,s)^{\frac{1}{2}})\subseteq{{\rm Ker}(Q(t))}$ is proven in \cite[Lemma 3.16]{def_2022}.

(ii) We apply again Proposition \ref{pseudo} with $L_1=U(t,s)Q(s,-\infty)^{\frac{1}{2}}$ and $L_2=Q(t,-\infty)^{\frac{1}{2}}$.
For every $x\in X$, we have 
\begin{align}
\|L_1^\star x\|_X^2
= & \norm{Q(s,-\infty)^{\frac{1}{2}}U(t,s)^\star x}_X^2
= \int_{-\infty}^s\ix{U(t,r)Q(r)U(t,r)^\star x}{x}\,dr\nonumber \\
= & \ix{Q(t,-\infty)x}{x}-\ix{Q(t,s)x}{x}
=  \norm{Q(t,-\infty)^{\frac{1}{2}}x}_X^2-\norm{Q(t,s)^{\frac{1}{2}}x}_X^2 \notag \\
\leq & \norm{Q(t,-\infty)^{\frac{1}{2}}x}_X=\|L_2^\star x\|_X^2,
\label{uguaglianzaQtinf}
\end{align}
and also statement (ii) follows.

(iii) By Proposition \ref{pseudo}, with $L_1=Q(s,-\infty)^{\frac12}U(t,s)^\star$ and $L_2=Q(t,-\infty)^{\frac12}$, inequality \eqref{normaL2} holds true if and only if there exists $\alpha\in(0,1)$ such that
\begin{align}\label{stimaalpha}
\norm{Q(s,-\infty)^{\frac{1}{2}}U(t,s)^\star x}_X\leq\alpha\norm{Q(t,-\infty)^{\frac12}x}_X,\ \ x\in X.
\end{align}
By \eqref{uguaglianzaQtinf}, for every $x\in X$  
\begin{align}
\label{uguaglianzaQtinf2}
\norm{Q(s,-\infty)^{\frac{1}{2}}U(t,s)^\star x}_X^2=\norm{Q(t,-\infty)^{\frac{1}{2}}x}_X^2-\norm{Q(t,s)^{\frac{1}{2}}x}_X^2.
\end{align}
Formula \eqref{uguaglianzaQtinf2} implies that \eqref{stimaalpha} is verified if and only if
\begin{align}\label{stimaQtsQt}
\norm{Q(t,-\infty)^{\frac{1}{2}}x}_X\leq\frac{1}{\sqrt{1-\alpha^2}}\,\norm{Q(t,s)^{\frac{1}{2}}x}_X,\ \ x\in X.
\end{align}
Still by Proposition \ref{pseudo}, \eqref{stimaQtsQt} is equivalent to $\mathcal H_t\subseteq \mathcal H_{t,s}$.
\end{proof}

Before to prove the identification of $P_{s,t}$ with $\Gamma_{\gamma_s,\gamma_t}((U(t,s)_{|\mathcal H_s})^\star)$, we need of the following result.
\begin{lemma}
\label{lemma:dualita}
For every $s<t$ and $x\in X$ we have 
\begin{align*}
(U(t,s)_{|\mathcal H_{t,s}})^\star Q(t,-\infty)x=Q(s,-\infty)U(t,s)^\star x.
\end{align*}
\end{lemma}
\begin{proof}
For every $s<t$ and $x,y\in X$, we get
\begin{align*}
\langle (U(t,s)_{|\mathcal H_{s}})^\star Q(t,-\infty)x,Q(s,-\infty)^\frac12 y\rangle _{\mathcal H_s}
= & \langle Q(t,-\infty)x,U(t,s)_{|\mathcal H_{s}}Q(s,-\infty)^\frac12 y\rangle _{\mathcal H_t} \\
= & \langle x,U(t,s)Q(s,-\infty)^\frac12 y\rangle _{X} \\
= & \langle Q(s,-\infty)^{\frac12}U(t,s)^\star x,y\rangle_X \\
= & \langle Q(s,-\infty)U(t,s)^\star x,Q(s,-\infty)^\frac12y\rangle_{\mathcal H_s},
\end{align*}
which gives the thesis from the arbitrariness of $y\in X$.
\end{proof}

\begin{proposition}
\label{prop:rapp_OU_second_q}
Assume that Hypothesis \ref{4} holds true. For every $s\leq t$ and $f\in B_b(X)$, we have 
$$P_{s,t} f=\Gamma_{\gamma_s,\gamma_t}(L)f,$$ where $L=(U(t,s)_{|\mathcal H_s})^\star$ belongs to $\mathscr L(\mathcal{H}_{t},\mathcal{H}_{s})$. As a byproduct, the extension of $P_{s,t}$ as an operator from $L^p(X,\gamma_t)$ into $L^p(X,\gamma_s)$, $p\in[1,\infty)$ and $s<t$, coincide with the extension of $\Gamma_{\gamma_s,\gamma_t}(L)$ as an operator from $L^p(X,\gamma_t)$ into $L^p(X,\gamma_s)$.
\end{proposition}
\begin{proof}
Let $s\leq t$, let $f\in B_b(X)$ and let $x\in X$. By Lemma \ref{lemma:prop_L_infty3}(iii) with $T=U(t,s)$, $\mu=\gamma_s$ and $\nu=\gamma_t$, it follows that $U(t,s)$ is a version of $(T_{|\mathcal H_s})_\infty^{\gamma_t,\gamma_s}=(L^\star)_\infty^{\gamma_t,\gamma_s}$. Formula \eqref{gammabuona}, with $\mu=\gamma_t$ and $\nu=\gamma_s$, implies that, for every $f\in B_b(X)$,
\begin{align}
 (\Gamma_{\gamma_s,\gamma_t}(L)f)(x)
= & \int_Xf((L^\star)_\infty^{\gamma_t,\gamma_s}x+z)\,\mathcal{N}(0,Q(t,-\infty)^{\frac12}(\widehat S)^\star\widehat SQ(t,-\infty)^{\frac12})(dz) \notag \\
= & \int_Xf(U(t,s)x+z)\mathcal N(0,Q(t,-\infty)^{\frac12}(\widehat S)^\star\widehat SQ(t,-\infty)^{\frac12})(dz),\qquad  x\in X,
\label{second_quant_oU}
\end{align}  
where $\widehat S=Q(t,-\infty)^{-\frac12}(I-L^\star L)^\frac12Q(t,-\infty)^{\frac12}\in\mathscr L(X)$. Let us notice that, from Lemma \ref{lemma:dualita},
\begin{align*}
& \langle Q(t,-\infty)^{\frac12}(\widehat S)^\star\widehat SQ(t,-\infty)^{\frac12}x,y\rangle_{X} \\
= &  \langle (I-L^\star L)Q(t,-\infty)x,Q(t,-\infty)y\rangle_{\mathcal H_t} \\
= & \langle Q(t,-\infty)x,y\rangle_{X}
-\langle (U(t,s)_{|\mathcal H_s})^\star Q(t,-\infty)x,(U(t,s)_{|\mathcal H_s})^\star Q(t,-\infty)y\rangle_{\mathcal H_s} \\
= & \langle Q(t,-\infty)x,y\rangle_{X}
-\langle Q(s,-\infty)U(t,s)^\star x,Q(s,-\infty)U(t,s)^\star y\rangle_{\mathcal H_s} \\
= & \langle Q(t,-\infty)x,y\rangle_{X}
-\langle Q(s,-\infty)^{\frac12}U(t,s)^\star  x,Q(s,-\infty)^{\frac12}U(t,s)^\star y\rangle_{X} \\
= & \langle Q(t,-\infty)x,y\rangle_{X}
-\langle U(t,s)Q(s,-\infty)U(t,s)^\star x,y\rangle_{X} \\
= & \langle Q(t,s)x,y\rangle_X
\end{align*}
for every $x,y\in X$. Replacing in \eqref{second_quant_oU}, we infer that
\begin{align*}
(\Gamma_{\gamma_s,\gamma_t}(L)f)(x)
= \int_Xf(U(t,s)x+z)\mathcal N(0,Q(t,s))(dz)=P_{s,t}f(x)
\end{align*}
for every $x\in X$. The last statement follows by the density of $B_b(X)$ in $L^p(X,\gamma_t)$.

\end{proof}

Applying what we have already shown for the second quantization operator, we have the following corollaries.

\begin{corollary}\label{HYPER}
If Hypothesis \ref{4} holds true and $\|U(t,s) \|_{\mathscr L(\mathcal{H}_{s},\mathcal{H}_{t})}<1$ for some $s<t$, then for every $p\geq 1$, the operator $P_{s,t}$ maps $L^p(X,\gamma_{t})$ into $L^q(X,\gamma_{s})$ for every $1\leq q\leq q_0$ and
\begin{equation*}
\norm{P_{s,t} f}_{L^q(X,\gamma_{s})}\leq\norm{f}_{L^p(X,\gamma_{t})},\ \ f\in L^p(X,\gamma_{t}),
\end{equation*} 
where $q_0 =1+(p-1)\|U(t,s)\|^{-2}_{\mathscr L(\mathcal{H}_{s},\mathcal{H}_{t})}$. Moreover, if $q>q_0$ then there exists $f\in L^p(X,\gamma_t)$ such that $P_{s,t}f\notin L^q(X,\gamma_s)$. 
\end{corollary}
\begin{proof}
The result is an immediate consequence of Propositions \ref{prop:hyp_contr_gen2} and \ref{prop:rapp_OU_second_q}.
 \end{proof}

\begin{corollary}\label{compattezzapst}
 We assume that Hypothesis \ref{4} holds true.  
\begin{enumerate}[\rm(i)]
\item If $\|U(t,s) \|_{\mathscr L(\mathcal{H}_{s},\mathcal{H}_{t})}<1$ is a compact strict contraction, then the operator $P_{s,t}$ is compact from $L^p(X,\gamma_t)$ into $L^q(X,\gamma_s)$ for every $p\in(1,\infty)$ and $1\leq q<1+(p-1)\|U(t,s) \|_{\mathscr L(\mathcal{H}_{s},\mathcal{H}_{t})}^{-2}$.
\item If $P_{s,t}:L^p(X,\gamma_t)\to L^1(X,\gamma_s)$ is compact for some $p\in(1,\infty)$, then $U(t,s):\mathcal H_s\rightarrow \mathcal H_t$ is compact. 
\item If $P_{s,t}$ is compact from $L^2(X,\gamma_t)$ into $L^2(X,\gamma_s)$, then $U(t,s):\mathcal H_s\rightarrow \mathcal H_t$ is a compact strict contraction.
\end{enumerate}
\end{corollary}
\begin{proof}
The result is an immediate consequence of Propositions \ref{compattezza} and \ref{prop:rapp_OU_second_q}.
\end{proof}

\begin{corollary}\label{Hilbert-Schmidtpst}
We assume that Hypothesis \ref{4} holds true. 
$U(t,s):\mathcal H_s\rightarrow \mathcal H_t$ is a strict contraction of Hilbert-Schmidt type if and only if $P_{s,t}$ is of Hilbert-Schmidt type from $L^2(X,\gamma_t)$ into $L^2(X,\gamma_s)$.
In this case, 
\begin{equation*}
\norm{P_{s,t}}_{\mathscr{L}_2(L^2(X,\gamma_t),L^2(X,\gamma_s))}=
\prod_{k=1}^\infty \frac{1}{1-t_k^2},
\end{equation*}
where $\{t_k\}_{k\in\N}$ are the eigenvalues of $U(t,s)_{|\mathcal H_s}(U(t,s)_{|\mathcal H_s})^\star $.
\end{corollary}
\begin{proof}
The result is an immediate consequence of Propositions \ref{Hilbert-Schmidt} and \ref{prop:rapp_OU_second_q}.
\end{proof}


Thanks to the representation by means of the second quantization operator, we are also able to study the asymptotic behavior of $P_{s,t}f$ with $f\in L^p(X,\gamma_t)$, $p\in(1,\infty)$. 

\begin{proposition}\label{comportamento_asintotico}
Assume that Hypotheses \ref{4} holds true. Then, for every $p\in(1,\infty)$ there exists a positive constant $c_p$ such that
\begin{align*}
\|P_{s,t}f-m_t(f)\|_{L^p(X,\gamma_s)}\leq c_p\|U(t,s)_{|\mathcal H_s}\|_{\mathscr L(\mathcal H_s,\mathcal H_t)}\|f\|_{L^p(X,\gamma_t)}, \qquad  f\in L^p(X,\gamma_t),\ s<t,   
\end{align*}
where, for every $t\in\R$ and $f\in L^p(X,\gamma_t)$,
\begin{align*}
m_t(f)=\int_Xfd\gamma_t.    
\end{align*}
\end{proposition}

\begin{proof}
Let $(\sigma,t)\in\Delta$ be as in the statement and let $f\in L^p(X,\gamma_t)$. We set $L=(U(t,s))^\star$, $c=\norm{L}_{\mathscr L(\mathcal{H}_{t},\mathcal{H}_{s})}$ and $\tau=-\log c$. 
From Proposition \ref{composizione0} and Remark \ref{rmk:esempi_second_quant_12}(ii), we infer that
\begin{align*}
m_t(f)= & \int_Xfd\gamma_t =  \int_X P_{s,t} f d\gamma_s =\int_X \Gamma_{\gamma_s,\gamma_t}(L) f  d\gamma_s \\
= & \int_X \Gamma_{\gamma_s,\gamma_t}\left(cc^{-1}L\right) f  d\gamma_s = \int_X \Gamma_{\gamma_s}(cI)\Gamma_{\gamma_s,\gamma_t}\left(c^{-1}L\right) f  d\gamma_s\\
&=\int_X T_{\tau}^{\gamma_s}\Gamma_{\gamma_s,\gamma_t}\left(c^{-1}L\right) f  d\gamma_s
=\int_X T_{\tau}^{\gamma_s}\Gamma_{\gamma_s,\gamma_t}\left(c^{-1}L\right) f  d\gamma_s,
\end{align*}
where $T_\tau^{\gamma_s}$ has been defined in \eqref{classical_OU}. Recalling that for every $h\in L^p(X,\gamma_s)$ we have
\begin{equation*}
\int_X T_{\tau}^{\gamma_s}h \, d\gamma_s=\int_X h \, d\gamma_s,
\end{equation*}
we obtain
\begin{equation*}
m_t(f)=\int_X \Gamma_{\gamma_s,\gamma_t}\left(c^{-1}L\right) f  d\gamma_s.
\end{equation*}
 Setting $g=\Gamma_{\gamma_s, \gamma_t}(c^{-1}L)f$, from \eqref{composizione2} we get
\begin{align}
\label{poinc_in_mu_s}
\|P_{s,t}f-m_t(f)\|_{L^p(X,\gamma_s)}
= & \|\Gamma_{\gamma_s, \gamma_t}(L)f-m_t(f)\|_{L^p(X,\gamma_s)}\notag  \\
= & \|\Gamma_{\gamma_s}(cI)g-m_s(g)\|_{L^p(X,\gamma_s)} \notag \\
= & \|T^{\gamma_s}_{\tau}g-m_s(g)\|_{L^p(X,\gamma_s)},
\end{align}
where the last equality follows from Remark \ref{rmk:esempi_second_quant_12}(ii).
It is well-known (see for instance \cite{fuh_1998, lun_pal_2020}) that for every $p\in(1,\infty)$ there exists a positive constant $c_p$ such that for every Gaussian measure $\mu$ we get
\begin{align}
\label{poinc_in_mu}
\left\|T^{\mu}_\sigma(\varphi)-m_\mu(\varphi)\right\|_{L^p(X,\mu)}\leq c_p e^{-\sigma}\|\varphi\|_{L^p(X,\mu)}
\end{align}
for every $\sigma>0$. From \eqref{poinc_in_mu_s} and \eqref{poinc_in_mu}, we infer that
\begin{align*}
\|P_{s,t}f-m_t(f)\|_{L^p(X,\gamma_s)}
\leq c_p\,c\,\|g\|_{L^p(X,\gamma_s)},
\end{align*}
which gives the thesis from Proposition \ref{prop_L_infty}.
\end{proof}
An immediate consequence of Proposition \ref{comportamento_asintotico} is the following corollary.
\begin{corollary}
\label{coro:exp_decay}
Assume that Hypothesis \ref{4} are satisfied and fix $t\in\R$. If there exist $M_t\geq 1$ and $\omega_t>0$ such that
\begin{align}
\label{expon_decay}
\|U(t,s)\|_{\mathscr L(\mathcal{H}_{s},\mathcal{H}_{t})}\leq M_te^{-\omega_t(t-s)}, \qquad s< t,    
\end{align}
then
$\|P_{s,t}f-m_t(f)\|_{L^p(X,\gamma_s)}\leq c_pM_te^{-\omega_t(t-s)}\|f\|_{L^p(X,\gamma_t)}$, for every $f\in L^p(X,\gamma_t)$ and $s<t$.   
\end{corollary}

Now we show a sufficient condition for \eqref{expon_decay} to be verified. To this aim, we set ${\rm H}_t:=Q(t)^{\frac{1}{2}}(X)$ for all $t\in\R$.

\begin{proposition}\label{prop_bignamini}
Let Hypothesis \ref{4} hold true. We fix $t\in\R$ and assume that $U(t,s)$ maps ${\rm H}_s$ into ${\rm H}_t$ for every $s<t$ and that there exist $M_t\geq 1$, $\omega_t>0$ and $\alpha_t\in[0,\frac{1}{2})$ such that 
\begin{equation}
\label{stima_norma_U(t,s)}
\norm{U(t,s)}_{\mathscr L ({\rm H}_s;{\rm H}_t)}\leq M_t\,\frac{e^{-\omega_t(t-s)}}{(t-s)^{\alpha_t}}.
\end{equation}
The following statements hold.
\begin{enumerate}[{\rm (i)}]
\item For every $s_1<s_2<t$, $\mathcal{H}_{t,s_2}=\mathcal{H}_{t,s_1}$ with equivalent norms.
\item For every $s<t$, $(U(t,s)_{|H_s})^\star Q(t)x=Q(s) U(t,s)^\star x$ for every $x\in X$. 
\item 
$\mathcal{H}_t=\mathcal{H}_{t,s}$ with equivalent norms.
\item For every $s<t$, $$\norm{U(t,s)}_{\mathscr{L}(\mathcal{H}_s;\mathcal{H}_t)}<\min\left\{1,M_t\,\frac{e^{-\omega_t(t-s)}}{(t-s)^{\alpha_t}}\right\}.$$
\end{enumerate}
In particular, the assumptions of Corollary \ref{coro:exp_decay} are satisfied.
\end{proposition}
\begin{proof}
{\rm (i) are (ii)} are analogous to \cite[Theorem 3.18]{def_2022}.
Let us prove (iii). Let $t\in\R$ be as in the and statement and let us denote $M_t,\omega_t$ and $\alpha_t$ by $M,\omega$ and $\alpha$, respectively. We know that $\mathcal{H}_{t,s}\subseteq \mathcal{H}_t$ for every $s<t$. Conversely, we use Proposition \ref{pseudo} with $L_1=Q(t,-\infty)^\frac{1}{2}$ and $L_2=Q(t,s)^\frac{1}{2}$ for $s<t$. Indeed for every $x\in X$ we have
\begin{align*}
\norm{Q(t,-\infty)^\frac{1}{2}x}^2_X&=\ps{Q(t,-\infty)x}{x}_X=\int_{-\infty}^t\norm{Q(r)^\frac{1}{2}U(t,r)^\star}^2_X\,dr\\
&=\int_{s}^t\norm{Q(r)^\frac{1}{2}U(t,r)^\star}^2_X\,dr+\int_{-\infty}^s\norm{Q(r)^\frac{1}{2}U(t,r)^\star}^2_X\,dr\\
&=\norm{Q(t,s)^\frac{1}{2}x}^2_X+\int_{-\infty}^s\norm{Q(r)^\frac{1}{2}U(t,r)^\star}^2_X\,dr.
\end{align*}
Setting $t-s+r=\sigma$, from \eqref{stima_norma_U(t,s)} and (ii) we get
\begin{align}\label{stima_per_stima_min}
\int_{-\infty}^s\norm{Q(r)^\frac{1}{2}U(t,r)^\star}^2_X\,dr&=\int_{-\infty}^t\norm{Q(\sigma+s-t)^\frac{1}{2}U(t,\sigma+s-t)^\star}^2_X\,d\sigma\nonumber\\
&=\int_{-\infty}^t\norm{Q(\sigma+s-t)U(\sigma,\sigma+s-t)^\star U(t,\sigma)^\star}^2_{{\rm H}_{\sigma+s-t}}\,d\sigma\nonumber\\
&=\int_{-\infty}^t\norm{(U(\sigma,\sigma+s-t)_{|H_{\sigma+s-t}})^\star Q(\sigma)U(t,\sigma)^\star}^2_{{\rm H}_{\sigma+s-t}}\,d\sigma\nonumber\\
&\leq\frac{M^2}{(t-s)^{2\alpha}}\,e^{-2\omega(t-s)}\int_{-\infty}^t \norm{Q(\sigma)U(t,\sigma)^\star}^2_{{\rm H}_{\sigma}}\,d\sigma\nonumber\\
&=\frac{M^2}{(t-s)^{2\alpha}}\,e^{-2\omega(t-s)}\int_{-\infty}^t \norm{Q(\sigma)^{\frac{1}{2}}U(t,\sigma)^\star}^2_{X}\,d\sigma\nonumber\\
&=\frac{M^2}{(t-s)^{2\alpha}}\,e^{-2\omega(t-s)}\norm{Q(t,-\infty)^\frac{1}{2}x}^2_X. 
\end{align}
Setting $$C_{s}=\frac{M^2}{(t-s)^{2\alpha}}\,e^{-2\omega(t-s)},$$ we choose $s_0<t$ such that $C_s<1$ for all $s<s_0$. Hence, 
\begin{equation*}
\norm{Q(t,-\infty)^\frac{1}{2}x}^2_X\leq \frac{1}{1-C_s}\norm{Q(t,s)^\frac{1}{2}x}^2_X, \qquad s<s_0<t.
\end{equation*}
By (i), (iii) follows for all $s<t$. 

Let us prove (iv). By (iii) and Proposition \ref{tesoretto}(iii) we get 
\begin{equation}\label{stima_min1}
\norm{U(t,s)}_{\mathscr L(\mathcal{H}_s;\mathcal{H}_t)}<1, \qquad s<t.
\end{equation}
Moreover, we apply Proposition \ref{pseudo} with $L_1=Q(s,-\infty)^{\frac12}U(t,s)^\star$ and $L_2=Q(t,-\infty)^{\frac12}$.
By \eqref{stima_per_stima_min}, we get
\begin{align*}
\norm{Q(s,-\infty)^{\frac12}U(t,s)^\star x}^2_X&=\int_{-\infty}^s\norm{Q(r)^\frac{1}{2}U(t,r)^\star}^2_X\,dr\leq\frac{M^2}{(t-s)^{2\alpha}}\,e^{-2\omega(t-s)}\norm{Q(t,-\infty)^\frac{1}{2}x}^2_X,
\end{align*}
and so
\begin{equation}\label{stima_min2}
\norm{U(t,s)}_{\mathscr L(\mathcal{H}_s;\mathcal{H}_t)}\leq M\,\frac{e^{-\omega(t-s)}}{(t-s)^\alpha}.
\end{equation}
By \eqref{stima_min1} and \eqref{stima_min2}, (iv) follows.
\end{proof}

\section{Examples}\label{Examples}
In this section, we adapt to our situation the examples studied in \cite{big_def_2024,cer_lun_2021,def_2022}.

\subsection{A non-autonomous parabolic problem}
Let $d\in\N$ and let $\mathcal{O}$ be a bounded open subset of $\R^d$ with smooth boundary. We consider the evolution operator $\{U(t,s)\}_{s\leq t}$ in $X:=L^2(\mathcal{O})$ associated to the evolution equation of parabolic type

\begin{align*}
\begin{cases}
u_t(t,x)=\mathcal{A}(t)u(t,\cdot)(x),\ \ (t,x)\in(s,\infty)\times\mathcal{O},\\
\mathcal{B}(t)u(t,\cdot)(x)=0,\ \ (t,x)\in(s,\infty)\times\partial\mathcal{O}.\\
u(s,x)=u_0(x),\ \ x\in\mathcal O.
\end{cases}
\end{align*}
The differential operators $\mathcal{A}(t)$ are given by
\begin{equation}\label{operatoreat2}
\mathcal{A}(t)\varphi(x)=\sum_{i,j=1}^d a_{ij}(t,x)D_{ij}^2\varphi(x)+\sum_{i=1}^d a_{i}(t,x)D_{i}\varphi(x)+a_0(t,x)\varphi(x),\quad t\in \R,\ x\in\mathcal{O},
\end{equation}
and the family of the boundary operators $\{\mathcal{B}(t)\}_{t\in \R}$ is either of Dirichlet, Neumann or Robin type, namely
\begin{align}\label{boundary}
\mathcal{B}(t)u=
\begin{cases}
\begin{aligned}
&u \ \ \quad \quad\quad  \quad\quad\quad \quad \quad \quad\quad \quad \quad \quad \mbox{(Dirichlet)},  \\[1ex]
&\displaystyle{\sum_{i,j=1}^da_{ij}(x,t)D_i u\,\nu_j}\quad \quad\quad\quad \quad \quad\, \mbox{(Neumann)},\\[1ex] 
&\displaystyle{\sum_{i,j=1}^da_{ij}(x,t)D_i u\,\nu_j}+b_0(x,t)u\quad \ \ \mbox{(Robin)},
\end{aligned}
\end{cases}
\end{align}
 where $\nu=(\nu_1,...,\nu_d)$ is the unit outer normal vector at the boundary of $\mathcal{O}$. 
 
We introduce spaces of functions depending both on time and space variables. For every $a,b\in\R\cup\{\pm\infty\}$, $a<b$, $\rho\in(0,1)$ and $k\in\N\cup\{0\}, $ we denote by $C^{\rho,k}\bigl([a,b]\times \overline{\mathcal{O}}\bigr)$ the space of all bounded continuous functions $\psi:[a,b]\times \overline{\mathcal{O}}\rightarrow \R$ such that 
\begin{align*}
\psi(t,\cdot)\in C^k(\overline{\mathcal{O}}),\quad t\in [a,b], \\
\psi(\cdot,x)\in C^\rho([a,b]),\quad x\in \overline{\mathcal{O}}.
\end{align*}
$C^{\rho,k}\bigl([a,b]\times \overline{\mathcal{O}}\bigr)$ is endowed with the norm
 $$\norm{\psi}_{C^{\rho,k}([a,b]\times \overline{\mathcal{O}})}=\sup_{x\in\overline{\mathcal{O}}}\norm{\psi(\cdot,x)}_{C^\rho([a,b])}+\sup_{t\in[a,b]}\norm{\psi(t,\cdot)}_{C^k(\overline{\mathcal{O}})}.$$

 
 We make the following assumptions.

\begin{hyp}\label{A}
We assume that $\mathcal{O}$ has smooth boundary. Moreover, we assume that $a_{ij}=a_{ji}$, $a_{ij},b_0\in C_b^{\rho,2}\left(\R\times\overline{\mathcal{O}}\right)$, $b_0\in C_b^{\rho,1}\left(\R\times\overline{\mathcal{O}}\right)$, $a_{i},a_0\in C_b^{\rho,0}\left(\R\times\overline{\mathcal{O}}\right)$ and there exist $\nu, \omega,\beta_0>0$ such that 
\begin{align*}
&\sum_{i,j=1}^da_{ij}(t,x)\xi_i\xi_j\geq \nu\abs{\xi}^2,\ \ t\in\R,\ x\in\mathcal{O},\  \xi\in\R^d,\\
&\sup_{(t,x)\in\R\times\mathcal{O}}a_0(t,x)\leq -\omega,\\
&\inf_{(r,x)\in\R\times\mathcal{O}}b_0(r,x)\geq\beta_0, \\
& \delta_0-\omega<0,
\end{align*}
where $\displaystyle{\delta_0=\frac{1}{\nu}\left(\sum_{i=1}^d\norm{a_i}^2_\infty\right)^{\frac{1}{2}}}$.
\end{hyp}

For every $t\in\R$ we denote by $A(t)$ the realization in $L^2(\mathcal{O})$ of $\mathcal{A}(t)$ with one of the boundary conditions \eqref{boundary}.  

\begin{itemize}
 \item \textbf{Dirichlet boundary condition} 
 
In this case for every $t\in\R$ we have
\begin{align*} 
D(A(t))&=H^2(\mathcal{O})\cap H^1_0(\mathcal{O}).
\end{align*} 
Moreover, by \cite[Theorems 3.5 \& 4.15]{gui_1991}, for every $t\in\R$ and $0<\gamma<1$ we have
\begin{align*}
\left(L^2(\mathcal{O}),H^2(\mathcal{O})\cap H^1_0(\mathcal{O})\right)_{\gamma,2}
=\begin{cases}
 H^{2\gamma}(\mathcal{O}), &\mbox{if}\ 0<\gamma<\frac{1}{4}, \\
 \mathring{H}^{\frac{1}{2}}(\mathcal{O}), &\mbox{if}\ \gamma=\frac{1}{4}, \\
\left\{u\in H^{2\gamma}(\mathcal{O})\ |\ u_{|_{\partial\mathcal{O}}}=0\right\}, \quad &\mbox{if}\ \frac{1}{4}<\gamma< 1.
\end{cases}
\end{align*}

 \item \textbf{Neumann or Robin boundary conditions} 
 
In this case, for every $t\in\R$ we have
\begin{align*} 
D(A(t))&=\left\{u\in H^2(\mathcal{O}):\ \mathcal{B}(t)u=0\right\}.
\end{align*} 
Moreover, by \cite[Theorems 3.5 \& 4.15]{gui_1991}, for every $t\in\R$ and $0<\gamma<1$ we have
\begin{align*}
(X,D(A(t))_{\gamma,2}&
=\begin{cases}
 H^{2\gamma}(\mathcal{O}), &\mbox{if}\ 0<\gamma<\frac{3}{4}, \\[1ex] 
 \left\{u\in H^{\frac{3}{2}}(\mathcal{O})\ |\ \tilde{\mathcal{B}}(t)u\in\mathring{H}^{\frac{1}{2}}(\mathcal{O})\right\}, &\mbox{if}\ \gamma=\frac{3}{4}, \\[1ex]
\left\{u\in H^{2\gamma}(\mathcal{O})\ |\ \mathcal{B}(t)u=0\right\},\quad &\mbox{if}\ \frac{3}{4}<\gamma<1,
\end{cases}
\end{align*}
where
\begin{align*}
\tilde{\mathcal{B}}(t)u=
\begin{cases}
\begin{aligned}
&\displaystyle{\sum_{i,j=1}^da_{ij}(x,t)D_i u\,\tilde{\nu_j}}, \quad\quad\quad \quad \quad\ \ \, \mbox{(Neumann)},\\
&\displaystyle{\sum_{i,j=1}^da_{ij}(x,t)D_i u\,\tilde{\nu_j}}+b_0(x,t)u,\ \ \ \ \mbox{(Robin)}, 
\end{aligned}
\end{cases}
\end{align*} 
$\tilde{\nu}$ is a smooth enough extension of $\nu$ to $\overline{\mathcal{O}}$ and $\mathring{H}^{\frac{1}{2}}(\mathcal{O})$ consists on all the elements $\varphi\in H^{\frac{1}{2}}(\mathcal{O})$ whose null extension outside $\overline{\mathcal{O}}$ belongs to $H^{\frac{1}{2}}(\R^d)$.
\end{itemize}

\begin{thm}\label{thm7.4}
Assume that the following conditions hold true.
\begin{enumerate}[{\rm (i)}]
\item $X:=L^2(\mathcal{O})$, where $\mathcal{O}\subseteq\R^d$ is a bounded open set with smooth boundary and $d=1,2,3,4,5$. 
\item The operators $\mathcal{A}(t)$, given by \eqref{operatoreat2}, verify Hypothesis \ref{A} for every $t\in\R$.
\item The realization $A(t)$ in $L^2(\mathcal{O})$ of $\mathcal{A}(t)$ with boundary conditions given by \eqref{boundary} is a negative operator for every $t\in\R$.
\item For every $t\in\R$ we have 
\[
B(t)=(-A(t))^{-\gamma},\quad t\in\R,
\]
for some $\gamma\geq 0$.
\end{enumerate}
In the following cases
\begin{align*}
\begin{cases}
0\leq\gamma<1,\ \ \ \ \mbox{if}\ d=1,\\[1ex]
0<\gamma<1,\ \ \ \ \mbox{if}\ d=2,\\[1ex]
 \frac{1}{4}<\gamma<1,\ \ \ \, \mbox{if}\ d=3,\\[1ex]
\frac{1}{2}<\gamma<1,\ \ \ \ \mbox{if}\ d=4,\\[1ex]
\frac{3}{4}<\gamma<1,\ \ \ \ \mbox{if}\ d=5,
\end{cases}
\end{align*}
there exists an evolution operator $U(t,s)$ satisfying  Hypothesis \ref{4} and the assumptions in Corollary \ref{HYPER} and Proposition \ref{prop_bignamini}.
\end{thm}
\begin{proof}
Let us show that Hypothesis \ref{4} is satisfied. In \cite[Examples 2.8 \& 2.9]{sch_2004} it is proven that the families $\{A(t)\}_{t\in\R}$ satisfies the assumptions of \cite{acq_1988,acq_ter_1987}. So by \cite[Theorem 2.3]{acq_1988} there exists an evolution operator $\{U(t,s)\}_{s\leq t}$ on $X$ such that Hypothesis \ref{4}(i) holds.
In \cite[Theorem 7]{big_def_2024} it is shown that the rest of Hypothesis \ref{4} is satisfied and there exists $C>0$ such that
\begin{align*}
\norm{U(t,s)}_{\mathscr L({\rm H_s};{\rm H_t})}\leq C e^{-(\omega-\delta_0)(t-s)},\ \ s<t,
\end{align*}
where $\omega$ and $\delta$ are the constant given in Hypothesis \ref{A}. Hence by Proposition \ref{prop_bignamini}(iv) we get 
\begin{align*}
\norm{U(t,s)}_{\mathscr L({\mathcal{H}_s};{\mathcal{H}_t})}< \min\left\{1,C e^{-(\omega-\delta_0)(t-s)}\right\},\ \ s<t,
\end{align*}
and the statement follows.

\end{proof}

\begin{remark} If $a_i\equiv 0$ for all $i=1,..,d$ hypothesis \textrm{(iii)} of condition \ref{thm7.4} holds. In the general case it is easy to find sufficient conditions on the coefficients $a_i$ such that assumption \textrm{(iii)} in Theorem \ref{thm7.4} still holds.
\end{remark}


\begin{thm}\label{thm7.5}
Assume the following conditions hold true.
\begin{enumerate}[{\rm(i)}]
\item $X:=L^2(\mathcal{O})$, where $\mathcal{O}\subseteq\R^d$ is a bounded open set with smooth boundary and $d=1,2,3$. 
\item The operators $\mathcal{A}(t)$, given by \eqref{operatoreat2}, verify Hypothesis \ref{A} for every $t\in\R$.
\item The realization $A(t)$ in $L^2(\mathcal{O})$ of $\mathcal{A}(t)$ with boundary conditions given by \eqref{boundary} is a negative operator for every $t\in\R$.
\item For every $t\in\R$ we can choose
\[
B(t)=(-\Delta)^{-\gamma(t)},\quad t\in\R,
\]
or 
\[
B(t)=(I-\widetilde{\Delta})^{-\gamma(t)},\quad t\in\R,
\]
where $\Delta$ is the realization of the Laplacian operator in $L^2(\mathcal{O})$ with Dirichlet boundary condition and $\widetilde{\Delta}$ is the realization of the Laplacian operator in $L^2(\mathcal{O})$ with Neumann or Robin boundary conditions, $\gamma:\R\rightarrow [0,\alpha]$ is a non-decreasing measurable function for some $\displaystyle{0<\alpha<\frac{1}{2}}$.
\end{enumerate}
In the following cases,
\begin{align*}
\begin{cases}
\displaystyle{\inf_{t\in\R}\gamma(t)\geq 0},\ \ \ \ \mbox{if}\ d=1,\\[1ex]
\displaystyle{\inf_{t\in\R}\gamma(t)> 0},\ \ \ \ \mbox{if}\ d=2,\\[1ex]
\displaystyle{\inf_{t\in\R}\gamma(t)>\frac{1}{4}},\ \ \ \, \mbox{if}\ d=3,
\end{cases}
\end{align*}
there exists an evolution operator $U(t,s)$ satisfying Hypothesis \ref{4} and the assumptions of Corollary \ref{HYPER} and Proposition \ref{prop_bignamini} for every $t\in\R$, with constants $M_t,\omega_t$ and $\alpha_t$ independent of $t$.
\end{thm}
\begin{proof}
The existence of $U(t,s)$ is still given by \cite[Theorem 2.3]{acq_1988}.
and in \cite[Theorem 8]{big_def_2024} it is shown that the rest of Hypothesis \ref{4} is satisfied. Moreover, still in \cite[Theorem 8]{big_def_2024} it is proven that there exists $C>0$ such that
\begin{align}
\label{stima_ex_sing}
\norm{U(t,s)}_{\mathscr L({\rm H_s};{\rm H_t})}\leq  C\max\left\{1,\frac{1}{(t-s)^{\gamma(t)-\gamma(s)}}\right\}e^{-(\omega-\delta_0)(t-s)},\ \ s<t,
\end{align}
where $\omega$ and $\delta$ are the constant given in Hypothesis \ref{A}. By (iv) we get 
\begin{align*}
\norm{U(t,s)}_{\mathscr L({\rm H_s};{\rm H_t})}\leq  C\max\left\{1,\frac{1}{(t-s)^{\alpha}}\right\}e^{-(\omega-\delta_0)(t-s)},\ \ s<t.
\end{align*}
Hence by Proposition \ref{prop_bignamini}(iv) we get 
\begin{align*}
\norm{U(t,s)}_{\mathscr L({\mathcal{H}_s};{\mathcal{H}_t})}< \min\left\{1,C\, \frac{e^{-(\omega-\delta_0)(t-s)}}{(t-s)^\alpha}\right\},\ \ s<t,
\end{align*}
and the statement follows.
\end{proof}
\begin{remark}
The fact that $\|U(t,s)_{|\mathcal H_s}\|_{\mathscr L(\mathcal H_s,\mathcal H_t)}$ blows up as $s$ tends to $t$ may occur. For instance, if there exists $t\in \R$ such that $\gamma$ is not continuous at $t$ and we set $\gamma(t^-):=\lim_{s\to t^-}\gamma(s)$, then the function $\Delta\ni (s,t)\mapsto(t-s)^{\gamma(t)-\gamma(s)}$ behaves like $(t-s)^{\gamma(t)-\gamma(t^-)}$ when $s$ tends to $t$. This implies that the right-hand side in \eqref{stima_ex_sing} blows up as $s$ tends to $t$. \\
If $\gamma$ is continuous at $t\in\R$, a condition which ensures the explosion of the right-hand side in \eqref{stima_ex_sing} is that $(\gamma(t)-\gamma(s))\log(t-s)$ diverges to $ -\infty$ as $s$ tends to $t$. 
\end{remark}

\subsection{A non-autonomous version of the classical Ornstein-Uhlenbeck operator}

Let $A(t)=a(t) I$, where $a$ is a continuous and bounded real valued map on $\R$ and set $\displaystyle{\sup_{t\in \R}a_0(t)=a_0}$. 
Hence,
\begin{equation}\label{OEmalliavin}
U(t,s)=\exp\biggl(\int_s^ta(\tau)\,d\tau\biggr) I,\qquad s<t,
\end{equation}
 is continuous with values in $\mathscr L (X)$ and it is the evolution operator associated to the family $\{A(t)\}_{t\in\R}$. 

Let $\{B(t)\}_{t\in\R}\subseteq\mathscr L (X)$ be a family of operators satisfying Hypothesis \ref{4}(ii).
Since
\begin{align*}
Q(t,s)=\int_s^t \exp\biggl(2\int_\sigma^ta(\tau)\,d\tau\biggr)Q(\sigma)\,d\sigma, \qquad s<t,
\end{align*}
where $Q(\sigma)=B(\sigma)B(\sigma)^\star$, a sufficient condition for ${\rm Tr}[Q(t,s)]<\infty$ is that ${\rm Tr}[Q(\sigma)]<\infty$ for every $\sigma\in\R$ and $\sigma\mapsto{\rm Tr}[Q(\sigma)]\in L^\infty(\R)$.
Indeed, if $\{e_k:k\in\N\}$ is a Hilbert basis of $X$, then we have
\begin{align}
\label{tracciamallia}
\ix{Q(t,s)e_k}{e_k}\leq \int_s^t e^{2a_0(t-\sigma)}\ix{Q(\sigma)e_k}{e_k}\,d\sigma, \qquad s<t.
\end{align}
In this case, $\{P_{s,t}\}_{(s,t)\in\overline \Delta}$, defined as in \eqref{OUnof}, is a non-autonomous generalization of the classical Ornstein-Uhlenbeck semigroup which arises in the Malliavin Calculus. 


In addition to the above assumptions on the trace of the operators $Q(\sigma)$, $\sigma\in\R$, we require that for all $t\in\R$ there exists $C_t>0$ such that
\begin{equation}\label{normbt}
\norm{B(s)x}_{X}\leq C_t\norm{B(t)x}_{X},\quad s<t,\ x\in X.
\end{equation}
Moreover we assume also that $a_0< 0$.

By \eqref{tracciamallia} and  \eqref{normbt}, it follows that $\norm{U(t,s)}_{\mathscr L (X)}\leq e^{a_0(t-s)}$ for every $s<t$. Moreover,
\begin{align*}
\ix{Q(t,s)e_k}{e_k}&\leq \int_s^te^{2a_0(t-\sigma)} \ix{Q(\sigma)e_k}{e_k}\,d\sigma= \int_s^t e^{2a_0(t-\sigma)}\norm{B(\sigma)^\star e_k}_X^2\,d\sigma\nonumber\\
&\leq  C_t^2 \norm{B(t)^\star e_k}_X^2\int_s^t e^{2a_0(t-\sigma)}\,d\sigma=-\frac{C_t^2}{2a_0} \left[1-e^{2a_0(t-s)}\right]\ix{Q(t)e_k}{e_k}
\end{align*} 
for all $s<t$. Therefore, $\sup_{s<t}{\rm Tr}[Q(t,s)]<\infty$ and by \cite[Theorem 4 \& Remark 12]{big_def_2024} there exists a unique evolution system of measures for $\{P_{s,t}\}_{s\leq t}$.

We now prove that $\mathcal H_s\subseteq \mathcal H_t$, $s<t$, with continuous embedding. For every $s<t$,
\begin{align*}
\norm{Q(s,-\infty)^{\frac{1}{2}}x}_X^2&=\ps{Q(s,-\infty)x}{x}_X=\int_{-\infty}^s \exp\left(2\int_r^sa(\tau)d\tau\right)\norm{B^\star(r)x}^2dr \\
&
\leq \exp\left(-2\int_s^ta(\tau)d\tau\right)\int_{-\infty}^t \exp\left(2\int_r^ta(\tau)d\tau\right)\norm{B^\star(r)x}^2dr \\
&=\exp\left(-2\int_s^ta(\tau)d\tau\right)\norm{Q(t,-\infty)^{\frac{1}{2}}x}_X^2\end{align*}
Then, by Proposition \ref{pseudo}, $\mathcal H_s\subseteq \mathcal H_t$ with continuous embedding and for every $x\in \mathcal H_s$ we have
\begin{align*}
\norm{U(t,s)x}_{\mathcal H_t}=\norm{\exp\biggl(\int_s^ta(\tau)\,d\tau\biggr)x}_{\mathcal H_t}\leq \norm{x}_{\mathcal H_s}.
\end{align*}
Since \eqref{normbt} holds, we have
\begin{align*}
\norm{Q(s)^{\frac 12}x}_X^2
=\norm{B(s)^\star x}_X^2
\leq C_t^2\norm{B(t)^\star x}_X^2
=C_t^2\norm{Q(t)^{\frac 12}x}_X^2,
\end{align*}
and so ${\rm H}_s\subseteq {\rm H}_t$ for all $t>s$. Moreover,
\begin{align*}
\norm{U(t,s)x}_{{\rm H}_t}=\norm{\exp\left(\int_s^ta(\tau)\,d\tau \right)x}_{{\rm H}_t}\leq C_t e^{a_0(t-s)}\norm{x}_{{\rm H}_s}, \qquad x\in H_s.
\end{align*}
By Proposition \ref{prop_bignamini}(iv), for all $s<t$ we get
\begin{align*}
\norm{U(t,s)}_{\mathscr L(\mathcal{H}_s;\mathcal{H}_t)}<\min\left\{1,C_t e^{a_0(t-s)}\right\}.
\end{align*}

We summarize what we have proved in the following theorem.

\begin{thm} Let $(X,\norm{\,\cdot\,}_X,\ps{\cdot}{\cdot}_X)$ be a separable Hilbert space, let $A(t)=a(t) I$, where $a$ is a continuous and bounded real valued map on $\R$ and let $\{B(t)\}_{t\in\R}\subseteq\mathscr{L}(X)$ be a family of operators satisfying Hypothesis \ref{4}-(ii). We set 
$$\sup_{t\in \R}a_0(t)=a_0,\qquad  Q(t)=B(t)B(t)^\star,\quad t\in\R .$$ 
We assume that ${\rm Tr}[Q(t)]<\infty$ for every $t\in\R$ and $t\mapsto{\rm Tr}[Q(t)]\in L^\infty(\R)$.
Then, $U(t,s)$ given by \eqref{OEmalliavin} satisfies Hypothesis \ref{4} and the operators $Q(t,s)$ have finite trace for all $s<t$. If in addition $a_0<0$ and \eqref{normbt} holds, then  there exists an evolution system of measures for $P_{s,t}$ given by \eqref{gammat} and the hypotheses of Corollary \ref{HYPER} and Proposition \ref{prop_bignamini} are satisfied, for every $t\in\R$, with $M_t=C_t$, $\omega=a_0$ and $\alpha_t=0$ for every $t\in\R$.
\end{thm}

\subsection{Diagonal operators}\label{ex_diag}
Let $(X,\norm{\,\cdot\,}_X,\ix{\cdot}{\cdot})$ be a separable Hilbert space. Let $t\in\R$ and let $A(t)$, $B(t)$ be self-adjoint operators in diagonal form with respect to the same Hilbert basis $\{e_k:k\in\N\}$, namely $$A(t)e_k=a_k(t)e_k,\ \ B(t)e_k=b_k(t)e_k\ \ t\in\R,\ \ k\in\N,$$ with continuous coefficients $a_k$ and measurable coefficients $b_k$. We set $\displaystyle{\lambda_k=\sup_{t\in\R}a_k(t)}$ and we assume that there exists $\lambda_0\in\R$ such that $\lambda_k\leq\lambda_0,\ \ \forall\ k\in\N.$

In this setting  the operator $U(t,s)$ defined by 
\begin{equation}\label{OEdiagonale}
U(t,s)e_k=\exp\biggl(\int_s^ta_k(\tau)\,d\tau\biggr)e_k,\ \ s<t,\ \ k\in\N,
\end{equation}
is the strongly continuous evolution operator associated to the family $\{A(t)\}_{t\in\R}$.
Moreover, we assume that there exists $K>0$ such that 
\begin{equation}\label{stimabkK}
\abs{b_k(t)}\leq K,\ \ t\in\R,\ k\in\N.
\end{equation}
Hence $B(t)\in\mathscr L (X)$ for all $t\in\R$, the function $B:\R\mapsto\mathscr L (X)$ is measurable and $$\sup_{t\in\R}\norm{B(t)}_{\mathscr L (X)}\leq K.$$
The operators $Q(t,s)$ are given by
\begin{equation}
Q(t,s)e_k=\int_s^t\exp\biggl(2\int_\sigma^ta_k(\tau)\,d\tau\biggr)(b_k(\sigma))^2\,d\sigma\,e_k=:q_k(t,s)e_k,\ \ s<t,\ k\in\N. \nonumber
\end{equation}
Hypothesis \ref{4}(iii) is fulfilled if 
\begin{equation}\label{qk1}
\sum_{k=0}^\infty q_k(t,s)<\infty,\ \ s<t.
\end{equation}
We give now a sufficient condition for \eqref{qk1} to hold.
We assume that $\lambda_k$ is eventually non-zero (say for $k\geq k_0$). Given $s<t$, we have
\begin{align}\label{qk2}
&\abs{\int_s^t\exp\biggl(2\int_\sigma^ta_k(\tau)\,d\tau\biggr)(b_k(\sigma))^2\,d\sigma}\leq \norm{b_k}^2_{\infty}\abs{\int_s^t\exp\bigl(2\lambda_k(t-\sigma)\bigr)\,d\sigma}\nonumber \\
&=\frac{\norm{b_k}^2_{\infty}}{2\abs{\lambda_k}}\abs{1-\exp(2\lambda_k(t-s))}\leq\frac{\norm{b_k}^2_{\infty}}{2\abs{\lambda_k}}\bigl(1+\exp(2\lambda_0(t-s))\bigr).
\end{align}
Hence, \eqref{qk1} holds if we require 
\begin{equation}\label{trfinitaqk}
\sum_{k=k_0}^\infty\frac{\norm{b_k}^2_{\infty}}{\abs{\lambda_k}}<\infty.
\end{equation}

If $\lambda_0< 0$ in \eqref{qk2}, then by \cite[Theorem 4 \& Remark 12]{big_def_2024} there exists a unique evolution system of measures for $P_{s,t}$.

We estimate now $\norm{U(t,s)}_{\mathscr{L}(\mathcal H_s;\mathcal H_t)}$. Firstly, by Proposition \ref{tesoretto} we know that $U(t,s)$ maps $\mathcal H_s$ into $\mathcal H_t$ and 
 \begin{equation*}
\norm{U(t,s)}_{\mathscr{L}(\mathcal H_s;\mathcal H_t)}=\norm{Q^{-\frac{1}{2}}(t,-\infty)U(t,s)Q(s,-\infty)^{\frac{1}{2}}}_{\mathscr{L}(X)}\leq 1.
\end{equation*}
Assume that there exist $C,M>0$ such that 
\begin{align}
\label{stimabkC}
\abs{b_k(t)}\geq C, \quad 
a_k(t)\geq M\lambda_k, \qquad \mbox{a.e}\ t\in\R,\ \forall\ k\in\N. 
\end{align}
It follows that, for every $s<t$, 
\begin{align*}
\norm{U(t,s)}_{\mathscr{L}(\mathcal H_s;\mathcal H_t)}^2
= & \sup_{k\in\N} \frac{q_k(s,-\infty)}{q_k(t,-\infty)} \exp\left(2\int_s^t a_k(\tau)d\tau\right)
=\sup_{k\in\N}\frac{\int_{-\infty}^s\exp\left(2\int_\sigma^ta_k(\tau)d\tau\right)(b_k(\sigma))^2d\sigma}{\int_{-\infty}^t\exp\left(2\int_\sigma^ta_k(\tau)d\tau\right)(b_k(\sigma))^2d\sigma} \\
= & \sup_{k\in\N}\frac{\int_{-\infty}^s\exp\left(2\int_\sigma^ta_k(\tau)d\tau\right)(b_k(\sigma))^2d\sigma}{\int_{-\infty}^s\exp\left(2\int_\sigma^ta_k(\tau)d\tau\right)(b_k(\sigma))^2d\sigma+\int_s^t\exp\left(2\int_\sigma^ta_k(\tau)d\tau\right)(b_k(\sigma))^2d\sigma} \\
\leq & \sup_{k\in\N}\frac{\int_{-\infty}^s\exp\left(2\int_\sigma^ta_k(\tau)d\tau\right)(b_k(\sigma))^2d\sigma}{\int_{-\infty}^s\exp\left(2\int_\sigma^ta_k(\tau)d\tau\right)(b_k(\sigma))^2d\sigma+C^2(2M|\lambda_k|)^{-1}(1-e^{2M\lambda_k(t-s)})} \\
\leq & \sup_{k\in\N}\frac{2M\int_{-\infty}^s|\lambda_k|\exp\left(2\int_\sigma^ta_k(\tau)d\tau\right)(b_k(\sigma))^2d\sigma}{2M\int_{-\infty}^s|\lambda_k|\exp\left(2\int_\sigma^ta_k(\tau)d\tau\right)(b_k(\sigma))^2d\sigma+C^2(1-e^{2M\lambda_0(t-s)})}.
\end{align*}
Notice that the function $[0,\infty)\ni\xi\mapsto \frac{\xi}{\xi+c}$ is increasing for every $c>0$. Since
\begin{align*}
\int_{-\infty}^s|\lambda_k|\exp\left(2\int_\sigma^ta_k(\tau)d\tau\right)(b_k(\sigma))^2d\sigma
\leq K^2\int_{-\infty}^s|\lambda_k|e^{2\lambda_k(t-\sigma)}d\sigma\leq \frac{K^2}{2}e^{2\lambda_k(t-s)}
\end{align*}
for every $s<t$ and $k\in\N$, it follows that 
\begin{align*}
\norm{U(t,s)}_{\mathscr{L}(\mathcal H_s;\mathcal H_t)}^2
\leq & \sup_{k\in\N}\frac{MK^2e^{2\lambda_k(t-s)}}{MK^2e^{2\lambda_k(t-s)}+C^2(1-e^{2M\lambda_0(t-s)})} \\
\leq & \frac{MK^2e^{2\lambda_0(t-s)}}{MK^2e^{2\lambda_0(t-s)}+C^2(1-e^{2M\lambda_0(t-s)})}, \qquad s<t.
\end{align*}
Further, for every $s\leq t$ we get
\begin{align*}
\norm{U(t,s)}_{\mathscr{L}(\mathcal H_s;\mathcal H_t)}^2
\leq & MK^2\max\left\{\frac{1}{C^2(1-e^{2M\lambda_0})},\frac{1}{MK^2e^{2\lambda_0}}\right\}e^{2\lambda_0(t-s)}.
\end{align*}
We stress that, if \eqref{trfinitaqk} and \eqref{stimabkC} hold true, then $\sum_{k\in\N}|\lambda_k|^{-1}<\infty$. From the above computations, we deduce that
\begin{align*}
\sum_{k\in\N}\norm{U(t,s)Q(s,-\infty)^{\frac12}e_k}_{\mathcal H_t}^2
\leq & \frac{MK^2}{C^2(1-e^{2M\lambda_0})}\sum_{k\in\N}e^{2\lambda_k(t-s)}<\infty,
\end{align*}
which shows that $U(t,s)$ is a strict contraction of Hilbert-Schmidt type (and hence a compact operator) from $\mathcal H_s$ into $\mathcal H_t$.


We summarize what we have obtained in the following theorem. 

\begin{thm} 
Let $(X,\norm{\,\cdot\,}_X,\ix{\cdot}{\cdot})$ be a separable Hilbert space. Let $A(t)$, $B(t)$ be self-adjoint operators in diagonal form with respect to the same Hilbert basis $\{e_k:k\in\N\}$, namely $$A(t)e_k=a_k(t)e_k,\ \ B(t)e_k=b_k(t)e_k\ \ t\in\R,\ \ k\in\N,$$ 
with $a_k$ continuous and $b_k$ measurable for every $k\in\N$, and $\displaystyle{\lambda_k=\sup_{t\in\R}a_k(t)}$. 
If there exists $\lambda_0\in\R$ such that $\lambda_k\leq\lambda_0,\ \ \forall\ k\in\N$ and \eqref{stimabkK} and \eqref{trfinitaqk} hold, then Hypothesis \ref{4} is satisfied with $U(t,s)$ given by \eqref{OEdiagonale}. If in addition $\lambda_0<0$, then there exists an evolution system of measure for $\{P_{s,t}\}_{s\leq t}$ given by \eqref{gammat}.
Moreover, if \eqref{stimabkC} is satisfied, then the hypotheses of Corollaries \ref{HYPER}, \ref{compattezzapst}, \ref{Hilbert-Schmidtpst} and \ref{coro:exp_decay}, for every $t\in\R$ with constants $\omega_t=\lambda_0$ and $M_t$ independent of $t$, are satisfied.
\end{thm}

\begin{remark}
As an explicit example we can choose $a_k(t)=-k^2({\rm arctg}\abs{t}+c_1)$, $c_1>0$ and 
\begin{equation*}
b_k(t)=
\begin{cases} 
\sin(kt)+c_2, &  t\neq -m\pi, \ m\in\N\cup\{0\},\\
0, & t=-m\pi, \ m\in\mathbb N\cup\{0\},
\end{cases}
\end{equation*}
with $c_2>1$. With these choices, we remark that for every $m\in\mathbb N\cup\{0\}$ we get $U(-m\pi,s){\rm H}_s\not\subseteq{\rm H}_{-m\pi}$ for $s<-m\pi$ and so  of \cite[Hypothesis (116), Theorem 5]{big_def_2024} is not satisfied. 
\end{remark}

\subsubsection*{Acknowledgments}
The authors are members of the italian group G.N.A.M.P.A. (INdAM).

\appendix
\section{Measurable linear operators}\label{mis_op_app}

In this section we prove some basic results on measurable linear operators. Unfortunately, we were able to find them only for measurable linear functionals and so we decided to write detailed proofs for the reader. The main references are \cite{bog_1998, lun_mir_pal_toappear}.

Let $X$ and $Y$ be two separable Banach spaces and let $\gamma$ be a Gaussian measure on X. Let $\mathcal{B}(X)$ be the $\sigma$-algebra of the Borel sets on $X$ and let  $\mathcal{B}(X)_\gamma$ the completion of $\mathcal{B}(X)$ with respect to $\gamma$.

\begin{defn}\label{opmis}   A mapping $F:X\to Y$ is called a $\gamma$-measurable linear operator if there exists a linear mapping $F_0:X\to Y$, measurable with respect to the pair of $\sigma$-fields $(\mathcal{B}(X)_\gamma, \mathcal{B}(Y))$ and such that $F=F_0$ almost everywhere. We say that $F_0$ is a linear version of $F$.
\end{defn}

%

We generalize here the result in \cite[Thm. 2.10.7]{bog_1998} proved for measurable linear functionals.

\begin{proposition}\label{propf=gsuH}
If $F,G:X\to Y$ are two measurable linear operators, then either $\gamma(\{F=G\})=1$ or $\gamma(\{F=G\})=0$. 
We have $\gamma(\{F=G\})=1$ if and only if $F=G$ on $H_\gamma$. 
\end{proposition}
\begin{proof}
Let $F_0$ and $G_0$ be the linear versions of $F$ and $G$, respectively. Then $L=\{F_0=G_0\}$ is a $\mathcal{B}(X)_\gamma$-measurable vector space. By \cite[Thm. 2.5.5]{bog_1998} either $\gamma(L)=0$, or $\gamma(L)=1$. 
If $\gamma(L)=1$ then $H_\gamma\subseteq L$ by Proposition \ref{teoria}(i) and $F=G$ on $H_\gamma$. 

Conversely, if $F_0=G_0$ on $H_\gamma$ then for all $h\in Y^\star$ the measurable functional $h(F_0-G_0)$ verifies 
$h(F_0-G_0)(x+h)=h(F_0-G_0)(x)$ for every $h\in H_\gamma$. By \cite[Thm. 2.5.2]{bog_1998} $h(F_0-G_0)=c$ $\gamma$-a.e. in $X$ and, since $h(F_0-G_0)$ is linear, $c=0$. \\
Let $\{y_n:n\in\N\}$ be a dense subset in $Y$. From Hahn-Banach Theorem, for every $n\in\N$ there exists $y_n^\star\in Y^\star$ such that $y^\star_n(y_n)=\|y_n\|_Y$. Further, for every $y\in Y$ we get $\|y\|_Y=\sup_{n\in\N}|y_n^\star(y)|$. For every $n\in\N$, we denote by $A_n$ the Borel subset of $X$ of $\gamma$-full measure such that $y_n^\star(F_0-G_0)=0$ on $A_n$ and we set $A:=\cap_{n\in\N}A_n$. Clearly, $\gamma(A)=1$ and, for every $x\in A$, we get $y_n^\star(F_0(x)-G_0(x))=0$. Therefore, $\|(F_0-G_0)(x)\|_Y=\sup_{n\in\N}|y_n^\star(F_0(x)-G_0(x))|=0$, which gives the desired implication. 
\end{proof}

We generalize here the result in \cite[Prop. 1]{urb_1975}, proved for measurable linear functionals.

\begin{proposition}\label{borelversion}
If  $F:X\rightarrow Y$ is  a $\gamma$-measurable linear operator, than there exists e Borel linear space $V\subseteq X$ such that $\gamma(V)=1$ and $(F_0)_{|V}$ is Borel measurable.
\end{proposition}

\begin{proof}
From Lusin's Theorem (see e.g. \cite[Thm. 7.1.13 Vol. II]{bog_2007}) for every $n\in\N$ there exists a compact subset $K_n\subseteq X$ such that $\gamma(K_n)> 1-\frac{1}{n}$ and $(F_{0})_{|K_n}$, is continuous. For every $j,k\in\N$, we introduce the set
$$L_{j,n}=\left\{\sum_{i=1}^jc_i x_i,\ :\ \abs{c_i}\leq j,\  x_i\in K_n,\  i=1,...,j \right\}$$
and, for every $j\in\N$, the mapping $T_j:\R^j\times X^j\to X$ defined as 
$$(c_1,...,c_j,x_1,...,x_j)\mapsto \sum_{i=1}^jc_i x_i,\quad c_1,...,c_j\in\R,\ x_1,...,x_j\in X.$$
For every $j\in\N$, the mapping $T_j$ is linear and continuous for all $j\in\N$ and $L_{j,n}$ is compact since $L_{j,n}=T_j\left([-j,j]^j\times K_n^j\right)$. We claim that $(F_{0})_{|L_{j,n}}$ is continuous. \\
Let $\{y_h\}\subseteq L_{j,n}$ and $y_0\in L_{j,n}$ be such that $y_h\to y_0$ in $X$. For every $h\in\N$, there exist $x_{i,h}\in K_n$ and $c_{i,h}\in[-j,j]$, with $i=1,...,j$, such that $y_h=\sum_{i=1}^jc_{i,h} x_{i,h}$. This implies that there exist an increasing sequence of indices $\{m_h\}_{h\in\N}$, $x_{i,0}\in K_n$ and $c_{i,0}\in[-j,j]$, with $i=1,...,j$, such that $x_{i,m_h}\to x_{i,0}$ in $X$ and $c_{i,h}\to c_{i,0}$ for every $i=1,...,j$.
It follows that
\begin{align*}
\norm{y_{m_h}-\sum_{i=1}^jc_{i,0} x_{i,0}}_X\leq \sum_{i=1}^j\abs{c_{i,m_h}-c_{i,0}} \norm{x_{i,m_h}}_X+\sum_{i=1}^j\abs{c_{i,0}} \norm{x_{i,m_h}-x_{i,0}}_X\to 0,\ \ \mbox{as}\ h\to\infty,
\end{align*}
which gives $y_0=\sum_{i=1}^jc_{i,0}x_{i,0}$. Since every subsequence of $\{y_h\}$ admits a subsequence which converges to $y_0$, we conclude that the whole sequence converges to $y_0$. 
The continuity of $F_{0|K_n}$ implies that
\begin{align*}
(F_{0})_{|L_{j,n}}(y_h)=\sum_{i=1}^jc_{i,h} (F_{0})_{|L_{j,n}}(x_{i,h})=\sum_{i=1}^jc_{i,h} (F_{0})_{|K_n}(x_{i,h})\to\sum_{i=1}^jc_{i,0} (F_{0})_{|K_n}(x_{i,0})=(F_{0})_{|L_{j,n}}(y_0)
\end{align*}
as $h$ tends to infinity, which gives the claim. Setting $$V=\bigcup_{n=1}^\infty\bigcup_{j=1}^\infty L_{j,n},$$
we have that $V$ is a Borel linear subspace of $X$ and $\gamma(V)=1$ since $K_n\subseteq V$ for all $n\in\N$.
Finally, we have to prove that $F_{0|V}$ is Borel measurable. Given $B\in\mathcal B(Y)$, we have
\begin{align*}
((F_{0})_{|V})^{-1}(B)=\bigcup_{n=1}^\infty\bigcup_{j=1}^\infty ((F_{0})_{|L_{j,n}})^{-1}(B)\in\mathcal B(X).
\end{align*}
and the statement follows.
\end{proof}

\begin{defn}
Let $F:X\rightarrow Y$ be  a $\gamma$-measurable linear operator. The version $(F_0)_{|V}$ of $F$ given in Proposition \ref{borelversion} is the Borel version of $F$ and we denote it with $\widetilde{F}$.
\end{defn}

We generalize here the result in \cite[Thm. 9.10]{kec_1995}, obtained for measurable linear functionals.

\begin{thm}
Let $F:X\to Y$ be a $\gamma$-measurable linear operator and let $\widetilde{F}:V\subseteq X\to Y$ be its Borel version. If $V=X$, then $\widetilde{F}\in \mathscr{L}(X)$.
\end{thm}

\begin{proof}
We prove the continuity at $0$. Let $\mathcal{D}$ be a countable dense subset of $Y$ and fix $\varepsilon>0$. Clearly
\[
Y=\bigcup_{y\in \mathcal{D}} B_\varepsilon(y), \qquad 
X=\bigcup_{y\in \mathcal{D}} \widetilde{F}^{-1}\left(B_\varepsilon(y)\right),
\]
and in particular 
\[
X=\bigcup_{y\in \mathcal{D}} \overline{\widetilde{F}^{-1}\left(B_\varepsilon(y)\right)}.
\]
By the Baire Lemma there exists $y_0\in \mathcal{D}$ such that
$\overline{\widetilde{F}^{-1}\left(B_\varepsilon(y_0)\right)}$ has non-empty interior. Since $\widetilde F$ is linear then $\widetilde{F}^{-1}\left(B_\varepsilon(y_0)\right)$ is convex and $\mathring{\overline{\widetilde{F}^{-1}\left(B_\varepsilon(y_0)\right)}}=\mathring{\overbrace{\widetilde{F}^{-1}\left(B_\varepsilon(y_0)\right)}}$.
Hence, even $\widetilde{F}^{-1}\left(B_\varepsilon(y_0)\right)$ has non-empty interior and there exist $x\in X$, $r>0$ such that
$B(x,r)\subseteq \widetilde{F}^{-1}\left(B_\varepsilon(y_0)\right)$. In particular $\norm{\widetilde{F}(x)-y_0}_Y<\varepsilon$ and $B(0,r)\subseteq  \widetilde{F}^{-1}\left(B_\varepsilon(y_0)\right)-x$.
If $y\in \widetilde{F}^{-1}\left(B_\varepsilon(y_0)\right)-x$, then
\[
\norm{\widetilde{F}(y)}_Y =\norm{\widetilde{F}(y)+\widetilde{F}(x)-y_0-\widetilde{F}(x)+y_0}_Y
\leq \norm{\widetilde{F}(y+x)-y_0}_Y+\norm{\widetilde{F}(x)-y_0}_Y<2\varepsilon.
\]
and we have $$B(0,r)\subseteq \widetilde{F}^{-1}\left(B_\varepsilon(y_0)\right)-x\subseteq \widetilde{F}^{-1}\left(B_{2\varepsilon}(0)\right)$$
Since $\varepsilon$ is arbitrary, we deduce the continuity
of $\widetilde{F}$ in $0$.
\end{proof}

Using the property of the operators $L^{\nu,\mu}_\infty$, it is possible to prove a composition rule for the second quantization operator defined by the integral formula \eqref{gammamunu}. Since it is interesting in itself, we provide the proof in this appendix. 

\begin{lemma}
\label{magliaia}
Let $\sigma, \mu$ and $\nu$ be centered Gaussian measures on $X$ with $H_\sigma$, $H_\mu$ and $H_\nu$ as Cameron Martin spaces, respectively. If $A\in \mathscr{L}(H_\sigma,H_\nu)$ and $B\in \mathscr{L}(H_\mu,H_\sigma)$, then $A_\infty^{\nu,\sigma}(B_\infty^{\sigma,\mu})$ is a version of $(AB)_\infty^{\nu,\mu}$.
\end{lemma}
\begin{proof}
By Remark \ref{opmisoss} (ii), there exists a Borel linear subspace $V$ of $X$ such that $\sigma(V)=1$ and $\widetilde{A_\infty^{\nu,\sigma}}:V\to X$ is linear on $V$ and Borel measurable. Setting $Z=((B_\infty^{\sigma,\mu})_0)^{-1}(V)$, $\widetilde{A_\infty^{\nu,\sigma}}((B_\infty^{\sigma,\mu})_0)_{|Z}:Z\to X$ is measurable with respect to the pair of $\sigma$-fields $(\mathcal{B}(X)_\mu, \mathcal{B}(X))$. Moreover since $\sigma(V)=1$, by Proposition \ref{prop:op_L_infty2} and \cite[Thm. 3.3.4]{bog_1998} it follows that $\mu\circ(B_\infty^{\sigma,\mu})^{-1}(V)=1$, which implies that $\mu(Z)=1$. Then $\widetilde{A_\infty^{\nu,\sigma}}((B_\infty^{\sigma,\mu})_0)_{|Z}$ is a linear version of $A_\infty^{\nu,\sigma}(B_\infty^{\sigma,\mu})$ on the subspace $Z$ with $\mu(Z)=1$ and $\widetilde{A_\infty^{\nu,\sigma}}((B_\infty^{\sigma,\mu})_0)_{|Z}$ is measurable with respect to the pair of $\sigma$-fields $(\mathcal{B}(X)_\mu, \mathcal{B}(X))$. Taking a Hamel basis of $X$ we can consider a linear extension $(A_\infty^{\nu,\sigma}(B_\infty^{\sigma,\mu}))_0$ of  $\widetilde{A_\infty^{\nu,\sigma}}((B_\infty^{\sigma,\mu})_0)_{|Z}$ on $X$, still measurable with respect to the pair of $\sigma$-fields $(\mathcal{B}(X)_\mu, \mathcal{B}(X))$. Moreover, given $x\in H_\mu$, by Lemma \ref{lemma:prop_L_infty3} (ii) and recalling that $H_\mu\subseteq Z$ (see Proposition \ref{teoria}(i)), we have
\begin{align*}
(A_\infty^{\nu,\sigma}(B_\infty^{\sigma,\mu}))_0 x&= \widetilde{A_\infty^{\nu,\sigma}}((B_\infty^{\sigma,\mu})_0)_{|Z} x= \widetilde{A_\infty^{\nu,\sigma}}Bx=ABx.
\end{align*}
By Remark \ref{opmisoss} the statement follows.
\end{proof}

\begin{proposition}
\label{composizionegenetale}
Let $p\in[1,\infty)$ and let $\nu, \sigma, \mu$ be three centered Gaussian measures on $X$ with Cameron-Martin spaces $H_\nu$, $ H_\sigma$ and $ H_\mu$, respectively. 
If $S\in\mathscr L(H_\mu,H_\sigma)$ with $\norm{S}_{\mathscr L(H_\mu,H_\sigma)}\leq 1$ and if
$L\in\mathscr L(H_\sigma,H_\nu)$ with $\norm{L}_{\mathscr L(H_\sigma,H_\nu)}\leq 1$, then
\begin{align*}
\Gamma_{\nu,\mu}(LS)f=\Gamma_{\nu,\sigma}(L)\Gamma_{\sigma,\mu}(S)f,\quad f\in L^p(X). 
\end{align*}
\end{proposition}

\begin{proof}
Let $p\in [1,\infty)$ and let $f\in \mathscr E(X)$. 
\begin{enumerate}
\item If $S=0$, then there is nothing to prove.
\item If $L=0$, then from \eqref{super_change} $\nu$-almost every $x\in X$ we have
\begin{align*}
(\Gamma_{\nu,\sigma}(0)\Gamma_{\sigma,\mu}(S)f)(x)=&\int_X(\Gamma_{\sigma,\mu}(S)f)(y)\,\sigma(dy)
= \int_Xf(y)\mu(dy)
=\Gamma_{\nu,\mu}(0)f(x).
\end{align*}

\item If $L,S\neq 0$, then $\nu$-almost every $x\in X$ by Lemma \ref{magliaia} we get
\begin{align*}
&(\Gamma_{\nu,\sigma}(L)\Gamma_{\sigma,\mu}(S)f)(x)=\int_X(\Gamma_{\sigma,\mu}(S)f)\left[(L^\star)^{\sigma,\nu}_\infty x+\left((I-L^\star L)^{\frac{1}{2}}\right)^{\sigma}_\infty y\right]\,\sigma(dy)\nonumber\\
&=\int_X\int_X f\left[(S^\star)^{\mu,\sigma}_\infty(L^\star)^{\sigma,\nu}_\infty x+(S^\star)^{\mu,\sigma}_\infty\left((I-L^\star L)^{\frac{1}{2}}\right)^{\sigma}_\infty y+\left((I-S^\star S)^{\frac{1}{2}}\right)^{\mu}_\infty z\right]\,\mu(dz)\sigma(dy)\nonumber\\
&=\int_X\int_X f\left[\left((LS)^\star\right)^{\mu,\nu}_\infty x+\left(S^\star (I-L^\star L 
)^{\frac{1}{2}}\right)^{\mu,\sigma}_\infty y+\left((I-S^\star S)^{\frac{1}{2}}\right)^{\mu}_\infty z\right]\,\mu(dz)\sigma(dy).
\end{align*}
Setting $\eta=\left(S^\star (I-L^\star L)^{\frac{1}{2}}\right)^{\mu,\sigma}_\infty y$ and $\zeta=\left((I-S^\star S)^{\frac{1}{2}}\right)^{\mu}_\infty z$, we have
\begin{equation}
\label{app_comp_gamma}
(\Gamma_{\nu,\sigma}(L)\Gamma_{\sigma,\mu}(S)f)(x)=\int_X f\left[\left((LS)^\star\right)^{\mu,\nu}_\infty x+\eta+\zeta\right]\,\gamma_{\mu,T}(d\eta)\gamma_{\mu,R}(d\zeta),
\end{equation}
where 
\begin{align*}\gamma_{\mu,T}=\mathcal{N}\left(0,Q_\mu^{\frac{1}{2}}(\widehat T)^\star \widehat TQ_\mu^{\frac{1}{2}}\right), \qquad \widehat T=Q_\sigma^{-\frac12}(I-L^\star L)^\frac12 SQ_\mu^{\frac12}, \\
\gamma_{\mu,R}=\mathscr{N}\left(0,Q_\mu^{\frac{1}{2}}(\widehat R)^\star RQ_\mu^{\frac{1}{2}}\right), \qquad \widehat R=Q_\mu^{-\frac12}(I-S^\star S)^\frac12 Q_\mu^\frac12.
\end{align*}
Setting $\xi=\eta+\zeta$, by Proposition \ref{corDecompositionHilbert}(ii) and arguing as in \eqref{conto_somma_cov} we have 
\begin{align}
\label{app_conv_mis}
\gamma_{\mu,T}*\gamma_{\mu,R}=\mathcal{N}\left(0,Q\right), \qquad Q=Q_\mu^{\frac{1}{2}}(\widehat T)^\star \widehat TQ_\mu^{\frac{1}{2}}+Q_\mu^{\frac{1}{2}}(\widehat R)^\star \widehat RQ_\mu^{\frac{1}{2}},
\end{align}
with
\begin{align*}
\langle Qx,y\rangle_X
= & \langle (I-L^\star L)^\frac12 SQ_\mu x,(I-L^\star L)^\frac12 SQ_\mu y\rangle_{H_\mu} +
\langle (I-S^\star S)^\frac12 Q_\mu x,(I-S^\star S)^\frac12 Q_\mu y\rangle_{H_\mu} \\
= & \langle (I-(LS)^\star (LS))Q_\mu x,Q_\mu y\rangle_{H_\mu}
= \langle \widehat H Q_\mu^\frac12 x,\widehat H Q_\mu^\frac12 y\rangle_X 
\end{align*}
for every $x,y\in X$, where $\widehat H=Q_\mu^{-\frac12}(I-(LS)^\star (LS))^\frac12Q_\mu^\frac12$. From \eqref{app_comp_gamma} and \eqref{app_conv_mis}, it follows that
\begin{align*}
(\Gamma_{\nu,\sigma}(L)\Gamma_{\sigma,\mu}(S)f)(x)
= \int_Xf(((LS)^\star)_\infty^{\mu,\nu}x+\xi)\gamma_{\mu,H}(d\xi), \qquad x\in X, 
\end{align*}
where $\gamma_{\mu,H}=\mathcal N(0,Q_\mu^\frac12(\widehat H)^\star\widehat HQ_\mu^\frac12)$ and $\widehat H=Q_\mu^{-\frac12}(I-(LS)^\star LS)^\frac12 Q_\mu^\frac12$. From \eqref{gammabuona} and \eqref{gammanumuL}, with $T=LS$, we conclude that 
\begin{align*}
\int_Xf(((LS)^\star)_\infty^{\mu,\nu}x+\xi)\gamma_{\mu,H}(d\xi)
=(\Gamma_{\nu,\mu}(LS)f)(x), \qquad x\in X.    
\end{align*}
\end{enumerate}
\end{proof}

\section{Monomials}
In this last appendix we prove that our approach is analogous to \cite{ree_sim_1972} where the second quantization operator is defined considering monomials of the type 
\begin{equation}\label{formula_monomials}
\prod_{i=1}^n W^\mu_{h_i},\ \ h_1,...,h_n\in X,\ n\in\N,
\end{equation}
and not exploiting the Hermite polynomials. In particular, we show how the operator, introduced in Definition \ref{def:op_gamma_tilden}, acts on monomials of the type \eqref{formula_monomials}. 

To begin with, notice that, if $n\in\N$, then from \eqref{def_pol_herm_1} for every $j\in\N\cup\{0\}$ with $j\leq \frac n2$ we get
\begin{align}
\label{int_pot_xi}
\int_{\R}\xi^n \phi_{n-2j}(\xi)\mu_1(d\xi)
= \frac{n!}{(2j)!(n-2j)!}\int_\R\xi^{2j}\mu_1(d\xi)= \frac{n!(2j-1)!!}{(2j)!(n-2j)!}=\frac{n!}{2^jj!(n-2j)!},
\end{align}
where $(-1)!!=1$ and $2^{j-1}(j-1)!(2j-1)!!=(2j-1)!$. On the other hand, for $m>n$ we get
\begin{align}
\label{int_pt_xi_zero}
\int_{\R}\xi^n\phi_{m}(\xi)\mu_1(d\xi)=0.    
\end{align}

\begin{proposition}
Let $\mu,\nu$ be two centered Gaussian measures on X with Cameron-Martin spaces $H_\mu$ and $H_\nu$ respectively. Let $T\in\mathscr{L}(H_\mu,H_\nu)$ be a contraction. For every $n\in\N$ and every $h_1,\ldots,h_n\in X$, we have
\begin{align*}
\widetilde \Gamma_{\nu,\mu,n}(T)\bigg(I^\mu_{n}\big(\prod_{j=1}^nW^\mu_{h_j}\big)\bigg)
=I^\nu_{n}\Big(\prod_{j=1}^nW^\nu_{\widehat Th_j}\Big),
\end{align*}
where $\widehat T=Q_\nu^{-\frac{1}{2}}TQ_\mu^\frac12$.
\end{proposition}
\begin{proof}
Fix $n\in\N$ and $h_1,\ldots,h_n\in H_\mu$. Hence,
\begin{align*}
I^\mu_{n}\left(\prod_{j=1}^nW^\mu_{h_j}\right)
=\sum_{\alpha\in\Lambda_{n}}\ps{\prod_{j=1}^nW^\mu_{h_j}}{\Phi_\alpha^\mu}_{L^2(X,\mu)}\Phi_\alpha^\mu,
\end{align*}
where $\Phi_\alpha^\mu$ has been defined in \eqref{pol_herm_gen}. Notice that, for every $\alpha\in \Lambda_{n}$ and every $i_1,\ldots,i_n\in\N$, from \eqref{int_pot_xi} with $j=0$ we get
\begin{align*}
\int_X\prod_{j=1}^n W_{Q_\mu^{-\frac12}e_{i_j}^\mu}^\mu(x)\Phi_\alpha^\mu(x)dx
=
\prod_{k=1}^\infty\sqrt{\alpha_k!}\int_X\xi^{\alpha_k}\phi_{\alpha_k}(\xi)\mu_1(d\xi)=\sqrt{\alpha!}, 
\end{align*}
if $(i_1,\ldots,i_n)=(i^\alpha_{\sigma(1)},\ldots,i^\alpha_{\sigma(n)})$ for some $\sigma\in \Sigma^\alpha_n$, and, from \eqref{int_pt_xi_zero},
\begin{align*}
\int_X\prod_{j=1}^n W_{Q_\mu^{-\frac12}e_{i_j}^\mu}^\mu(x)\Phi_\alpha^\mu(x)dx=0  
\end{align*}
otherwise. It follows that
\begin{align*}
\langle \prod_{j=1}^nW^\mu_{h_j},\Phi_\alpha^\mu\rangle_{L^2(X,\mu)} 
= & \sum_{i_1,\ldots,i_n=1}^\infty\int_X\prod_{j=1}^n\langle h_j,Q_\mu^{-\frac12}e_{i_j}^\mu\rangle_X W^\mu_{Q_\mu^{-\frac12}e_{i_j}^\mu}(x) \Phi_\alpha^\mu(x)  \mu(dx) \\
= & \sum_{\sigma\in\Sigma^\alpha_n}\langle h_{\sigma(1)},Q_\mu^{-\frac12}e_{i_{\sigma(1)}^\alpha}^\mu\rangle_X\cdots \langle h_{\sigma(n)},Q_\mu^{-\frac12}e_{i_{\sigma(n)}^\alpha}^\mu\rangle_X\int_X\prod_{j=1}^n W_{Q_\mu^{-\frac12}e_{i_j}^\mu}^\mu(x)\Phi_\alpha^\mu(x)dx \\
= & \sqrt{\alpha!}\sum_{\sigma\in\Sigma^\alpha_n}\langle h_{\sigma(1)},Q_\mu^{-\frac12}e_{i_{\sigma(1)}^\alpha}^\mu\rangle_X\cdots \langle h_{\sigma(n)},Q_\mu^{-\frac12}e_{i_{\sigma(n)}^\alpha}^\mu\rangle_X.
\end{align*}
Recalling that $\#\Sigma_n^\alpha=\frac{n!}{\alpha!}$, from Lemma \ref{lemma:isomorfismo} we infer that
\begin{align}
\label{conto_monomi_1}
& (\Psi_n^\mu)^{-1}\bigg(I^\mu_{n}\Big(\prod_{j=1}^nW^\mu_{h_j}\Big)\bigg) \notag \\
= & \sum_{\alpha\in\Lambda_n}\frac{{\alpha!}}{n!}\sum_{\sigma\in\Sigma^\alpha_n}\langle h_{\sigma(1)},Q_\mu^{-\frac12}e_{i_{\sigma(1)}^\alpha}^\mu\rangle_X\cdots \langle h_{\sigma(n)},Q_\mu^{-\frac12}e_{i_{\sigma(n)}^\alpha}^\mu\rangle_X
\sum_{\vartheta\in\Sigma^\alpha_n}e_{i^\alpha_{\vartheta(1)}}^\mu\otimes \cdots \otimes e_{i^\alpha_{\vartheta(n)}}^\mu \notag \\
= & \sum_{\alpha\in\Lambda_n}
\sum_{\sigma\in\Sigma^\alpha_n}\langle h_{\sigma(1)},Q_\mu^{-\frac12}e_{i_{\sigma(1)}^\alpha}^\mu\rangle_Xe_{i^\alpha_{\sigma(1)}}^\mu\otimes \cdots \otimes  \langle h_{\sigma(n)},Q_\mu^{-\frac12}e_{i_{\sigma(n)}^\alpha}^\mu\rangle_X
e_{i^\alpha_{\sigma(n)}}^\mu \notag \\
= & \sum_{i_1,\ldots,i_n=1}^\infty \langle h_1,Q_\mu^{-\frac12}e_{i_1}^\mu\rangle_X e_{i_1}^\mu\otimes\cdots  \otimes \langle h_n,Q_\mu^{-\frac12}e_{i_n}^\mu\rangle_Xe_{i_n}^\mu.
\end{align}
This gives
\begin{align*}
& T^{\otimes n}(\Psi_n^{\mu})^{-1}\bigg(I^\mu_{n}\big(\prod_{j=1}^nW^\mu_{h_j}\big)\bigg) \notag \\
= &  \sum_{i_1,\ldots,i_n=1}^\infty \langle h_1,Q_\mu^{-\frac12}e_{i_1}^\mu\rangle_X Te_{i_1}^\mu\otimes\cdots  \otimes \langle h_n,Q_\mu^{-\frac12}e_{i_n}^\mu\rangle_X Te_{i_n}^\mu \notag \\
= & \sum_{i_1,\ldots,i_n,j_1,\ldots,j_n=1}^\infty \langle Q_\mu^{\frac12}h_{1},e_{i_1}^\mu\rangle_{H_\mu}\langle Te_{i_n}^\mu,e_{j_1}^\nu\rangle_{H_\nu} e_{j_1}^\nu\otimes \cdots \otimes \langle Q_\mu^\frac12h_{n},e_{i_n}^\mu\rangle_{H_\mu}\langle T e_{i_n}^\mu,e_{j_n}^\nu\rangle_{H_\nu}e_{j_n}^\nu \notag \\ 
= & \sum_{j_1,\ldots,j_n=1}^\infty\langle TQ_\mu^\frac12 h_1,e_{j_1}^\nu\rangle_{H_\nu}e_{j_1}^\nu
\otimes \cdots\otimes \langle TQ_\mu^\frac12 h_n,e_{j_n}^\nu\rangle_{H_\nu}e_{j_n}^\nu \notag \\
= & \sum_{j_1,\ldots,j_n=1}^\infty\langle \widehat T h_1,Q_\nu^{-\frac12}e_{j_1}^\nu\rangle_{X}e_{j_1}^\nu
\otimes \cdots\otimes \langle \widehat T h_n,Q_\nu^{-\frac12}e_{j_n}^\nu\rangle_{X}e_{j_n}^\nu,
\end{align*}
which, compared with \eqref{conto_monomi_1}, means that
\begin{align*}
T^{\otimes n}(\Psi_n^{\mu})^{-1}\bigg(I^\mu_{n}\big(\prod_{j=1}^nW^\mu_{h_j}\big)\bigg)
= (\Psi_n^\nu)^{-1}\bigg(I^\mu_{n}\big(\prod_{j=1}^nW^\nu_{\widehat Th_j}\big)\bigg)
\end{align*}
or, equivalently,
\begin{align*}
\widetilde \Gamma_{\nu,\mu,n}(T)\bigg(I^\mu_{n}\big(\prod_{j=1}^nW^\mu_{h_j}\big)\bigg)
=\Psi_n^\nu (T^{\otimes n}(\Psi_n^{\mu})^{-1})\bigg(I^\mu_{n}\big(\prod_{j=1}^nW^\mu_{h_j}\big)\bigg)
= 
I^\nu_{n}\Big(\prod_{j=1}^nW^\nu_{\widehat Th_j}\Big)
\end{align*}
for every $n\in\N$ and every $h_1,\ldots,h_n\in X$.
\end{proof}

\bibliographystyle{siam}
\bibliography{preprint.bib}
\end{document}